\documentclass[a4paper,twoside,10pt]{article}

\usepackage[a4paper,left=3cm,right=3cm, top=3cm, bottom=3cm]{geometry}
\usepackage[latin1]{inputenc}
\usepackage{mathrsfs}
\usepackage[titletoc,title]{appendix}
\usepackage{graphicx}
\usepackage{epstopdf}
\usepackage{amsmath}
\usepackage{amsthm}
\usepackage{amssymb}
\usepackage{stmaryrd}
\usepackage{subfigure}
\usepackage{float}
\usepackage{bigints}
\usepackage{cite}
\usepackage{color}
\usepackage[abs]{overpic}
\usepackage[font=footnotesize,labelfont=bf]{caption}
\usepackage{cases}
\usepackage{tikz}
\usepackage{algorithmic}
\usepackage{rotating}
\usepackage[author={Lorenzo}]{pdfcomment}
\usepackage{soul,xcolor}

\restylefloat{table}
\theoremstyle{plain}
\newtheorem{thm}{Theorem}[section]

\newtheorem{lem}[thm]{Lemma}

\theoremstyle{definition}

\newtheorem{exmp}{Example}[section]
\theoremstyle{remark}
\newtheorem{remark}{Remark}

\newcommand{\xbold}{\mathbf x} 
\renewcommand{\u}{u} 
\newcommand{\utilde}{\widetilde \u} 
\newcommand{\card}{\text{card}} 
\newcommand{\taun}{\mathcal T _n} 
\newcommand{\tautilden}{\widetilde {\mathcal T} _n} 
\newcommand{\vn}{v_{n}} 
\newcommand{\wn}{w_{n}} 
\newcommand{\un}{u_{n}} 
\newcommand{\Vn}{V_n} 
\newcommand{\Nu}{ \mathcal V }
\newcommand{\E}{K} 
\newcommand{\an}{ a _{n} } 
\newcommand{\anE}{ a _{n}^{\E} } 
\newcommand{\aE}{ a ^{\E} } 
\newcommand{\fn}{ f _{n} } 
\newcommand{\Pinabla}{\Pi ^{\nabla}} 
\newcommand{\bigfnE}{ \mathcal{F} _n^\E} 
\newcommand{\upi}{\u_{\pi}} 
\newcommand{\uI}{\u_{I}} 
\newcommand{\uE}{\pi \u} 
\newcommand{\uhatE}{\widehat {\pi \u}} 
\newcommand{\dof}{\text{dof}} 
\newcommand{\boldalpha}{\boldsymbol \alpha} 
\newcommand{\boldbeta}{\boldsymbol{\beta}} 
\newcommand{\el}{\ell} 
\newcommand{\diam}{\text{diam}} 
\newcommand{\Omegaext}{\Omega^{\text{ext}}} 
\newcommand{\uhat}{\widehat u} 
\newcommand{\Qhat}{\widehat Q} 
\newcommand{\Babuska}{Babu\v{s}ka } 
\newcommand{\dist}{\text{dist}} 
\newcommand{\Cn}{C_n} 
\newcommand{\Q}{Q} 
\renewcommand{\c}{c} 
\newcommand{\pE}{\p_{\E}}
\newcommand{\sE}{s_{\E}} 
\newcommand{\e}{E} 
\newcommand{\pe}{p_{\e}} 
\newcommand{\q}{q} 
\newcommand{\cinv}{c_{\text{inv}}} 
\newcommand{\p}{p} 
\newcommand{\dx}{dx} 
\newcommand{\dy}{dy} 
\newcommand{\Jalphan}{J_n^{\alpha}} 
\newcommand{\rhoalpha}{\rho_{\alpha}} 
\newcommand{\T}{T} 
\newcommand{\bT}{b_{\T}} 
\newcommand{\bThat}{b_{\That}} 
\newcommand{\That}{\widehat \T} 
\newcommand{\bE}{b_{\E}} 
\newcommand{\Ihat}{\widehat I} 
\newcommand{\SE}{S^\E} 
\newcommand{\Pizpmd}{\Pi^0_{\pE-2}} 
\newcommand{\Pinablap}{\Pi^{\nabla}_{\pE}} 
\newcommand{\VnE}{V(\E)} 
\newcommand{\ultraspherical}{shifted ultraspherical } 

\begin{tiny}
\author{
\normalsize{
L. Beir\~ao da Veiga
\footnote{Dip. di Matematica e Applicazioni,  Universit\`a degli Studi di Milano-Bicocca, E-mail: {\tt lourenco.beirao@unimib.it}}
\footnote{IMATI-CNR, Pavia}$\,\,\,$,
A. Chernov
\footnote{Inst. f\"ur Mathematik, C. von Ossietzky Universit\"at Oldenburg, E-mail: {\tt alexey.chernov@uni-oldenburg.de}}$\,\,$,
L. Mascotto
\footnote{Dip. di Matematica,  Universit\`a degli Studi di Milano, E-mail: {\tt lorenzo.mascotto@unimi.it}
and Inst. f\"ur Mathematik, C. von Ossietzky Universit\"at Oldenburg, E-mail: {\tt mascotto.lorenzo@uni-oldenburg.de}}$\,\,$,
A. Russo
\footnote{Dip. di Matematica e Applicazioni,  Universit\`a degli Studi di Milano-Bicocca, E-mail: {\tt alessandro.russo@unimib.it}}
\footnote{IMATI-CNR, Pavia}
}}
\end{tiny}

\date{}

\title{\textbf{\normalsize{Exponential convergence of the $hp$ Virtual Element Method with corner singularities}}}

\begin{document}
\maketitle

\begin{abstract}
In the present work, we analyze the $hp$ version of Virtual Element methods for the 2D Poisson problem.
We prove exponential convergence of the energy error employing sequences of polygonal meshes geometrically refined, thus extending the classical choices for the decomposition in the $hp$ Finite Element framework to very general decomposition of the domain.
A new stabilization for the discrete bilinear form with explicit bounds in $h$ and $\p$ is introduced.
Numerical experiments validate the theoretical results.
We also exhibit a numerical comparison between $hp$ Virtual Elements and $hp$ Finite Elements.
\end{abstract}

\section {Introduction} \label{section introduction}
The Virtual Element Method (VEM) is a very recent generalization of the Finite Element Method (FEM), see \cite{VEMvolley}.
VEM utilizes polygonal/polyhedral meshes \emph{in lieu of} the classical triangular/tetrahedral and quadrilateral/hexaedral meshes.
This automatically includes nonconvex elements, hanging nodes (enabling natural handling of interface problems with nonmatching grids), easy construction of adaptive meshes and efficient approximations of geometric data features.

Among the properties of VEM, in addition to the employment of polytopal meshes, we recall
the possibility of handling approximation spaces of arbitrary $\mathcal C^k$ global regularity \cite{Brezzi-Marini:2012,BeiraoManzini_VEMarbitraryregularity}
and approximation spaces that satisfy exactly the divergence-free constraint \cite{streamvirtualelementformulationstokesproblempolygonalmeshes, BLV_StokesVEMdivergencefree}.

The main idea of VEM consists in enriching the classical polynomial space with other functions, whose explicit knowledge is not needed for the construction of the method (this explains the name \emph{virtual}).

We point out that the literature concerning methods based on polytopal meshes is not restricted to the Virtual Element Method. A (very short and incomplete) list of other polytopal-based methods follows:
Hybrid High Order Methods \cite{dipietroErn_hho}, Mimetic Finite Difference \cite{BLM_MFD,BLS_MFD},
Hybrid Discontinuous Galerkin Methods \cite{cockburn_HDG}, Polygonal Finite Element Method \cite{SukumarTabarraeipolygonalintroduction, talischi2010polygonal, GilletteRandBajaj_generalizedbarycentric},
Polygonal Discontinuous Galerkin Methods \cite{cangianigeorgoulishouston_hpDGFEM_polygon}, BEM-based Finite Element Methods \cite{Weisser_basic}.

Although VEM is a very recent technology, the associated bibliography is widespread. Among the treated topic, we recall the following works concerning implementation issues \cite{hitchhikersguideVEM}, general linear second-order elliptic problems \cite{equivalentprojectorsforVEM, bbmr_VEM_generalsecondorderelliptic, cangianimanzinisutton_VEMconformingandnonconforming, BBMR_generalsecondorder},
Stokes problem \cite{streamvirtualelementformulationstokesproblempolygonalmeshes, BLV_StokesVEMdivergencefree,Gatica-1, CGM_nonconformingStokes},
Cahn-Hillard equation \cite{absv_VEM_cahnhilliard}, locking-free linear elasticity problem \cite{VEMelasticity,Paulino-VEM}, small deformation problems in structural mechanics \cite{BLM_VEMsmalldeformation},
plate bending problem \cite{Brezzi-Marini:2012}, Steklov eigenvalue problem \cite{VEMchileans}, residual based a-posteriori error estimation \cite{ManziniBeirao_VEMresidualaposteriori, cangianigeorgulispryersutton_VEMaposteriori, MoraRiveraRodriguez_aposSteklov},
serendipity Virtual Elements \cite{serendipityVEM}, application to discrete fracture network simulations \cite{Berrone-VEM, Benedetto-VEM-2, Benedetto-VEM-3},
contact problem \cite{Wriggers-contact}, comparison with the Smoothed Finite Element Method \cite{Bordas-VEM}, topology optimization \cite{Topology-VEM}, geomechanics problem \cite{Andersen-geo},
Helmoltz problem \cite{Helmholtz-VEM}.


In all these works, the convergence of VEM approximations has been achieved by increasing the number of mesh elements while keeping the degree of local approximation fixed. In other words and according to the existing terminology, these methods utilize the \emph{$h$-version of VEM}.
An alternative avenue to construct convergent approximations is the \emph{$\p$-version of VEM} which is based on increasing the degree of local approximations while keeping the underlying mesh fixed. A combination of both methodologies is termed the \emph{$hp$-version of VEM}.

The recent work \cite{hpVEMbasic} provides a mathematical ground for the $\p$-version of VEM for the two-dimensional Poisson problem. In particular, it includes the \emph{a priori} convergence theory for the $\p$- and $hp$-version of VEM on quasiuniform polygonal meshes and for uniform distributions of local degrees of approximation. An exponential convergence has been established for analytic solutions and convergence at algebraic rates for solutions having finite Sobolev regularity.

The objectives of the present paper are the following.
First, we generalize the results in \cite{hpVEMbasic} and in particular the definition of $H^1$ conforming Virtual Element to the case with varying local degree of accuracy from element to element.
Such construction fits very naturally in the framework of VEM without additional complications.
Furthermore, we extend the $hp$-VEM approach to nonquasiuniform approximations and prove their exponential convergence for nonsmooth solutions having typical corner singularities, see \cite{babuvska1988regularity, babuvska1989regularity}.

Similarly to the $hp$-version of FEM (see \cite{SchwabpandhpFEM} and the references therein) the approximation is based on geometrically refined meshes with appropriate (linearly varying) local degree of accuracy.
In order to derive the proofs, we introduce a new stabilization for the method, which turns out to be more suitable for the $\p$ and $hp$ version of VEM;
in particular, explicit bounds of the new stabilizing term with respect to the $H^1$ seminorm in terms of the local degree of accuracy are shown.
This proof requires a particular inverse estimate on polynomials, presented in the first appendix.
To the best of the authors knowledge, such an inverse estimate has been never published before.

As a byproduct of this work, new tools for the approximation by means of functions in the Virtual Element Space are presented; such tools permits to avoid additional assumptions on the polygonal decomposition of the computational domain.

The structure of the paper is the following.
After presenting the model problem and the Virtual Element Method in Section \ref{section model problem and functional spaces} and \ref{section VEM with non uniform degree of accuracy} respectively,
we deal with explicit bounds in terms of the degree of accuracy of a new stabilization term in Section \ref{section stability}.
In Section \ref{section exponential convergence for corner singularity on geometric meshes} we show the approximation results and the main theorem of the paper, namely the exponential convergence of the energy error in terms of the dimension of the virtual space,
while in Section \ref{section numerical results} we validate the theoretical results with numerical tests, including a set of experiments on the comparison between $hp$ FEM and $hp$ VEM.
Finally, in the two Appendices \ref{section appendix A} and \ref{section appendix B}, we discuss two particular polynomial inverse estimates needed for the stability bounds of Section \ref{section stability}.

\section {Model problem} \label{section model problem and functional spaces}
In this section, we discuss the functional space setting and the model problem under consideration.

Firstly, we introduce the functional spaces that will be used throughout the paper.
Let $\el \in \mathbb N$ and let $D \subseteq \mathbb R^2$ be a given domain whose closure contains the origin, i.e. $\mathbf 0 \in D$.
We denote with $L^\el(D)$ the Lebesgue space of $\el$-summable functions and we denote with $H^\el(\omega)$ the Sobolev space of order $\el$ on the domain $D$, respectively;
let $\| \cdot \|_{\el,D}$ and $| \cdot |_{\el,D}$ be the Sobolev norm and seminorm, see \cite{adamsfournier}.
Let $H^1_0(D)=\{u\in H^1(D) : \u|_{\partial D}=0\}$.

Let now $\beta \in (0,1)$, $\Phi _{\beta} (\mathbf x)=|\mathbf x|^{\beta}$, where $|\cdot|$ represents the Euclidean norm in $\mathbb R ^2$.
Given $u: \omega \rightarrow \mathbb R$, $m,\,l \in \mathbb N$, $m\ge \el$, we set
\begin{equation}	\label{weighted seminorm}
|u|^2_{H^{m,\el}_{\beta}(D)} : = \sum_{k=\el}^{m} \| |D^{k} u| \Phi _{\beta+k-\el} \|^2_{0,D},
\end{equation}
where:
\[
\vert D^k \u \vert^2 = \sum_{\boldalpha \in \mathbb N^2,\, \vert \boldalpha \vert = n} \vert D^{\boldalpha} \u \vert^2.
\]
We define the weighted Sobolev spaces
\begin{equation} \label{weighted Sobolev space}
H^{m,\el}_{\beta}(D) := \left \{ u \in L^2(D)  \text{ }\bigg | \text{ } \|u\|_{H^{m,\el}_{\beta}(D)} < \infty   \right\},
\end{equation}
where the corresponding weighted Sobolev norm reads
\begin{equation} \label{weighted Sobolev norm}
\|u\|^2_{H^{m,\el}_{\beta}(D)} : =
\begin{cases}
\| u  \|^2_{\el-1,D} + |u|^2_{H^{m,\el}_{\beta}(D)} 		& \text{if }\el\ge 1,\\
\sum_{k=0}^{m} \| |D^{k} u| \Phi _{\beta+k} \|^2_{0,D}		& \text{if }\el= 0.\\
\end{cases}
\end{equation}
Having this, we recall the countably normed spaces (also known as \Babuska spaces), see \cite{SchwabpandhpFEM} and the references therein:
\begin{equation} \label{countably normed space}
\mathcal B _{\beta}^\el (D):= \left \{  u \mid u\in H^{m,\el}_{\beta}(D),\, \forall m\ge \el\;\; \text{and} \;\; \Vert \vert D^k u\vert \Phi_{\beta + k - \el} \Vert_{0,D}\le c\cdot d ^{k-\el} \cdot (k-\el)!,\, \forall \,k=\el, \el+1,\dots           \right \},
\end{equation}
where $c\ge 0$ and $d\ge 1$ are two constants depending on $u$ and $D$.
We point out that space \eqref{countably normed space} is not empty since it contains functions of the form $u=|\mathbf x|^{\alpha}$, for some $\alpha>0$.

Definition \eqref{countably normed space} can be generalized to the case of functions with ``multiple'' singularities at the vertices of a polygonal domain
i.e. adding in definition \eqref{weighted seminorm} weights of the form $|\mathbf x -\mathbf x_{0}|$, $\mathbf{x_0}$ being a vertex of the polygon different from $\mathbf 0$; see \cite{SchwabpandhpFEM}.
In particular, defining $N_V(D)$ and $\{\mathbf {A}_i(D)\}_{i=1}^{N_V(D)}$ the number and set of vertices of $D$ respectively, we will denote the general space with $H^{m,\boldsymbol \el}_{\boldbeta}(D)$, where
$\boldbeta=(\beta_1,\dots,\beta_{N_V(D)})$  and $\boldsymbol \el =(\el_1(\beta_1),\dots, \el_{N_V(\omega)}(\beta_{N_V(D)}))$ are the vectors associated with the singularities at the vertices of $\omega$.
The associated weight function reads:
\[
\Phi_{\boldbeta} (\mathbf x) = \Pi_{i=1}^{N_V(D)} r_i(\mathbf x)^{\beta_i},\qquad r_i(\mathbf x)=\min (1,\vert \mathbf x -\mathbf {A}_i(D) \vert).
\]

Secondly, we introduce the model problem.
Let $\Omega$ be a open simply connected polygonal domain. Let $f\in L^2(\Omega)$ be given. We consider the two dimensional Poisson problem:
\begin{equation} \label{continuous strong problem}
-\Delta u = f \text{ in } \Omega,\quad u=0 \text{ on } \partial \Omega
\end{equation}
and its weak formulation:
\begin{equation} \label{continuous weak problem}
\text{find } u\in V:=H^1_0(\Omega)\qquad \text{ such that } \qquad a(u,v)=(f,v)_{0,\Omega},\quad \forall v \in V,
\end{equation}
where $(\cdot, \cdot)_{0,\Omega}$ is the $L^2$ scalar product on $\Omega$ and $a(\cdot, \cdot) = (\nabla \cdot, \nabla \cdot)_{0,\Omega}$.
The Lax-Milgram lemma guarantees the existence of a unique weak solution $\u \in V$.

We recall a regularity result for the solution of problem \eqref{continuous weak problem}.
Let $N_V$ and $\{\mathbf {A}_i\}_{i=1}^{N_V}$ be the number and the set of vertices of $\Omega$ respectively;
let $\omega_i$ be the (internal) angle associated with vertex $\mathbf {A}_i$, $i=1,\dots, N_V$.
To each $\mathbf {A}_i$, we associate the set of the so-called \emph{singular exponents} for Poisson problem with Dirichlet condition (see also \cite[formula (4.2.2)]{SchwabpandhpFEM}):
\begin{equation} \label{singular exponents for Dirichlet conditions}
\alpha_{i}= \frac{\pi}{\omega_i},\quad \forall i=1,\dots, N_V.
\end{equation}
Then, the following holds:
\begin{thm} \label{theorem Babuska regularity}
Following the notation of problem \eqref{continuous weak problem}, let $f\in H^{0,0}_{\boldbeta}(\Omega)$, $\beta \in [0,1)$;
assume that the singular exponents $\alpha_{i}$ defined in \eqref{singular exponents for Dirichlet conditions}
satisfy:
\[
1-\alpha_{i} <\beta_i<1,\quad \text{if } \alpha_{i} <1,\quad \forall i=1,\dots, N_V.
\]
Then the solution of \eqref{continuous weak problem}belongs to $H^{2,2}_{\boldbeta}(\Omega)$ and the a propri estimate:
\[
\Vert u \Vert _{H^{2,2}_{\boldbeta}(\Omega)} \le c \Vert f \Vert _{H^{0,0}_{\boldbeta}(\Omega)},
\]
holds. Moreover, if $f \in \mathcal B^0_{\boldbeta}(\Omega)$, then $\u \in \mathcal B^2_{\boldbeta} (\Omega)$.
\end{thm}
\begin{proof}
See \cite{babuvska1988regularity, babuvska1989regularity}.
\end{proof}
For the sake of simplicity, we will assume in the rest of the paper that 
\begin{equation} \label{highlight one singularity}
\text{$\mathbf 0 \in \partial \Omega$ is the only vertex at which the solution $u$ of problem \eqref{continuous weak problem} can be singular.}
\end{equation}
Finally, we point out that throughout the paper we will write $a \approx b$ and $a \lesssim b$ meaning that
there exist $c_1$, $c_2$ and $c_3$ positive constants independent of the discretization parameters, such that  $c_1 \, a \le b \le c_2 \, a$ and $ a \le c_3 \, b$ respectively;
we will also denote by $\mathbb P _\el(\E)$ and $\mathbb P _\el (\e)$ the spaces of polynomials of degree $\el$ over a polygon $\E$ and an edge $\e$ respectively.

\section {Virtual Element Spaces with non uniformly distributed degree of accuracy} \label{section VEM with non uniform degree of accuracy}
In this section, we introduce a Virtual Element Method for problem \eqref{continuous weak problem} with nonuniform local degree of accuracy.

Let $\{\taun\}$ be a sequence of polygonal decompositions of the domain $\Omega$.
The approximation will have a ``geometric layer'' structure;
hence, in the sequel, the integer $n$ will represent the number of layers used for the corner singularity refinement as in \cite{SchwabpandhpFEM}; see Section \ref{subsection geometric meshes} for the precise definition of layers.

Let $\Nu_n$ be the set of vertices  of $\taun$, $\Nu_n^{b} = \{\nu \in \Nu_n \mid \nu \in \partial \Omega\}$ be the subset of boundary vertices,
$\mathcal E _n$ be the set of edges $\e$ of $\taun$, $\mathcal E _n ^b = \{e\in \mathcal E_n \mid e\subseteq \partial \Omega\}$ be the subset of boundary edges.
To each $\E \in \taun$, we associate $h_{\E}=\diam (\E)$, $\Nu^{\E}=\{\nu \in \Nu_n \mid \nu \in \partial \E\}$ and $\mathcal E ^{\E}=\{e\in \mathcal E_n \mid e\subseteq \partial \E\}$.
We require the two following basic assumptions on the regularity of the decomposition:
\begin{itemize}
\item[$(\mathbf{D1})$] $\forall \E \in \taun$, $\E$ is star-shaped with respect to a ball of radius greater than or equal to $h_{\E} \, \gamma$, where $\gamma$ is a positive constant independent of the decompositions;
see \cite{BrennerScott} for the definition of star-shapedness.
We note that this condition can be satisfied by possibly many balls. Henceforth, we fix for each $\E\in \taun$ a unique ball $B(\E)$.
\item[$(\mathbf{D2})$] $\forall \E \in \taun$, $\forall \e \in \mathcal E ^{\E}$, $|\e|\ge h_E \, \widetilde\gamma$,
where $\widetilde \gamma$ is a positive constant independent of the decompositions. Moreover, $\forall \E \in \taun$, $\card (\mathcal E^\E)$ is uniformly bounded.
\end{itemize}
More technical assumptions on the mesh will be introduced in Section \ref{section exponential convergence for corner singularity on geometric meshes} for the construction of proper geometric meshes.
\begin{remark} \label{remark regular subtriangulation}
Assuming that (\textbf{D1}) and (\textbf{D2}) hold true, then the following is also valid.
The subtriangulation $\tautilden(\E)$ of $\E$ obtained by joining the vertices of $\E$ to the center of the ball $B(\E)$ introduced in assumption (\textbf{D1})
is made of triangles that are star-shaped with respect to a ball of radius greater than or equal to $\gamma_1 \, h_\T$,
$h_\T$ being the diameter of $\T$, $\forall \T \in \tautilden(\E)$, and $\gamma_1$ being a positive constant independent of the decompositions.
\end{remark}

Given $\E \in \taun$, let $i_{\E}$ be the position of polygon $\E$ in the ordered sequence $\taun$. Let $\mathbf p \in \mathbb N^{\text{card}(\taun)}$.
We associate to each $\E \in \taun$ the local degree of accuracy $\p_{i_{\E}}=(\mathbf \p)_{i_{\E}}$. In order to simplify the notation, we write $\pE := p_{i_{\E}}$.

Henceforth, we assume that $\taun$ is a \emph{conforming} decomposition into polygons of $\Omega$, i.e., for all edges $\e \in \mathcal E$,
either $\e$ belongs to two polygons if it is an \emph{internal} edge or it belongs to a single polygon if it is a \emph{boundary} edge.

In the former case, it must hold that there exist $\E_1$, $\E_2\in \taun$ such that $\e\in \mathcal E^{\E_1} \cap \mathcal E^{\E_2}$;
we associate to $\e$ the degree $\pe=\max \{\p_{\E_1}, \p_{\E_2}\}$, that is we adopt the so-called \emph{maximum rule}; see Remark \ref{remark degree on the edges} for further comments.
In the latter case, let $\E\in\taun$ be the unique polygon in the decomposition such that $\e \in \mathcal E^{\E}$; we associate to $\e$ the degree $\pe=\pE$.

Let $\E \in \taun$. We firstly define the space of piecewise continuous polynomials on the boundary of $\E$:
\begin{equation} \label{definition local non uiniform boundary}
\mathbb B (\partial \E) := \{\vn \in \mathcal C ^0 (\partial \E) \mid \vn|_\e \in \mathbb P _{\pe}(\e),\; \forall \e \in \mathcal E ^{\E}\}.
\end{equation}
The local virtual space on $\E$ reads:
\begin{equation} \label{local non uniform virtual space}
\VnE:=\{\vn \in H^1(\E) \mid \Delta \vn \in \mathbb P_{\pE -2}(\E) \text{ and } \vn \in \mathbb B (\partial \E)\},
\end{equation}
with the convention $\mathbb P_{-1} (\E) = 0$ and where $\mathbb B(\partial \E)$ is defined in \eqref{definition local non uiniform boundary}.

Definition \eqref{local non uniform virtual space} and the maximum rule immediately imply that $\mathbb P _{\pE} (\E)\subseteq \VnE$.

We associate with the local space the following set of degrees of freedom:
\begin{itemize}
\item the values at the vertices of $\E$;
\item the values at $\pe-1$ internal nodes (e.g. Gau\ss-Lobatto nodes) for all $\e\in \mathcal E^{\E}$;
\item the scaled internal moments of the form $\frac{1}{\vert \E \vert} \int _{\E} q_{\boldalpha} \vn$, where $\{q_{\boldalpha}\}_{|\boldalpha|=0}^{\pE- 2}$ is a properly chosen basis of $\mathbb P_{\pE-2}(\E)$;
			see \cite{hitchhikersguideVEM} for a possible explicit choice of the basis.
\end{itemize}
This is in fact a set of degrees of freedom for the local space \eqref{local non uniform virtual space}; see \cite{VEMvolley}.
If we set $\dof_i$ the $i$-th degree of freedom, $i=1,\dots,\dim(\VnE)$, then we can define the local virtual canonical basis $\{\varphi_j,\, j=1,\dots, \dim(\VnE)\}$ by:
\begin{equation} \label{definition canonical basis}
\dof  _i(\varphi_j) = \delta_{i,j},\quad \forall i,\,j=1,\dots, \dim(\VnE).
\end{equation}\\
The global virtual space is obtained by matching the boundary degrees of freedom on each edge, i.e.:
\begin{equation} \label{global non uniform virtual space}
\Vn := \{\vn \in \mathcal C ^0 (\overline \Omega) \;\mid\; \vn|_{\E}\in \VnE,\; \forall \E \in \taun;\; \vn|_{\partial \Omega} =0\}.
\end{equation}

We note that we can split the global continuous bilinear form $a(\cdot, \cdot)$, introduced with the continuous problem \eqref{continuous weak problem}, into a sum of local contributions as follows:
\begin{equation} \label{splitting the continuous bilinear form}
a(\u,v) = \sum _{\E \in \taun} a^\E(\u,v),\quad \text{ where } \quad a^\E (\u, v) = (\nabla \u, \nabla v)_{0,\E}.
\end{equation}
We observe that we cannot compute the bilinear form $a(\cdot, \cdot)$ applied on virtual functions since it is not possible in principle to know the values of such functions at any internal points of each polygon.
The same argument applies to the computation of the right-hand side.
For this reason, we must approximate both the stiffness matrix and the right-hand side.

Thus, the structure of VEM approximation is based on the two following ingredients which will be defined in what follows:
\begin{itemize}
\item a symmetric bilinear form $\an: \Vn \times \Vn \rightarrow \mathbb R$, which we decompose into a sum of local symmetric bilinear forms $\an^{\E}: \VnE \times \VnE \rightarrow \mathbb R$ as follows:
\begin{equation} \label{global bilinear form}
\an (\vn, w_n) = \sum _{\E \in \taun} \an ^{\E} (\vn, w_n),\; \forall \vn, w_n \in \Vn;
\end{equation}
\item a piecewise discontinuous polynomial $\fn$,which is piecewise of degree $\pE$, and the associated linear functional $( \fn,\cdot )_{0,\Omega}$.
\end{itemize}
The discrete bilinear form $a_n(\cdot, \cdot)$ and the discrete right-hand side $\fn$ are chosen in such a way that the discrete counterpart of \eqref{continuous weak problem}
\begin{equation} \label{discrete weak problem}
\text{find } \un\in \Vn\text{ such that } \an(\un,\vn)=( \fn,\vn )_{0,\E},\, \forall \vn \in \Vn
\end{equation}
is well-posed and it is possible to recover \emph{local} $hp$-estimates analogous to those proved in \cite{hpVEMbasic}.

We start by discussing the construction of the discrete bilinear form.
We require that $\anE$ in \eqref{global bilinear form} satisfy the two following assumptions:
\begin{itemize}
\item [($\mathbf {A1}$)] \textbf{polynomial consistency}: $\forall \E \in \taun$, it must hold:
\begin{equation} \label{pE consistency}
a^{\E} (q, \vn) = \an ^{\E} (q,\vn),\quad \forall q\in \mathbb P_{\pE}(\E),\, \forall \vn \in \VnE;
\end{equation}
\item [($\mathbf {A2}$)]\textbf{local stability}: $\forall \E \in \taun$, it must hold
\begin{equation} \label{stability an}
\alpha_* (\E) \vert \vn \vert^2_{1,\E} \le \an^{\E} (\vn,\vn) \le \alpha^* (\E) \vert \vn\vert_{1,\E}^2,\quad \forall \vn \in \VnE,
\end{equation}
where $0<\alpha_*(\E)\le \alpha^*(\E)< +\infty$ are two constants, which may depend only on the local space $\VnE$.
\end{itemize}
On each $\E \in \taun$, we can introduce a local energy projector $\Pi^{\nabla}_{\pE,\E}: \VnE \rightarrow \mathbb P_{\pE}(\E)$ via
\begin{subnumcases}{}
a^{\E}(q, \Pi^{\nabla}_{\pE,\E} \vn - \vn)=0 & $\forall q\in \mathbb P_{\pE}(\E),\; \forall \vn \in \VnE$, \label{pi nabla 1}\\
\int_{\partial \E} (\Pi^{\nabla}_{\pE,\E} \vn - \vn )ds=0 & $\forall \vn \in \VnE.$ \label{pi nabla 2} 
\end{subnumcases}
When no confusion occurs, we will write $\Pinablap$ in lieu of $\Pi^{\nabla}_{\pE,\E}$.

Note that condition \eqref{pi nabla 2} only fixes an additive constant of the projection and can be modified if necessary, \cite{equivalentprojectorsforVEM, VEMvolley}.
Importantly, this local energy projector can be computed by means of the degrees of freedom of space \eqref{local non uniform virtual space}, see \cite{hitchhikersguideVEM,VEMvolley}, without the need of knowing explicitely functions in the virtual space.

In \cite{VEMvolley,hitchhikersguideVEM}, it was also shown that a \emph{computable} candidate for $\an^{\E_i}$ may have the following form:
\begin{equation} \label{choice of the discrete local bilinear form}
\an ^{\E}(\vn, \wn) = a^{\E} (\Pinablap \vn, \Pinablap \wn)  +  S^{\E} (\vn -\Pinablap \vn, \wn -\Pinablap \wn),\quad \forall \vn,\, \wn \in \VnE,
\end{equation}
where $S^{\E}$ is any \emph{computable} symmetric bilinear form on $\VnE$ such that
\begin{equation} \label{bilinear form S}
c_{*}(\E) \vert \vn \vert_{1,\E}^2  \lesssim S^{\E} (\vn, \vn) \lesssim c^{*}(\E) \vert \vn \vert^2_{1,\E},\quad \forall\, \vn \in \VnE \,\text{ with } \Pinablap \vn=0,
\end{equation}
where $0<c_*(\E)\le c^*(\E)< +\infty$ are two constants, depending possibly on the local space $\ker( \Pinablap )$.
In \cite{VEMvolley} it was shown that \eqref{bilinear form S} implies \eqref{stability an} with:
\begin{equation} \label{relation alpha c}
\alpha_*(\E)=\min(1,c_*(\E)),\;\alpha^*(\E)=\max(1,c^*(\E)),\quad \forall \E \in \taun.
\end{equation}

Now, we introduce a \emph{computable} discrete loading term $\fn$.
Let $S^{\mathbf p, -1}(\Omega, \taun)$ be the set of piecewise discontinuous polynomials over the decomposition $\taun$ of degree $\pE$ on each $\E \in \taun$.
For $\el \in \mathbb N$, let $\Pi^{0}_{\el} := \Pi^{0}_{\el, \E}$ be the $L^2(\E)$ projector from the local space \eqref{local non uniform virtual space} to $\mathbb P_{\el} (\E)$, the space of polynomials of degree $\el$ over $\E$;
such a projector can easily be computed whenever $\el \le \pE -2$ by means of the internal degrees of freedom of the space \eqref{local non uniform virtual space}, see \cite{hitchhikersguideVEM}.

We define the discrete loading term as follows: $\fn \in S^{\mathbf p, -1}(\Omega, \taun)$ is such that
\begin{equation} \label{discrete loading term}
(\fn, \vn)_{0,\Omega} = \sum_{\E \in \taun} ( \fn, \vn )_{0,\E},\quad \text{where} \quad ( \fn, \vn )_{0,\E}:=  \int_\E \Pizpmd f \vn,\quad \forall \vn \in \Vn.
\end{equation}
A deeper analysis on the discrete loading term can be found in \cite{equivalentprojectorsforVEM} and \cite{VEMelasticity}.

We remark that in this paper we will not consider the case of approximation with $\pE=1$ in order to avoid technical discussions on the right-hand side.

\begin{remark} \label{remark degree on the edges}
We point out that in the definition of the local Virtual Space \eqref{local non uniform virtual space}, we fixed the degree of the edge to be the maximum of the degree of the two adjacent polygons (\emph{maximum-rule}).
One could also fix such an edge degree to be the minimum of the degree of the neighbouring polygons (\emph{minimum-rule}).
The first choice leads to $\mathbb P_{\pE}(\E) \subseteq \VnE$;
therefore, it is possible to recover local (i.e. on each polygon) classical $hp$-estimates, see \cite{hpVEMbasic}.
On the other hand, in view of Section \ref{section exponential convergence for corner singularity on geometric meshes} also the choice of the minimum would yield the same convergence result.
\end{remark}
Let $\bigfnE$, $\E\in \taun$, be the smallest positive constants such that:
\begin{equation} \label{big F n}
\begin{split}
&|(\fn,\vn)_{0,\E}          - (f,\vn)_{0,\E}| \le \bigfnE |\vn|_{1,\E},          \; \forall \vn \in \VnE
\end{split}
\end{equation}
and let
\begin{equation} \label{local parameter stability}
\alpha (\E) := \frac{1+\alpha^*(\E)}{\min_{\E ' \in \taun}\alpha_*(\E')}  ,\qquad \forall \E \in \taun,
\end{equation}
where $\alpha_*(\E)$ and $\alpha^*(\E)$ are introduced in \eqref{stability an}.

We show how the energy error $\vert \u -\un \vert_{1,\E}$ can be bounded, $\u$ and $\un$ being respectively the solutions of \eqref{continuous weak problem} and \eqref{discrete weak problem}.
We carry out, in particular, an abstract error analysis which is similar to the one presented in \cite{VEMvolley}; nevertheless, we decide to show the details, since assumption (\textbf{A2}) is here weaker than its $h$ counterpart
in \cite{VEMvolley}, where the stability constants $\alpha_*(\E)$ and $\alpha^*(\E)$ are assumed to be independent of the local spaces.
\begin{lem} \label{lemma local VEM decomposition}
Assume that (\textbf{A1}) and (\textbf{A2}) hold.
Let $\u$ and $\un$ be the solutions of problems \eqref{continuous weak problem} and \eqref{discrete weak problem} respectively.
Then, for all $\uI \in \Vn$ and for all $\upi \in S^{\mathbf p, -1}(\Omega, \taun)$, it holds that
\begin{equation}\label{global VEM decomposition}
\begin{split}
\vert \u -\un \vert _{1,\Omega} 	\le \sum_{\E \in \taun} \alpha(\E) \left\{ \bigfnE + \vert \u - \upi \vert_{1,\E} + \vert \u -\uI \vert_{1,\E} \right\},\\
\end{split}
\end{equation}
where $\bigfnE$ and $\alpha(\E)$ are defined in \eqref{big F n} and \eqref{local parameter stability} respectively.
\end{lem}
\begin{proof}
Given any $\upi \in S^{\mathbf p, -1}(\Omega, \taun)$ and $\uI \in \Vn$:
\[
\begin{split}
&\vert \un -\uI \vert^2_{1,\Omega}	= \sum_{\E\in \taun} \vert \un -\uI \vert^2_{1,\E} \overbrace{\le}^{(\mathbf{A2})} \sum_{\E \in \taun} \alpha_*^{-1}(\E) \anE(\un-\uI,\un-\uI)\\
& \overbrace{\le}^{(\textbf{A1}),\, \eqref{continuous weak problem},\, \eqref{discrete weak problem}}  \left( \max_{\E'\in \taun} \alpha_* ^{-1}(\E') \right) \sum_{\E \in \taun} \left\{ ( \fn-f, \un-\uI )_{0,\E}  - \anE(\uI-\upi, \un-\uI) - \aE(\upi-u,\un-\uI)   \right\} \\
& \overbrace{\le}^{(\textbf{A1}), \, \eqref{big F n}} 
\left( \max_{\E'\in \taun} \alpha_* ^{-1}(\E') \right) \sum_{\E \in \taun} \left\{ \bigfnE \vert \uI-\un \vert_{1,\E} +\alpha^*(\E)\vert \uI - \upi \vert _{1,\E} \vert \un -\uI \vert_{1,\E} + \vert \u -\upi\vert _{1,\E} \vert\un-\uI\vert_{1,\E}   \right\} \\
& \le  \left( \max_{\E'\in \taun} \alpha_* ^{-1}(\E') \right) \left( \sum_{\E\in \taun} \left( \bigfnE + (1+\alpha^*(\E)) \vert \u-\upi \vert_{1,\E} + \alpha^*(\E) \vert \u-\uI\vert_{1,\E}  \right)^2   \right) ^{\frac{1}{2}} \vert \un -\uI\vert_{1,\Omega}.
\end{split}
\]
where the Cauchy-Schwarz inequality has been used in the last step.
Applying a triangular inequality, we get:
\[
\vert \u - \un \vert_{1,\Omega} \le \sum_{\E \in \taun} \frac{1 + \alpha^*(\E)}{\min_{\E' \in \taun} \alpha_*(\E')} \left\{ \bigfnE + \vert \u - \upi \vert_{1,\E} + \vert \u - \uI \vert_{1,\E}   \right\}.
\]
This finishes the proof.
\end{proof}

\section{Stability } \label{section stability}
In this section, we present an explicit choice for the stabilizing bilinear form $\SE$ introduced in \eqref{bilinear form S} and we discuss the associated stability bounds \eqref{bilinear form S} in terms of the local degree of accuracy.
Our choice for the stabilization is the following:
\begin{equation} \label{stability choice}
\SE(\un,\vn) =
\frac{\pE}{h_\E} (\un,\vn)_{0,\partial \E} + \frac{\pE^2} {h_\E^2} (\Pizpmd \un, \Pizpmd \vn)_{0,\E}.
\end{equation}
We note that this local stabilization term is explicitely computable by means of the local degrees of freedom, since on the boundary virtual functions are known polynomials and the $L^2$  projections are computable using only the internal degrees of freedom,
see \cite{hitchhikersguideVEM}.

Following the guidelines of \cite[formula (4.5.61)]{SchwabpandhpFEM}, that is the $\p$-version of the Aubin-Nitsche duality argument, it holds for a convex $\E$:
\begin{equation} \label{Aubin Nitsche convex case}
\Vert \vn - \Pinablap \vn \Vert_{0,\E} \lesssim \frac{h_\E}{\pE} \vert \vn -\Pinablap \vn \vert_{1,\E},\quad \forall \vn \in \VnE.
\end{equation}
Note that in order to apply the Aubin-Nitsche argument we use the fact that $\vn\in \ker(\Pinablap)$, which guarantees that $\vn - \Pinabla \vn$ has zero average on the boundary.

Assume now that $\E$ is nonconvex. Let $\pi < \omega _\E< 2\pi$ the largest angle of $\E$. Then, the Aubin-Nitsche analysis in addition to interpolation theory, see \cite{triebel, tartar}, can be refined giving:
\begin{equation} \label{modified Aubin Nitsche}
\Vert \vn - \Pinabla \vn \Vert_{0,\E} \lesssim \left( \frac{h_\E}{\pE} \right)^{\frac{\pi}{\omega_\E}} \vert \vn -\Pinablap \vert _{1,\E},\quad \forall \vn \in \VnE.
\end{equation}
We now prove the following result.
\begin{thm} \label{theorem stability bounds}
Assume that $\pE$, the degree of accuracy of the method on the element $\E$, coincides with the polynomial degrees $\pe$, for all edges $\e \in \mathcal E^\E$ of polygon $\E$.
Then, using definition \eqref{stability choice}, the bounds in \eqref{bilinear form S} hold with:
\begin{equation} \label{stability bound constants}
\c_*(\E)\ge
\pE^{-5},\quad\quad
\c^*(\E)\le\begin{cases}
1 & \text{if } \E \text{ is convex},\\
\pE^{2 \left( 1- \frac{\pi}{\omega_\E}  \right)} & \text{otherwise},
\end{cases}
\end{equation}
where $\omega_\E$ denotes the largest angle of $\E$.
\end{thm}
\begin{proof}
We assume without loss of generality that the size of polygon $\E$ is 1. The general result follows from a scaling argument.

We start by proving the estimate on $\c_*(\E)$. Integrating by parts, we obtain for $\vn \in \ker (\Pinablap)$:
\begin{equation} \label{integration by parts stability}
\vert \vn \vert_{1,\E}^2 = \int_{\E} \nabla \vn \cdot \nabla \vn = \int_{\E} -\Delta \vn \Pizpmd \vn + \int_{\partial \E} \frac{\partial \vn}{\partial n} \vn.
\end{equation}

We split our analysis into two parts. We firstly investigate the integral over $\E$ in \eqref{integration by parts stability}.
For this purpose, we need a technical result, namely the following $hp$ polynomial inverse estimate in two dimensions,
see Corollary \ref{corollary inverse polynomial estimate with negative norm} (which can be applied thanks to Remark \ref{remark regular subtriangulation}):
\begin{equation} \label{inverse estimate polynomials negative norm}
\Vert \q \Vert_{0,\E} \lesssim (\pE-1)^2 \Vert \q \Vert_{-1,\E} \le \pE^2 \Vert \q \Vert_{-1,\E},\quad \forall \q \in \mathbb P_{\pE-2}(\E),
\end{equation}
where we denote with $\Vert \cdot \Vert_{-1,\E}$ the dual norm associated with $H^1_0(\E)$, i.e.
\[
\Vert \cdot \Vert _{-1,\E} = \Vert \cdot \Vert _{[H^1_0(\E)]^*} = \sup_{\Phi\in H^1_0(\E)\setminus\{0 \}} \frac{(\Phi, \cdot)_{0,\E}}{\vert \Phi \vert_{1,\E}}.
\]
Subsequently, we note that, owing to \eqref{inverse estimate polynomials negative norm}, we have:
\begin{equation} \label{estimate on virtual laplacian}
\begin{split}
&\Vert \Delta \vn \Vert _{0,\E} \lesssim \pE ^2 \Vert \Delta \vn \Vert_{-1,\E} = \pE^2 \sup_{\Phi\in H^1_0(\E)\setminus\{0 \}} \frac{(\Delta \vn, \Phi)_{0,\E}}{\vert \Phi \vert_{1,\E}}=
\pE^2 \sup_{\Phi\in H^1_0(\E)\setminus\{0 \}} \frac{(\nabla \Phi, \nabla \vn)_{0,\E}}{\vert \Phi \vert_{1,\E}}
\le \pE^2 \vert \vn \vert_{1,\E}.
\end{split}
\end{equation}
As a consequence:
\begin{equation} \label{stability inequality A}
\int_{\E} -\Delta \vn \Pizpmd \vn \le \Vert \Delta \vn \Vert_{0,\E} \cdot \Vert \Pizpmd \vn \Vert_{0,\E} \le \pE^2 \Vert \Pizpmd \vn \Vert_{0,\E} \vert \vn \vert_{1,\E}.
\end{equation}
Next, we turn our attention to the integral over $\partial \E$ in \eqref{integration by parts stability}. Applying a Neumann trace inequality (see e.g. \cite[Theorem A33]{SchwabpandhpFEM}):
\begin{equation} \label{stability inequality B}
\int_{\partial \E} \frac{\partial \vn}{\partial n} \vn \le \left \Vert \frac{\partial \vn}{\partial n} \right \Vert _{-\frac{1}{2},\partial \E} \Vert \vn \Vert_{\frac{1}{2},\partial \E}
\lesssim (\vert \vn \vert_{1,\E} + \Vert \Delta \vn \Vert_{0,\E}) \Vert \vn \Vert_{\frac{1}{2},\partial \E}.
\end{equation}
Then, we use \eqref{estimate on virtual laplacian} on the second term in the first factor and a one dimensional $hp$ inverse estimate in addition to interpolation theory on the second factor (see \cite{triebel, tartar}),
thus obtaining:
\[
\int_{\partial \E} \frac{\partial \vn}{\partial n} \vn \lesssim \pE^2 \vert \vn \vert_{1,\E} \cdot \pE \Vert \vn \Vert_{0,\partial \E}.
\]
Plugging \eqref{stability inequality A} and \eqref{stability inequality B} in \eqref{integration by parts stability}, we deduce:
\[
\vert \vn \vert^2_{1,\E} \lesssim \vert \vn \vert_{1,\E} \left\{  \pE^2 \Vert \Pizpmd \vn \Vert _{0,\E} + \pE^3 \Vert \vn \Vert_{0,\partial \E}   \right\},
\]
whence
\[
\vert \vn \vert ^2_{1,\E} \lesssim \pE^2 \left(  \pE^2 \Vert \Pizpmd \vn \Vert^2_{0,\E} \right) + \pE^5 \left(  \pE \Vert \vn \Vert^2_{0,\partial \E} \right) \le  \pE^5 \SE(\vn,\vn).
\]

Next, we estimate $\c^*(\E)$. Let $\vn \in \ker(\Pinablap)$, then:
\[
\begin{split}
& \SE(\vn,\vn) = \pE \Vert \vn \Vert^2_{0,\partial \E} + \pE^2 \Vert \Pizpmd \vn \Vert _{0,\E}^2 \\
& \lesssim \pE \Vert \vn \Vert^2_{0,\partial \E} + \pE^2 \Vert \vn -\Pizpmd \vn \Vert _{0,\E}^2 + \pE^2 \Vert \vn -\Pinablap \vn \Vert_{0,\E}^2.
\end{split}
\]
We estimate the three terms separately. We begin with the first one. Applying the multiplicative trace inequality (see e.g. \cite{BrennerScott}),
the $\p$ version of the Aubin-Nitsche duality argument \eqref{Aubin Nitsche convex case} for convex $\E$ and \eqref{modified Aubin Nitsche} for nonconvex $\E$:
\begin{equation} \label{boundary estimates aubin and multiplicative trace}
\begin{split}
\pE \Vert \vn \Vert^2_{0,\partial \E} 	&\lesssim \pE \left(\Vert \vn \Vert_{0,\E} \vert \vn \vert _{1,\E} + \Vert \vn \Vert^2_{0,\E}  \right) = \pE \left(\Vert \vn - \Pinablap \vn \Vert_{0,\E} \vert \vn \vert _{1,\E} + \Vert \vn -\Pinablap \vn \Vert^2_{0,\E}\right)\\
&\lesssim  \begin{cases}
\pE \left(\pE^{-1} \vert \vn \vert_{1,\E}^2 + \pE^{-2} \vert \vn \vert^2_{1,\E} \right) \le \vert \vn \vert_{1,\E}^2, & \text{if } \E \text{ is convex,}\\
\pE \left(\pE^{-\frac{\pi}{\omega_\E}} \vert \vn \vert_{1,\E}^2 + \pE^{-2 \frac{\pi}{\omega_\E}} \vert \vn \vert^2_{1,\E} \right) \le \p^{1-\frac{\pi}{\omega_\E}}\vert \vn \vert_{1,\E}^2, &\text{otherwise,}\\
\end{cases} 
\end{split}
\end{equation}
where we recall $\omega_\E$ is the largest angle in $\E$.

We now deal with the second term; using \cite[Lemma 4.1]{hpVEMbasic}:
\[
\pE^2 \Vert \vn -\Pizpmd \vn \Vert _{0,\E}^2 \lesssim \pE^2 \pE^{-2} \Vert \vn \Vert^2_{1,\E} = \Vert \vn - \Pinablap \vn \Vert^2_{1,\E} \lesssim \vert \vn \vert_{1,\E}^2,
\]
where in the last inequality we used that $\vn - \Pinablap \vn$ has zero average on $\partial \E$.

Finally, we treat the third term; using Aubin-Nitsche argument \eqref{Aubin Nitsche convex case} and its modified version for nonconvex polygon \eqref{modified Aubin Nitsche}:
\begin{equation} \label{equation where we apply Aubin Nitsche}
\pE^2 \Vert \vn - \Pinablap \vn \Vert_{0,\E}^2 \lesssim 
\begin{cases}
\pE^2 \pE ^{-2} \vert \vn -\Pinablap \vn \vert_{1,\E}^2 = \vert \vn \vert^2_{1,\E} & \text{if } \E \text{ is convex},\\
\pE^2 \pE ^{-2 \frac{\pi}{\omega_\E}} \vert \vn -\Pinablap \vn \vert_{1,\E}^2 = \pE ^{2\left( 1- \frac{\pi}{\omega_\E} \right)} \vert \vn \vert^2_{1,\E} & \text{otherwise.}\\
\end{cases}
\end{equation}
Collecting the three bounds, we obtain the claim.
\end{proof}

\begin{remark} \label{remark on stabilization theorem}
In order to keep the notation simpler, we proved Theorem \ref{theorem stability bounds} assuming that the polynomial degrees $\pe$ on each edge $\e \in \mathcal E ^\E$ coincide with the degree of accuracy $\pE$ of the local space $\VnE$;
the same result remains valid if $\pE \approx \pe$, for all $\e \in \mathcal E ^\E$. In view of the forthcoming definition \eqref{assumption mu},
the case of interest in the following will satisfy such condition and therefore, for the proof of the main result of this work, namely Theorem \ref{theorem exponential convergence}, we will not use directly Theorem \ref{theorem stability bounds},
but its nonuniform degree version.
\end{remark}

As a consequence of Theorem \ref{theorem stability bounds}, the quantity $\alpha(\E)$, defined in \eqref{local parameter stability}, can be bounded in terms of $\pE$ as follows:
\begin{equation} \label{anaisotropy dependence on p}
\alpha(\E) = \frac{1+\alpha^*(\E)}{\min_{\E' \in \taun} \alpha_*(\E')} = \frac{1+\max(1,\c^*(\E))}{\min_{\E' \in \taun} (\min(1,\c_*(\E')))} \lesssim
\begin{cases}
\max_{\E \in \taun}\pE^5 & \text{if all }\E \text{ are convex}\\
\pE^{2 \left( 1- \frac{\pi}{\omega} \right)} \max_{\E\in \taun} \pE^5 & \text{otherwise}
\end{cases}.
\end{equation}

\begin{remark}\label{remark quadrature formulas}
Owing to \cite[formula (2.14)]{bernardimaday1992polynomialinterpolationinsobolev}, we could replace the boundary term of $\SE$, defined in \eqref{stability choice}, with a spectrally equivalent algebraic expression employing Gau\ss-Lobatto nodes.
In particular, let $\Ihat=[-1,1]$ and let $\{\rho_j^{\p_{\Ihat}+1}\}_{j=0}^{\p_{\Ihat}}$ and $\{\xi_j^{\p_{\Ihat}+1}\}_{j=0}^{\p_{\Ihat}}$ be the Gau\ss-Lobatto nodes and weights on $\Ihat$ respectively. Then:
\begin{equation} \label{Bernardi Maday formula}
\c \sum_{j=0}^{\p_{\Ihat}} \q^2 (\xi_j^{\p_{\Ihat}+1}) \rho_j^{\p_{\Ihat}+1} \le \Vert \q \Vert^2_{0,\Ihat} \le \sum_{j=0}^{\p_{\Ihat}} \q^2(\xi_j^{\p_{\Ihat}+1}) \rho_j^{\p_{\Ihat}+1},\quad \forall \q \in \mathbb P_{\p_{\Ihat}}(\Ihat),
\end{equation}
where $\c$ is a positive universal constant.
We could replace in \eqref{stability choice} the $L^2$ integral on the boundary with a piecewise Gau\ss-Lobatto combination,
mapping each edge on the reference interval $\Ihat$ and using \eqref{Bernardi Maday formula};
the advantage of such a choice is that we can automatically use the \emph{nodal}
degrees of freedom on the skeleton, assuming that they have a Gau\ss-Lobatto distribution on each edge.

The boundary term of the new stabilization is now very close to the classical stabilization choice (see e.g. \cite{VEMvolley} and \cite{hpVEMbasic})
and its implementation is much easier than the implementation of \eqref{stability choice}, where one should reconstruct polynomials on each edge;
in fact, it suffices to take instead of the Euclidean inner product of \emph{all} the degrees of freedom only the boundary one with some Gau\ss-Lobatto weights.

For additional issues concerning the stabilization (only for the $h$ version of VEM) see \cite{beiraolovadinarusso_stabilityVEM},
while for more details concerning the implementation of the method we refer to \cite{hitchhikersguideVEM}.
\end{remark}

\subsection{Numerical tests for the stability bounds} \label{subsection  numerical tests for the stability bounds}
In Theorem \ref{theorem stability bounds}, we proved the stability bounds \eqref{bilinear form S} for a possible choice of $\SE$.
Such bounds, which also reflect on $\alpha_*(\E)$ and $\alpha^*(\E)$ introduced in \eqref{stability an}, are rigorously proven but have a quite stray dependence on $\p$.
In the following, we check numerically whether the dependence on $\p$ of the above-mentioned constants is sharp.

In order to do that, we note that finding $\alpha_*(\E)$ and $\alpha^*(\E)$ in \eqref{stability an} is equivalent to find the minimum and maximum eigenvalues $\lambda_{min}$ and $\lambda_{max}$ of the generalized eigenvalue problem:
\begin{equation} \label{generalized eigenvalue problem}
\mathbf {A}_n^\E \mathbf {\vn} = \lambda \mathbf {A}^\E \mathbf {\vn},\quad \vn \in \VnE,
\end{equation}
Here, $\mathbf A^\E_n$ and $\mathbf A^\E \in \mathbb R^{\dim(\VnE) \times \dim(\VnE)}$ are defined as
\[
(\mathbf {A}^\E_n) _{i,j} = \anE(\varphi_i, \varphi _j),\qquad (\mathbf A^\E)_{i,j} = \aE(\varphi_i, \varphi_j)
\]
where $\{\varphi_i\}_{i=1}^{\dim(\VnE)}$ denotes the virtual canonical basis of $\VnE$, see \eqref{definition canonical basis}.
We are adopting the usual notation, by calling $\mathbf{\vn} \in \mathbb R^{\dim(\VnE)}$ the vector of the degrees of freedom associated with $\vn \in \VnE$.

We note that we restrict our analysis on functions having zero average on $\E$, since both $\mathbf A_n^\E$ and $\mathbf A^\E$ have constant functions in their kernel;
this strategy allows to avoid the problems related to solving the generalized eigenvalue problem for singular matrices.
Moreover, the entries of matrix $\mathbf A^\E$ are not computable exactly, since virtual functions are not known explicitly;
therefore, we approximate them by solving numerically the associated diffusion problem, by means of a \emph{fine} and high-order finite element approximation.

In Table \ref{tabular generalized eigenvalue}, we present the results on three different types of polygon (namely, those which we will employ for the tests in the forthcoming Section \ref{section numerical results}):
a square, a nonconvex decagon (like any of the polygons in the outer layer of Figure \ref{figure possible meshes on Lshaped}, b), a nonconvex hexagon (like any of the polygons in the outer layer of Figure \ref{figure possible meshes on Lshaped}, c).
\begin{table}[H]
	\begin{tabular}{|c|c|c|c|c|c|c|}
		\hline
		$\p$ & sq. $\lambda_{min}$ & sq. $\lambda_{max}$ & dec. $\lambda_{min}$ & dec. $\lambda_{max}$ & hex. $\lambda_{min}$ & hex. $\lambda_{max}$ \\
		\hline

		2			& 7.8559e-01&  1.0000e+00&  7.9262e-02&  5.5516e+00&  1.6168e-01 & 1.1183e+00\\
		3 			& 4.6667e-01&  1.0000e+00&  1.0306e-01&  8.6605e+00&  1.3342e-01 & 1.4751e+00\\
		4 			& 3.3195e-01&  1.0000e+00&  4.5039e-02&  1.0852e+01&  1.0321e-01 & 1.6253e+00\\
		5 			& 2.7547e-01&  1.0000e+00&  3.4944e-02&  1.0513e+01&  7.4247e-02 & 1.8672e+00\\
		6 			& 2.1557e-01&  1.0000e+00&  2.3463e-02&  1.1835e+01&  5.5556e-02 & 1.6707e+00\\
		7 			& 1.8994e-01&  1.0000e+00&  2.0730e-02&  9.7514e+00&  3.5664e-02 & 1.9013e+00\\
		8			& 1.4136e-01&  1.0000e+00&  1.6122e-02&  1.0447e+01&  2.7559e-02 & 1.8801e+00\\
		9			& 1.2446e-01&  1.0000e+00&  1.8555e-02&  7.9781e+00&  2.1313e-02 & 1.8337e+00 \\
		10			& 9.2933e-02&  1.0000e+00&  1.3736e-02&  3.9577e+01&  1.7991e-02 & 5.6544e+00\\
		\hline
		\end{tabular}
	\caption {Minimum and maximum eigenvalues of the generalized eigenvalue problem \eqref{generalized eigenvalue problem} on: sq.= a square; dec.= a nonconvex decagon; hex.= a nonconvex hexagon.} \label{tabular generalized eigenvalue}
\end {table}
As theoretically expected, the maximum generalized eigenvalue always scales like 1. On the contrary, the minimum eigenvalue behaves in all the three cases like $\p^{-1}$.
This means that in fact the bounds of Theorem \ref{theorem stability bounds} are abundant, whereas the \emph{actual} behaviour of the stability bounds may be much milder.
Unfortunately, currently we are not able to improve the stability bounds of Theorem \ref{theorem stability bounds}.
It is worth mentioning that this has no impact on the asymptotic exponential convergence results in the next section.

\section {Exponential convergence for corner singularity on geometric meshes} \label{section exponential convergence for corner singularity on geometric meshes}
In this section, we want to show that exponential convergence is achieved if geometric mesh refinement and degree of accuracy distribution are chosen appropriately.

In order to achieve such a convergence we employ geometrically graded polygonal meshes, which will be discussed in Section \ref{subsection geometric meshes}.
Then, we show in Section \ref{subsection discontinuous estimates} estimates for the first and the second terms in the error decomposition \eqref{global VEM decomposition},
in particular proving bounds for the local right-hand side approximation and for the local best approximation by means of polynomials.
In Section \ref{subsection virtual estimates}, we obtain estimates for the third term in \eqref{global VEM decomposition},
in particular illustrating bounds for the best approximation by means of functions belonging to the virtual space defined in \eqref{global non uniform virtual space}.
Finally, in Section \ref{subsection exponential convergence}, under a proper choice for the polynomial degree vector \textbf{p} introduced in Section \ref{section VEM with non uniform degree of accuracy}
and the sequence $\{\taun\}_n$ of polygonal decompositions, we combine together the above error bounds; as a consequence, we guarantee exponential convergence for the error in the energy norm in terms of the number of degrees of freedom
of the global virtual space $\Vn$ defined in \eqref{global non uniform virtual space}.

\subsection{Geometric meshes} \label{subsection geometric meshes}
Here, we describe a class of sequences of \emph{nested} geometric meshes which we will employ later in order to show error convergence.
We recall we are assuming that the only ``singular'' corner is the origin $\mathbf 0 \in \partial \Omega$, see \eqref{highlight one singularity}.
Let $\sigma\in (0,1)$ be the grading parameter of the mesh.

The decomposition $\taun$ consists of $n+1$ layers defined as follows.
We set $L_0$ the $0$-th layer as the set of all polygons $\E$ in decomposition $\taun$ such that $\mathbf 0 \in \mathcal V ^\E_n$; next, we define by induction:
\begin{equation} \label{jth layer}
L_j = \left\{ \E_1 \in \taun \mid \overline {\E_1} \cap \overline{\E_2} \neq \emptyset, \, \text{ for some }\E_2 \in L_{j-1},\, \E_1 \nsubseteq \cup_{i=0}^{j-1} L_i \right\},\, j=1,\dots,n.
\end{equation}

We set $\mathcal T_0 = \{\Omega\}$. Given $\taun$, the decomposition $\mathcal T_{n+1}$ is obtained by refining $\taun$ \emph{only} in the layer around the singularity (i.e. $L_0$).
We require that at level $n$, the decomposition satisfies the following grading condition:
\begin{itemize}
\item[(\textbf{D3})]
\begin{equation} \label{graded like mesh property}
 h_{\E}\approx 
\begin{cases}
\displaystyle{\frac{1-\sigma}{\sigma}}\dist (\mathbf 0, \E),& \text{if } \E \notin L_0\\
\sigma^{n},& \text{otherwise}.\\
\end{cases}
\end{equation}
Furthermore, the number of elements in each layer is uniformly bounded with respect to the discretization parameters.
We will also assume that $\pE \ge 2$.
A more precise choice will be discussed in the forthcoming definition \eqref{assumption mu}.
\end{itemize}

Assumption (\textbf{D3}) justifies the name geometric for the sequence; more specifically, the closer a polygon is to \textbf{0} the smaller its diameter is.
Moreover, it is possible to check that the ratio between the size of two neighbouring layers is proportional to $\frac{1-\sigma}{\sigma}$.
As a consequence of assumption (\textbf{D3}), we also have, for $\E \in L_j$, $h_\E\approx\sigma ^{n-j}$.

\exmp \label{example graded square Schwab}
A possible sequence satisfying (\textbf{D1})-(\textbf{D3}) is the graded mesh of squares elements with hanging nodes on the $L$-shaped domain, that is used in \cite[Definition 4.30]{SchwabpandhpFEM}, see Figure \ref{figure possible meshes on Lshaped} (left).
We note that in the VEM context, this mesh contains pentagons and squares, whereas in the Finite Element counterpart the very same mesh is ``afflicted'' by the presence of squares with hanging nodes.

\exmp	\label{example graded decagons Chernov}
Another choice is depicted in Figure \ref{figure possible meshes on Lshaped} (center). 
This mesh is obtained by merging all the elements that correspond to one layer in the mesh from Example \ref{example graded square Schwab} in a single large element.
We observe that this mesh is made of $n$ decagons and one hexagon around \textbf{0}.
Moreover, we want to stress the fact that this mesh, that cannot be used in the conforming FEM enviroment, needs less then one third of the degrees of freedom of the previous one.
Finally, we point out that such a mesh does not satisfy the star-shapedness assumption (\textbf{D1}).\\

\exmp	\label{example graded decagons Chernov with cut}
As a third example, see Figure \ref{figure possible meshes on Lshaped} (right), we modify the mesh in Example \ref{example graded decagons Chernov} by adding an oblique cut on the ``central'' diagonal.
This mesh still cannot be used for conforming FEM approximations.
Obviously, it contains more elements and hence will result in more degrees of freedom than the mesh in Example \ref{example graded decagons Chernov}.
Notwithstanding, it satisfies (\textbf{D1}). Moreover, it satisfies also the technical assumption (\textbf{D4}) we will need in what follows, whereas the mesh in Example \ref{example graded decagons Chernov} does not.

\begin{figure}  [h]
\centering
\subfigure {\includegraphics [angle=0, width=0.31\textwidth]{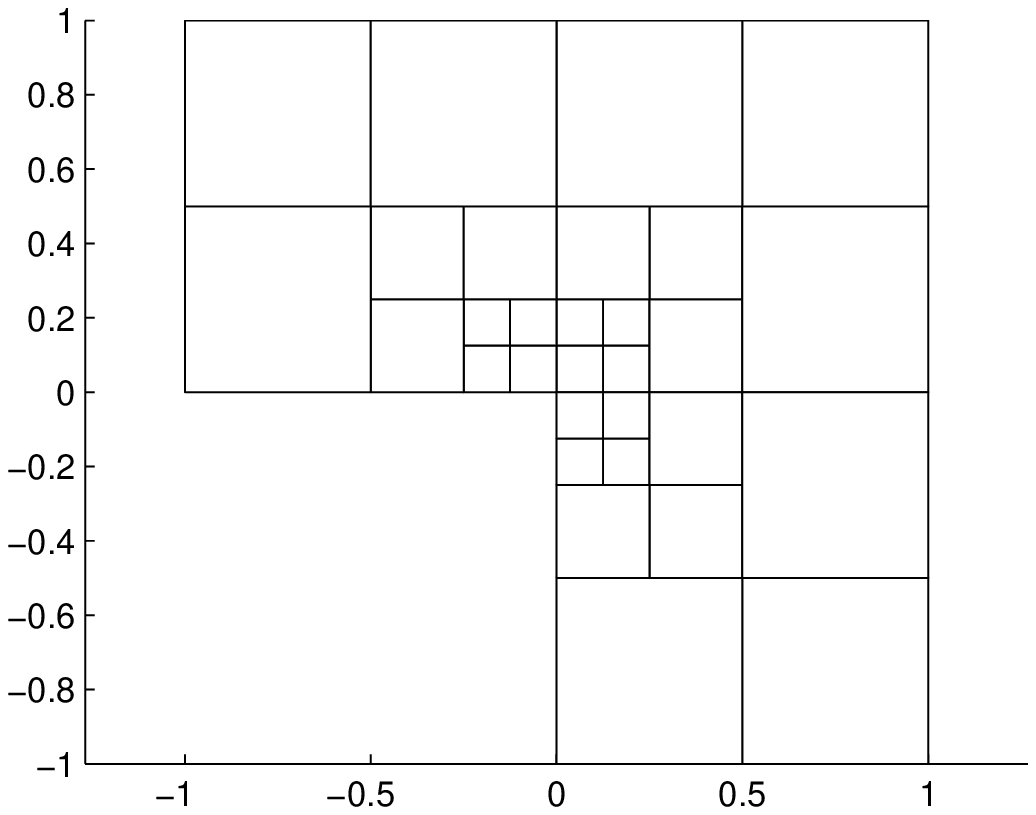}}
\subfigure {\includegraphics [angle=0, width=0.31\textwidth]{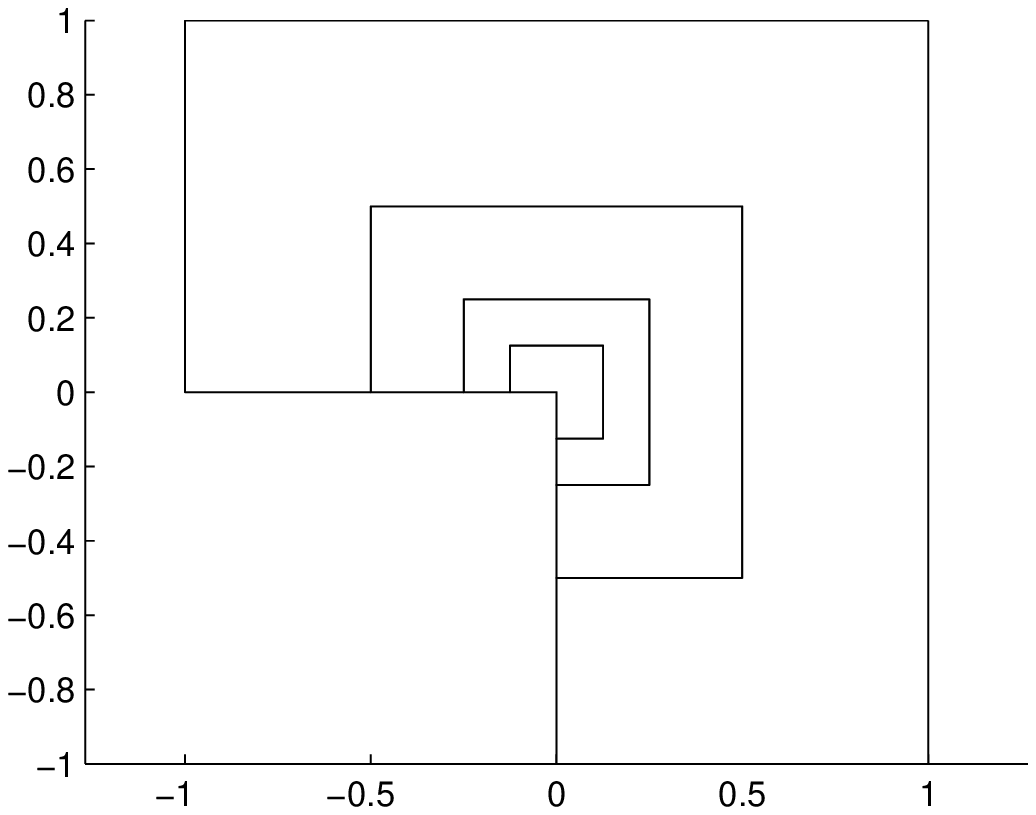}}
\subfigure {\includegraphics [angle=0, width=0.31\textwidth]{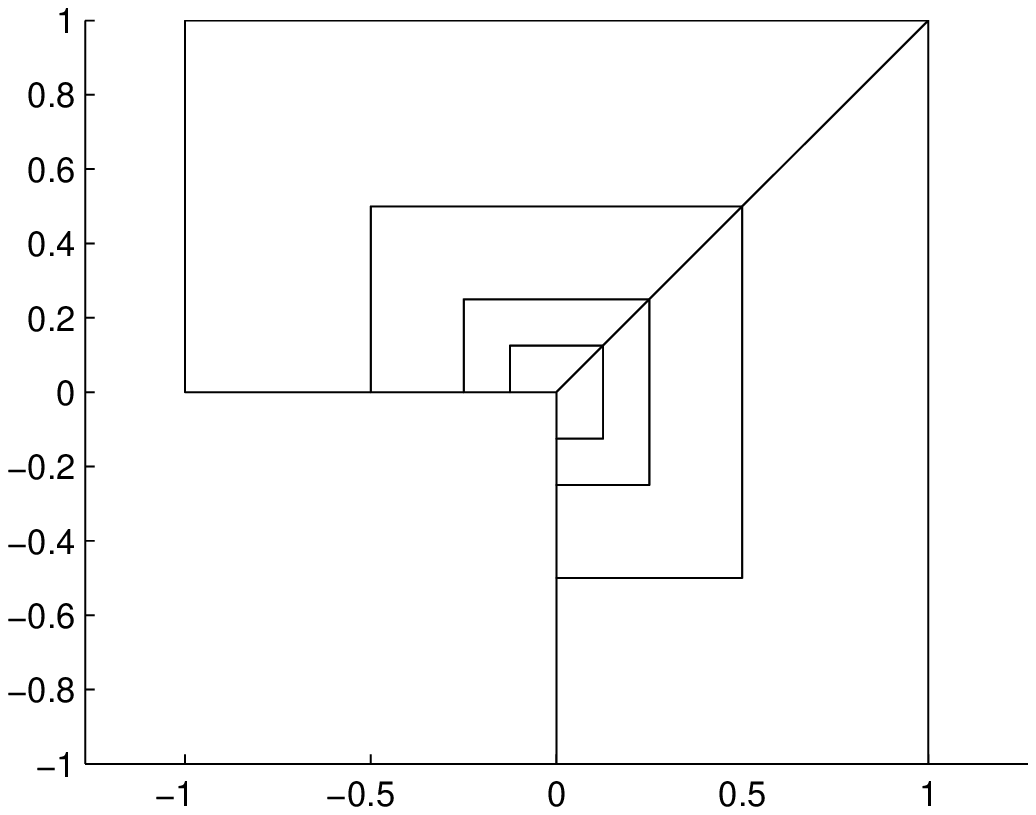}}
\caption{Decomposition $\taun$, $n=3$, for Example \ref{example graded square Schwab} (left), Example \ref{example graded decagons Chernov} (center), Example \ref{example graded decagons Chernov with cut}(right).}
\label{figure possible meshes on Lshaped}
\end{figure}

\medskip
\medskip
We require the following additional assumption on the geometry of the decomposition. We will need it to state approximation results in Sections \ref{subsection discontinuous estimates} and \ref{subsection virtual estimates}.

\begin{itemize}
\item [(\textbf{D4})]
Let $\taun$ be a geometric polygonal decomposition; write $\taun = \taun ^0 \cup \taun^1$, where $\taun^0 = L_0$ and $\taun^1=\cup_{j=1}^n L_j$.
Then, there exists a collection $\Cn^1$ of squares such that:
\begin{itemize}
\item $\card(\Cn^1)=\card(\taun^1)$; for each $\E\in \taun^1$, there exists $\Q=\Q(\E)\in \Cn^1$ such that $\Q\supseteq \E$ and $h_{\E}\approx h_{\Q}$; in addition, it must hold dist($\mathbf 0,Q(\E$))$\approx h_{\E}$;
\item every $\mathbf x \in \Omega$ belong at most to a fixed number of squares $\Q$, independently on all the discretization parameters;
\item $\forall \E \in \taun^0$, $\E$ is star-shaped with respect to $\mathbf 0$; moreover, the subtriangulation of $\E$ obtained by joining $\mathbf 0$ with the other vertices is uniformly shape regular
($\gamma$ being the shape-regularity constant).
\end{itemize}

We set $\Omegaext_{n} = (\cup_{\Q\in \Cn^1} \Q) \cup (\cup_{\E \in \taun^0} \E)$.
\end{itemize}

We point out that (\textbf{D4}) seems to be a rather technical requirement.
Indeed, we will show in Section \ref{section numerical results} that also meshes not satisfying (\textbf{D4}) may produce the expected convergence behaviour shown in Theorem \ref{theorem exponential convergence}.

We note that (\textbf{D4}) is in the spirit of the strategy of the overlapping square technique used in \cite{hpVEMbasic, cangianigeorgoulishouston_hpDGFEM_polygon}.
We here additionally require that squares covering polygons far from the singularity cannot cover also such a singularity (since in this case $\p$ approximation results would not hold, thus invalidating Theorem \ref{theorem exponential convergence}).
We also stress that the decomposition in Example \ref{example graded decagons Chernov} does not satisfy neither (\textbf{D1}) nor (\textbf{D4}).
Finally, we point out that instead of considering a decomposition of squares $\Cn$, it is possible to consider in (\textbf{D4}) a decomposition in sufficiently regular quadrilaterals (e.g. parallelograms),
since the same analysis by means of Legendre polynomials that follows (for instance in Lemmata \ref{lem local discontinuous approximation far from origin} and \ref{lem local discontinuous approximation far from origin loading term}) could be performed.

\subsection{Local approximation by polynomials} \label{subsection discontinuous estimates}
Here, we deal with the approximation of the first and the second term in the right-hand side \eqref{global VEM decomposition}.
What we are going to prove are $hp$ approximation properties by means of local polynomials on polygons.
In $hp$-FEM literature, classical approximation of this type is not effectuated on general polygons but only on squares and triangles,
see \cite{SchwabpandhpFEM, lishen_optimalerrorestimateJacobi, guowang_JacobiapproximationSobolev, babuskasurihpversionFEMwithquasiuniformmesh, BabuSuri_optimalconvergence} and the references therein.\\
The basic tool behind this approach is the employment of orthogonal bases, namely tensor product of Legendre polynomials on the square, see \cite{SchwabpandhpFEM},
and Koornwinder polynomials (that is collapsed tensor product of Jacobi polynomials) on triangles, see \cite{dubinerspectraltriangles, koornwinderclassicalorthogonalpolynomials};
with such basis, explicit computations can be performed, owing to properties of Legendre and Jacobi polynomials.
On a generic polygon an explicit basis with good approximation properties is not available.
%

The error analysis follows the lines of \cite{SchwabpandhpFEM, hpVEMbasic} and is summarized below.
Let $\mathbf p$ be the vector of the local degree of accuracy on each polygon.
We recall that we denote with $S^{\mathbf p,-1}(\taun, \Omega)$ the space of piecewise discontinuous polynomials over the decomposition $\taun$ of degree $\pE$ on each polygon $\pE$.

The first result is a polynomial approximation estimate regarding regular functions on polygons far from the singularity.
This result will be used for the approximation of the local second term in \eqref{global VEM decomposition} for the elements $\E$ separated from the singularity.
\begin{lem} \label{lem local discontinuous approximation far from origin}
Under assumptions (\textbf{D1})-(\textbf{D4}), let $\E\in L_j$, $j = 1, \dots, n$. Let $\Q(\E)$ be defined in (\textbf{D4}) and let
$\u \in H_{\beta}^{s_\E+3,2}(\Q(\E))$, $1\le s_{\E}\le p_{\E}$.
Then, there exists $\Phi\in \mathbb P_{p_{\E}}(\Q(\E))$ such that:
\begin{equation} \label{error approximation 4 discontinuous far from origin}
\Vert  D^m(u-\Phi) \Vert ^2_{0,\E} \lesssim \sigma ^{2(n-j)(2-m-\beta)} \frac{\Gamma (p_{\E} - s_{\E} + 1)}{\Gamma (p_{\E} + s_{\E} + 3 - 2m)} \left( \frac{\rho}{2} \right)^{2s_{E}} \vert u  \vert^2_{H_{\beta}^{s_\E+3,2}(\Q(\E))}
\end{equation}
where $m=0,1,2$;
$2\le j\le n+1$; $\rho=\max(1,\frac{1-\sigma}{\sigma})$, $\sigma$ is the grading parameter of the mesh and $\Gamma$ is the Gamma function.
\end{lem}
\begin{proof}
The result follows from classical scaling arguments and \cite[Lemma 4.53]{SchwabpandhpFEM}. Here, we only give the idea of the proof.
Firstly one encapsulates polygon $\E$ into the corrensponding square $\Q(\E)$.
It is possible to bound the left hand side of inequality \eqref{error approximation 4 discontinuous far from origin} with the same (semi)norm on the square.
After that, the square is mapped into the reference square $\widehat Q = [-1,1]^2$ and a $\p$ analysis by means of tensor product of Legendre polynomials is developed (see \cite[Theorem 4.46]{SchwabpandhpFEM}).
Subsequently, the reference square is pushed forward to square $\Q$. Using the property of the geometric mesh stated in assumption (\textbf{D3}) and \cite[Lemma 4.50]{SchwabpandhpFEM}, the result follows.
\end{proof}

Estimate on polygons around the singularity are discussed in the following lemma.
We point out that for the error control in layer $L_0$ we can work directly on the element without the need of employing covering squares, as effectuated for the analysis on the polygons of the other layers,
see Lemma \ref{lem local discontinuous approximation far from origin}.
The proof is an extension to polygonal domains of that in Theorem \cite[Lemma 4.16]{SchwabpandhpFEM}.
\begin{lem} \label{lem local discontinuous approximation near origin}
Under assumptions (\textbf{D1})-(\textbf{D4}), let $\E \in L_0$. Let $\u \in H_{\beta}^{2,2}(\E)$, $\beta\in [0,1)$.
Then, there exists $\Phi\in \mathbb P_{1}(\E)$ such that:
\begin{equation} \label{error approximation 4 discontinuous near origin}
\vert u-\Phi \vert ^2_{1,\E} \lesssim h_{\E} ^{2(1-\beta)} \Vert \vert \mathbf x \vert ^{\beta} |D^2 u| \Vert ^2_{0, \E} \lesssim  \sigma ^{2(1-\beta)n} \Vert \vert \mathbf x \vert ^{\beta} |D^2 u| \Vert ^2_{0, \E},
\end{equation}
where $\sigma$ is the grading factor from assumption (\textbf{D3}).
\end{lem}
\begin{proof}
We start by proving the following Hardy inequality on polygons with a vertex at $\mathbf 0$.
Let $\alpha >0$, let be given a function $u$ such that $\int_\E \vert \mathbf x \vert ^\alpha \vert D^1 u \vert ^2 < +\infty $ and $u \in \mathcal C^0(\overline \E)$. Then:
\begin{equation} \label{Hardy inequality}
\int _\E \vert \mathbf x \vert ^{\alpha-2} \vert u - u(\mathbf 0) \vert ^2 \le c \int_\E \vert \mathbf x \vert ^\alpha \vert D^1 u \vert ^2.
\end{equation}
We consider the regular subtriangulation by joining $\mathbf 0$ with the other vertices of $\E$; the existence of such a decomposition is guaranteed by assumption (\textbf{D4}).
Thanks to \cite[Lemma 4.18]{SchwabpandhpFEM}, the ``triangular'' counterpart of \eqref{Hardy inequality} holds:
\begin{equation} \label{local Hardy inequality}
\int _\T \vert \mathbf x \vert ^{\alpha-2} \vert u - u(\mathbf 0) \vert ^2 \le c \int_\T \vert \mathbf x \vert ^\alpha \vert D^1 u \vert ^2,\quad \forall\, \T \text{ in the subtriangulation of $\E$}.
\end{equation}
It suffices then to split the integral over $\E$ into a sum of integrals over the triangles of the subtriangulation, apply \eqref{local Hardy inequality} and collect all the terms.

Using \eqref{Hardy inequality} and applying the argument of \cite[Lemma 4.19]{SchwabpandhpFEM} to the polygon $\E$, we observe that $H^{2,2}_\beta(\E)$ is compactly embedded in $H^1(\E)$.
Using such a compact embedding and proceeding as in \cite[Lemma 4.16]{SchwabpandhpFEM}, the following inequality holds true for a polygon $\E$ star-shaped with respect to $\mathbf 0$:
\begin{equation} \label{inequality following compact embedding}
\vert U \vert _{1,\E}^2 \lesssim h_\E^{2(1-\beta)} \Vert |\mathbf x|^{\beta} D^2 U \Vert^2_{0,\E} + \sum_{i=1}^3 \vert U(\mathbf A_i) \vert^2,\quad \forall \, U  \in H^{2,2}_\beta (\E),
\end{equation}
where $\{ \mathbf A_i \}_{i=1}^3$ is a set of three arbitrary nonaligned vertices of $\E$.\\
Let $\Phi$ be the linear interpolant of $\u$ at $\mathbf A_i$, $i=1,\dots, 3$. Then, plugging $U = \u - \Phi$ in \eqref{inequality following compact embedding},
noting that $U(A_i)=0$, $i=1,2,3$, and using the geometric assumption (\textbf{D3}), we get the claim.

\end{proof}

We note that \eqref{error approximation 4 discontinuous near origin} does not rely on $\p$ approximation results, but only on scaling argument. 
This will be enough in order to prove the main result of this work, that is Theorem \ref{theorem exponential convergence},
and it is in accordance with the choice of the vector of local degrees of accuracy that will be effectuated in the forthcoming definition \eqref{assumption mu}.
We emphasize that this is in the spirit of classical $hp$ refinement, see \cite{SchwabpandhpFEM}.


We now turn our attention to the approximation of the first local term in \eqref{global VEM decomposition}, i.e. to the local approximation of the loading term.
Since we are approximating it with piecewise polynomials of local degree $\pE-2$, we set $\overline {\mathbf p} = \mathbf p-2$, i.e. $\forall \E \in \taun$, $\overline p_{\E}= p_{\E}-2$.
We have, for all $\vn \in \Vn$:
\begin{equation} \label{loading decomposition}
( \fn, \vn )_{0,\E} -(f,\vn)_{0,\E}  = \sum _{\E\in \taun} \int _{\E} (\Pi ^0 _{p-2, \E}f-f)(\vn - \Pi^0_{0,\E}\vn) =: \sum_{\E\in \taun}  F_{\E}(\vn),
\end{equation}
where we recall we are assuming for the sake of simplicity $\pE \ge 2$, $\forall \E \in \taun$, see Section \ref{section VEM with non uniform degree of accuracy}.\\

As above, we develop a different analysis for polygons near and far from the singularity. We start with the ``far'' case.
\begin{lem} \label{lem local discontinuous approximation far from origin loading term}
Under assumptions (\textbf{D1})-(\textbf{D4}), let $\E\in L_j$, $j = 1,\dots, n$. Let $\Q(\E)$ be defined in (\textbf{D4}). Let $f \in H_{\beta}^{\overline s_{\E}+3,2}(\Q(\E))$, $0\le \overline s_{\E}\le \overline \p_\E$, with $\overline \p_\E = \pE-2$.
Then, for all $\vn \in \VnE$,
\[
F_{\E}(\vn)	 \le \vert \vn \vert _{1,\E} \left\{ \sigma^{(n-j)(2-\beta)}  \left(\frac{\Gamma(\overline p_{\E}- \overline s_{\E} + 1)}{\Gamma (\overline p_{\E} + \overline s_{\E} +1)}\right)^{\frac{1}{2}} 
\left( \frac{\rho}{2}\right)^{\overline s_{\E}}    \vert f \vert _{H_{\beta}^{\overline s_{\E}+3,2}(\Q(\E))}      \right\}
\]
with the same notation of Lemma \ref{lem local discontinuous approximation far from origin}.
\end{lem}
\begin{proof}
It suffices to use a Cauchy-Schwarz inequality in \eqref{loading decomposition}, standard bounds for the projection errors
and analogous estimate to those in Lemma \ref{lem local discontinuous approximation far from origin}.
\end{proof}
Assume now that $\E$ is an element in the finest level $L_0$. We work here a bit differently from what we did in Lemma \ref{lem local discontinuous approximation near origin}. In particular we get the following.
\begin{lem} \label{lem local discontinuous approximation near origin loading term}
Under assumptions (\textbf{D1})-(\textbf{D3}), let $\E\in L_0$. Assume $f \in L^2(\E)$. Let $\beta \in [0,1)$. Then:
\[
F_{\E}(\vn) \le h_{\E}^{1-\beta} \vert \vn \vert _{1,\E} \Vert f \Vert _{0, \E} \lesssim \sigma^{n(1-\beta)} \vert \vn \vert_{1,\E},\quad \forall \vn \in \VnE,
\]
where $\sigma$ is the grading factor from assumption (\textbf{D3}).
\end{lem}
\begin{proof}
Using a Cauchy-Schwarz inequality and Bramble Hilbert lemma (see \cite{BrennerScott}), we obtain:
\[
F_{\E}(\vn) \lesssim h_{\E} \vert \vn \vert _{1,\E} \Vert f \Vert _{0,\E} \lesssim h_{\E}^{1-\beta} \vert \vn \vert _{1,\E} \Vert f \Vert _{0,\E}\lesssim \sigma^{n(1-\beta)}\vert \vn \vert_{1,\E},\quad \forall \vn\in \VnE.
\]
\end{proof}

We point out that for the proof of Lemmata \ref{lem local discontinuous approximation near origin} and \ref{lem local discontinuous approximation near origin loading term}
we work directly on the polygon without the need of using the covering squares technique of assumption (\textbf{D4}),
like in Lemmata \ref{lem local discontinuous approximation far from origin} and \ref{lem local discontinuous approximation far from origin loading term}.
This justifies the fact that in assumption (\textbf{D4}) we did not require the existence of a collection of squares $\mathcal C_n^0$ associated with the finest layer $L_0$ but only the existence of collection $\mathcal C_1^n$
associated with all the other layers.

\subsection{Approximation by functions in the virtual space} \label{subsection virtual estimates}
Here, we treat the approximation of the third term in the right-hand side of \eqref{global VEM decomposition}.
We observe that this term has two main differences with respect to the other two.
The first difference is that we need an approximant $\uI$ which is globally continuous; the second one is that $\uI$ is not a piecewise polynomial but a function belonging to the virtual space $\Vn$.

As done in Section \ref{subsection discontinuous estimates}, we split the analysis into two parts.
Firstly, we work on polygons abutting the singularity, see Lemma \ref{lemma virtual approximation far from sing}; secondly, we work on elements $\E$ in the first layer $L_0$,
see Lemma \ref{lemma virtual approximation near sing}.

\begin{lem} \label{lemma virtual approximation far from sing}
Let assumptions (\textbf{D1})-(\textbf{D4}) hold. Let $\E \in L_j$, $j=1,\dots, n$.
Let $f$, the right-hand side of \eqref{continuous weak problem}, belong to space $\mathcal B _{\beta}^0 (\Omega)$; consequently,
$\u$, the solution of problem \eqref{continuous weak problem}, belongs to space $\mathcal B _{\beta}^2 (\Omega)$, see Theorem \ref{theorem Babuska regularity} and definition \eqref{countably normed space}.
Assume that $\pe \approx \pE$ if $\e \in \partial \E$. Assume moreover that if $\E \in L_1$, then $\pE \approx 2$.
Then, for all $1\le \sE \le \pE$, there exists $\uI \in \VnE$ such that:
\begin{equation} \label{claim lemma virtual approximation far from sing}
\vert \u - \uI \vert _{1,\E}^2 \lesssim \Vert f - \Pizpmd f \Vert_{0,\E}^2 + \sigma^{(n-j)(3-2\beta)} \pE ^{-2\sE -1 } \left( \frac{\rho e}{2}  \right)^{2\sE+1} \sum_{\e \in \mathcal E ^\E} \vert \u \vert^2_{H_\beta^{\sE+1,2}(\e)},
\end{equation}
where we recall that $\Pizpmd$ is the $L^2(\E)$ orthogonal projection from $\VnE$ into $\mathbb P_{\pE-2}(\E)$,
$\sigma$ is the grading factor from assumption (\textbf{D3}) and $\rho = \max \left( 1, \frac{1-\sigma}{\sigma} \right)$.
\end{lem}
\begin{proof}
Before starting the proof, we observe that the boundary norm in the right-hand side of \eqref{claim lemma virtual approximation far from sing} exists,
since $\u \in \mathcal B^2_{\beta}(\Omega)$ implies that $\u \in H^t(\E)$ for all $t \in \mathbb N$ and polygons $\E \notin L_0$.

We define $\uI$ as the \emph{weak} solution of the following problem:
\begin{equation} \label{definition virtual approximant}
\begin{cases}
-\Delta \uI = \Pizpmd f & \text{in } \E,\\
\uI = \uE & \text{on } \partial \E,\\
\end{cases}
\end{equation}
where $\uE \in \mathbb B(\partial \E)$, see \eqref{definition local non uiniform boundary},
is defined in the following way. Assume for the time being that $\E \notin L_1$. Let $\Ihat = [-1,1]$.
Given an edge $\e \subseteq \partial \E$, $\uE$ is defined as the push-forward of a function $\uhatE \in \mathbb P_{\pe}(\Ihat)$ which we fix as follows.
Let $\uhat$ be the pull-back of $\u|_\e$ on $\Ihat$.
Then, $\uhatE'$ is the Legendre expansion of $\uhat$ up to order $\pe - 1$. In particular, we write:
\begin{equation} \label{Legendre expansion}
\uhat'(\xi) = \sum_{i=0}^{\infty} c_i L_i(\xi),\quad  \quad \uhatE'(\xi) = \sum_{i=0}^{\pe-1} c_i L_i(\xi).
\end{equation}
Here $\{L_i(\xi)\}_{i=0}^{\infty}$ is the $L^2(\Ihat)$ orthogonal basis of Legendre polynomials, with $L_i(-1)=(-1)^i$ and $L_i(1)=1$.
Next, we define $\uhatE$ as:
\[
\uhatE(\xi) = \int_{-1}^{\xi} \uhatE'(\eta) d\eta + \uhat(-1).
\]
It is possible to prove that $\uhatE$ interpolates $\uhat$ at the endpoints of $\Ihat$ using the definition of $\uhatE$ and the fundamental theorem of calculus.
Recalling \cite[Theorem 3.14]{SchwabpandhpFEM} and using simple algebra, the following holds true:
\begin{equation} \label{virtual error boundary}
\Vert \uhat - \uhatE \Vert_{\ell, \Ihat} \lesssim e^{\sE} \pe ^{-\sE - 1 + \ell} \vert \u \vert_{\sE+1, \Ihat}, \quad \ell=0,\,1,\quad \forall\; 1\le \sE \le \pE.
\end{equation}
Applying a scaling argument, interpolation theory (see \cite{triebel, tartar}) and summing on all the edges, we get:
\begin{equation} \label{virtual error boundary}
\Vert \u - \uE \Vert_ {\frac{1}{2},\partial \E}^2 \lesssim \sum_{\e \in \mathcal E^\E} e^{2\sE+1} \left( \frac{h_\e}{\pe} \right)^{2\sE + 1} \vert \u \vert^2_ {\sE + 1, \e}, \quad \forall \;1\le \sE \le \pE.
\end{equation}
If now $\E \in L_1$, we define $\uE|_\e$ as above if $\e$ does not belong to the interface between $L_0$ and $L_1$, otherwise $\uI$ is defined as the linear interpolant of $\u$ at the two endpoints of $\e$.
We point out that
\eqref{virtual error boundary} remains valid also if $\E \in L_1$ paying an additional constant $c^{2\sE+1}$, since $\pE \approx 2$ whenever $\E\in L_1$.
We also note that \eqref{virtual error boundary} implies, recalling that $\pe \approx \pE$ if $\e \subseteq \partial \E$ and following the ideas in \cite[Lemma 3.39]{SchwabpandhpFEM}:
\begin{equation} \label{virtual error H12 boundary}
\Vert \u - \uI \Vert^2_{\frac{1}{2},\partial \E} = \Vert \u - \uE \Vert^2_{\frac{1}{2},\partial \E} \lesssim \sigma^{(n-j)(3-2\beta)} \pE^{-2\sE - 1} \left(  \frac{\rho e}{2}  \right)^{2\sE + 1} \sum_{\e \in \mathcal E^\E} \vert \u \vert^2_ {H^{\sE+1,2}_{\beta}(\e)},
\end{equation}
where we recall that $j$ denotes the number of the layer to which $\E$ belongs.

We are now ready to prove the error estimate. For arbitrary constants $c_1$ and $c_2$, there holds (also recalling that $(f- \Pizpmd f)$ is $L^2$-orthogonal to constants):
\[
\begin{split}
\vert \u - \uI \vert^2_{1,\E}	& = \int_\E \vert \nabla (\u - \uI - c_1) \vert^2 = \int _{\partial \E} \frac{\partial (\u - \uI)}{\partial \mathbf n} (\u - \uI-c_1) - \int_\E (f- \Pizpmd f) (\u - \uI - c_2)\\
					& \le \left\Vert \frac{\partial (\u - \uI)}{\partial \mathbf n} \right \Vert_{-\frac{1}{2},\partial \E} \Vert \u - \uE -c_1\Vert_{\frac{1}{2},\partial \E} + \Vert f- \Pizpmd f \Vert_{0,\E} \Vert \u - \uI -c_2\Vert_{0,\E}.
\end{split}
\]
Applying the trace inequalities on Neumann and Dirichlet traces, choosing $c_2$ to be the average of $\u - \uE$ on $\E$ and applying a Poincar\'e inequality, we get:
\[
\begin{split}
\vert \u - \uI  \vert_{1,\E}^2	& \lesssim \left(  \vert \u - \uI \vert_{1,\E}  +\Vert f - \Pizpmd f \Vert_{0,\E}  \right) \Vert  \u - \uE -c_1 \Vert_{\frac{1}{2},\partial \E}  +\Vert f - \Pizpmd f \Vert_{0,\E} \vert \u - \uI \vert_{1,\E}\\
					& \lesssim  \Vert  \u - \uI - c_1 \Vert_{1,\E} \left\{ \Vert f - \Pizpmd f \Vert_{0,\E} + \Vert \u - \uI -c_1 \Vert_{\frac{1}{2},\partial \E}  \right\}.
\end{split}
\]
We deduce, picking $c_1$ to be the average of $\u-\uI$ on $\partial \E$ and applying a Poincar\'e inequality:
\[
\vert \u - \uI \vert^2_{1,\E} \lesssim \Vert f- \Pizpmd f \Vert^2_{0,\E} + \Vert \u - \uE \Vert^2_{\frac{1}{2},\partial \E}.
\]
In order to conclude, it suffices to apply \eqref{virtual error H12 boundary}.
\end{proof}
We turn now our attention to the approximation on the polygons abutting the singularity.
\begin{lem} \label{lemma virtual approximation near sing}
Let assumptions (\textbf{D1})-(\textbf{D4}) hold.
Let $f$, the right-hand side of \eqref{continuous weak problem}, belong to space $\mathcal B _{\beta}^0 (\Omega)$; consequently,
$\u$, the solution of problem \eqref{continuous weak problem}, belongs to space $\mathcal B _{\beta}^2 (\Omega)$, see Theorem \ref{theorem Babuska regularity} and definition \eqref{countably normed space}.
Assume that $\pE=2$ if $\E \in L_0$ and $\pE \approx 2$ if $\E \in L_1$. Then there exists $\uI \in \VnE$ such that:
\[
\vert \u - \uI \vert^2_{1,\E} \lesssim \sigma^{2(1-\beta)n} \Vert \vert \mathbf x \vert^\beta \vert D^2 \u \vert \Vert^2_{0,\E} + \Vert f - \Pizpmd f \Vert^2_{0,\E}.
\]
where we recall that $\Pizpmd$ is the $L^2(\E)$ orthogonal projection from $\VnE$ into $\mathbb P_{\pE-2}(\E)$,
$\sigma$ is the grading factor discussed in assumption (\textbf{D3}) and $n+1$ is the number of layers.
\end{lem}
\begin{proof}
We consider $\uI$ defined as in \eqref{definition virtual approximant}; in particular, we fix $\uE$, the trace of $\uI$ on $\partial \E$ to be the piecewise affine interpolant of $\u$ at the vertices of $\E$.
From Lemma \ref{lemma virtual approximation far from sing}, we have:
\begin{equation} \label{formula lemma 5.6}
\vert \u - \uI \vert^2_{1,\E} \lesssim \Vert f- \Pizpmd f \Vert_{0,\E}^2 + \Vert \u - \uI -c_1\Vert^2_{\frac{1}{2},\partial \E},
\end{equation}
where $c_1$ is the average of $\u-\uI$ on $\partial \E$.

In order to get the claim, it suffices to bound the second term.
As in Lemma \ref{lem local discontinuous approximation near origin}, we consider the subtriangulation $\tautilden = \tautilden(\E)$ of $\E$ obtained by connecting all the vertices of $\E$ to $\mathbf 0$, see assumption (\textbf{D4}).
In particular, every triangle $\T \in \tautilden$ is star-shaped with respect to a ball of radius $\ge \widetilde \gamma h_\T$, where $\widetilde \gamma$ is a positive universal constant.
We define $\utilde _\E$ as the piecewise linear interpolant polynomials over the triangular subtriangulation, interpolating $\u$ at the vertices of $\T$, for every $\T \in \tautilden$.
Using \cite[Lemma 4.16]{SchwabpandhpFEM} and applying a Poincar\'e inequality, yield to:
\begin{equation} \label{estimate near sing uE}
\begin{split}
&\Vert \u - \uI -c_1 \Vert_{\frac{1}{2},\E}^2 \\& \lesssim \Vert \u - \utilde_\E -c_1 \Vert_{1,\E} ^2
\lesssim \sum_{\T \in \tautilden} \vert  \u - \utilde_\E \vert^2_{1,\T} \lesssim \sum_{\T \in \tautilden} h_\T^{2(1-\beta)} \Vert \vert \xbold\vert^\beta \vert D^2\u \vert \Vert^{2}_{0,\T}
\lesssim \sigma^{2(1-\beta)} \Vert \vert \mathbf x \vert^\beta \vert D^2 \u \vert \Vert^2_{0,\E}.	
\end{split}
\end{equation}
We stress that the third inequality in \eqref{estimate near sing uE} holds since $\utilde_\E |_{\T}$ is a linear polynomial and therefore $D^2 \utilde _\E=0$ on all $\T \in \tautilden$.
\end{proof}
We note that in Lemmata \ref{lemma virtual approximation far from sing} and \ref{lemma virtual approximation near sing} the error between $f$ and its $L^2$ projection can be bounded using Lemmata
\ref{lem local discontinuous approximation far from origin loading term} and \ref{lem local discontinuous approximation near origin loading term}.
We also point out that the hypotesis concerning the distribution of the local degrees of accuracy, i.e. the fact that $\pE =2$ if $\E\in L_0$, $\pE\approx 2$ if $\E \in L_1$, $\pe \approx \pE$ if $\e \subseteq \partial \E$,
are in accordance with the forthcoming definition \eqref{assumption mu} that we will introduce for the proof of Theorem \ref{theorem exponential convergence}.
Finally, we point out in Lemmata \ref{lemma virtual approximation far from sing} and \ref{lemma virtual approximation near sing} we introduced a function $\uI$ which is locally in $\VnE$ and globally continuous;
thus, $\uI$ is a function in the global Virtual Element Space $\Vn$ introduced in \eqref{global non uniform virtual space}.

\subsection{Exponential convergence} \label{subsection exponential convergence}
We set $\Omegaext=\cup_{n\in \mathbb N} \Omegaext_{n} = \Omegaext_{1}$, where the $\Omegaext_{n}$ are introduced in (\textbf{D4}).
We recall that we are assuming that $\mathbf 0 \notin \partial {\Omegaext}$.

We observe that our error analysis needs regularity on $f$ and subsequently on $\u$, the right-hand side and the solution of problem \eqref{continuous weak problem}, respectively. In particular, we will require:
\begin{equation} \label{regularity on u}
f \text{ can be extended to a function in } \mathcal B^0_\beta\left(\Omegaext \right), \quad \u \text{ can be extended to a function in } \mathcal B^2_\beta\left(\Omegaext \right).
\end{equation}
With a little abuse of notation we will call this two functions $f$ and $\u$.
Assuming $f \in  \mathcal B^0_\beta\left(\Omega \right)$ automatically implies that $\u$ is in $\mathcal B^2_\beta\left(\Omega \right)$; this follows from classical elliptic regularity theory, see Theorem \ref{theorem Babuska regularity}.
In the classical $hp$ Finite Element Method, this regularity leads to exponential convergence of the energy error, see \cite{SchwabpandhpFEM}.
In order to prove the same exponential convergence with $hp$ VEM, we need \eqref{regularity on u} since the approximation by means of polynomials on the polygons not abutting the singularity needs regularity of the target function on a square containing the polygon,
see Lemmata \ref{lem local discontinuous approximation far from origin} and \ref{lem local discontinuous approximation far from origin loading term}.


We recall the inflated domain $\Omegaext$ has been built in such a way that the singularity is never at the interior of $\Omegaext$, see assumption (\textbf{D3}).
We highlight also the fact that \eqref{regularity on u} can be easily generalized to the case of multiple singularities, see e.g. \cite{SchwabpandhpFEM}.

In order to obtain exponential convergence of the energy error in terms of the number of degrees of freedom, we will henceforth assume that the vector $\mathbf p$ of the degrees of accuracy associated with $\taun$ is given by:
\begin{equation} \label{assumption mu}
\mathbf p_{\E} = \begin{cases}
2& \text{if } \E\in L_0,\\
\max \left(2,\lceil \mu \cdot (j + 1)\rceil \right)& \text{if } \E \in L_j,\, j\ge 1,
\end{cases}
\end{equation}
where $\mu$ is a positive constant which will be determined in the proof of Theorem \ref{theorem exponential convergence}
and where $\lceil \cdot \rceil$ is the ceiling function. 
Note that choice \eqref{assumption mu} could be modified asking for $\pE=1$ if $\E \in L_0$; under this requirement in fact Lemmata
\ref{lem local discontinuous approximation near origin}, \ref{lem local discontinuous approximation near origin loading term} and \ref{lemma virtual approximation near sing} are still valid.
Nonetheless, we prefer to use \eqref{assumption mu} in order to avoid technical discussions on the construction of the right-hand side of the method and keep the simple representation \eqref{discrete loading term}.

It is clear from \eqref{assumption mu} that if $\E_1$ and $\E_2$ belong to the $j$-th and the $(j+1)$-th layers respectively, for some $j=1,\dots,n-1$, then $\p_{\E_1} \approx \p_{\E_2}$, independently on all the other discretization parameters.
Thus, owing to Section \ref{section stability}, we also have $\alpha(\E_1)\approx \alpha({\E_2})$, independently on all the other discretization parameters.
Besides, $\pe \approx \pE$ whenever $\e \subseteq \partial\E$.
\begin{thm} \label{theorem exponential convergence}
Let $\{\taun\}_n$ be a sequence of polygonal decomposition satisfying (\textbf{D1})-(\textbf{D4}). Let $\u$ and $\un$ be the solutions of problems \eqref{continuous weak problem} and \eqref{discrete weak problem} respectively;
let $f$ be the right-hand side of problem \eqref{continuous weak problem}. Let $N=N(n)=\dim (\Vn)$. Assume that $\u$ and $f$ satisfy \eqref{regularity on u}.
Then, there exists $\mu >0$ such that $\mathbf p$, defined in \eqref{assumption mu}, guarantees the following exponential convergence of the $H^1$ error in terms of the number of degrees of freedom:
\begin{equation} \label{exponential convergence in term of N dofs}
\Vert u-\un \Vert _{1,\Omega} \lesssim \exp (-b \sqrt[3]{N}),
\end{equation}
with $b$ a constant independent of the discretization parameters.
\end{thm}
\begin{proof}
It suffices to combine Lemma \ref{lemma local VEM decomposition}, the results of Section \ref{section stability}, Lemmata from \ref{lem local discontinuous approximation far from origin} to \ref{lemma virtual approximation near sing}
and to use the same arguments of \cite[Theorem 4.51]{SchwabpandhpFEM}, properly choosing the parameter $\mu$.

The basic idea behind the proof is that around the singularity, geometric mesh refinement are employed, since $\p$ approximation leads only to an algebraic decay of the error;
on the other hand, on polygons far from the singularity, it suffices to increase the degree of accuracy, since on such polygons both the loading term and the exact solution of \eqref{continuous weak problem}
are assumed to belong to the Babuska space $\mathcal B_\beta ^2(\Omegaext)$ defined in \eqref{countably normed space}
and therefore $\p$ approximation leads to exponential convergence of the local errors (see \cite[Theorem 5.6]{hpVEMbasic}).

Following \cite[Theorem 4.51]{SchwabpandhpFEM} and using Lemma \ref{lemma local VEM decomposition} yield:
\begin{equation} \label{copying Schwab}
\vert \u - \un \vert _{1,\Omega} \le \c \max_{\E' \in \taun}\alpha(\E') \sigma ^{2(1-\beta)(n+1)},
\end{equation}
where $\c$ is a constant independent of both the discretizations parameters and the number of layers. Applying \eqref{anaisotropy dependence on p}, we obtain:
\begin{equation} \label{bound every alpha}
\alpha(\E) \lesssim  \pE^2 \max_{\E' \in \taun}\p_{\E'}^5 \lesssim (n+1)^{7},\quad \forall \E \in \taun,
\end{equation}
where we recall $n+1$ denotes the number of layers. Plugging \eqref{bound every alpha} in \eqref{copying Schwab}, we get:
\[
\vert \u - \un \vert _{1,\Omega} \le \c (n+1)^{7} \sigma ^{2(1-\beta)(n+1)}.
\]
We infer:
\[
\vert \u - \un \vert_{1,\Omega} \lesssim \exp (-b (n+1)), \quad \text{for some }b>0.
\]

Now, we prove that $N \lesssim (n+1)^3$.
In order to see this, we proceed as follows. In each layer there exists a fixed maximum number of layers; this follows from the geometric assumptions (\textbf{D1}) and (\textbf{D3}),
applying for instance the arguments in \cite[Section 4]{hmps_harmonicpolynomialsapproximationandTrefftzhpdgFEM}.
Using geometric assumption (\textbf{D2}) (which guarantees a maximum number of edges per each element), the definition of the local virtual space \eqref{local non uniform virtual space} and the distribution of the local degrees of accuracy \eqref{assumption mu},
it is straightforward to note that $\forall \E \in \taun$ the dimension of each local space $\VnE$ is proportional to $\pE^2$, with $\pE^2 \approx \ell^2$ for $\E \in L_{\ell}$.

Recalling again \eqref{assumption mu}, we can now compute a bound for the dimension of the local space, viz. the number of the degrees of freedom:
\[
N \lesssim \sum_{\el=0}^L \el^2 \le L \max_{\el=0}^L \el^2 = L^3,
\]
where we stress that we are using that in each layer there is a fixed maximum number of elements.
The result follows from Poincar\'e inequality.
\end{proof}

\section{Numerical results} \label{section numerical results}
We show here numerical experiments validating Theorem \ref{theorem exponential convergence}.
Let $u$, the solution of \eqref{continuous weak problem}, given by the classical benchmark
\begin{equation} \label{L shaped benchmark}
u(r,\theta)= r^{\frac{2}{3}} \sin \left(\frac{2}{3}\left( \theta +\frac{\pi}{2}\right) \right),
\end{equation}
on the $L$-shaped domain:
\begin{equation} \label{L shaped domain}
\Omega= [-1,1]^2 \setminus [-1,0]^2.
\end{equation}
\subsection{Tests on different meshes} \label{subsection tests on different meshes}
We consider sequences of the meshes depicted in Figure \ref{figure possible meshes on Lshaped}
and we consider two different choices for the degree of accuracy distribution $\mathbf{p}$.
As a first selection, we pick on all the elements a constant local degree of accuracy which is equal to the number of layers, i.e. $\mathbf p = (n+1,n+1,\dots,n+1)$.
As a second selection, we pick $\mathbf p_{\E}$ as in \eqref{assumption mu}, with $\mu=1$, $\mu$ being the parameter introduced for the construction of the vector of the degrees of accuracy.
In Figures \ref{figure energy error sigma 05}, \ref{figure energy error sigma sqrt2m1} and \ref{figure energy error sigma sqrt2m1s}, the numerical results are shown.

On the $y$-axis, we plot a log scale of the relative energy error between $u$, defined in \eqref{L shaped benchmark},
and the energy projection $\Pinabla_{\pE}$, defined in \eqref{pi nabla 1}, \eqref{pi nabla 2}, of the solution $\un$ of the discrete problem \eqref{discrete weak problem}, i.e.
\begin{equation} \label{computed VEM error}
\vert \u - \Pinabla_{\mathbf p} \un \vert_{1,n,\Omega} := \sqrt{\sum_{\E \in \taun} \vert \u - \Pinabla_{\pE} \un \vert_{1,\E}^2},
\end{equation}
On the other hand, in the $x$-axis we plot the cubic root of the number of the degrees of freedom of the relative virtual space.
The reason for choice \eqref{computed VEM error} is that it is not possile to compute the \emph{true} energy error since virtual functions are not known explicitely.

We consider the behaviour of the error with three different $\sigma$, grading parameter, namely $\sigma=\frac{1}{2}$, $\sqrt 2 -1$, $(\sqrt 2 -1)^2$ and we compare the three types of meshes.
\begin{figure}  [h]
\centering
\subfigure {\includegraphics [angle=0, width=0.49\textwidth]{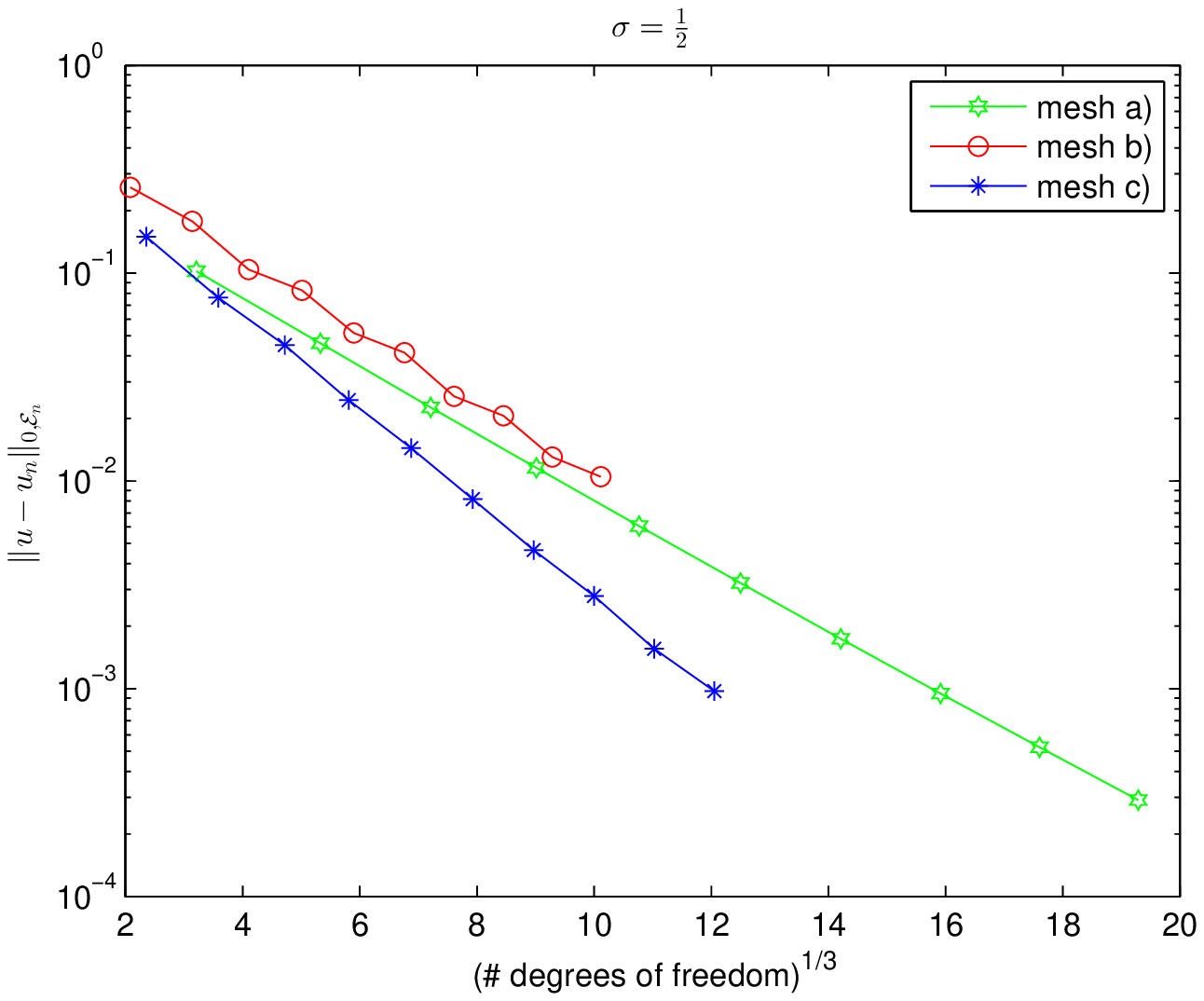}}
\subfigure {\includegraphics [angle=0, width=0.49\textwidth]{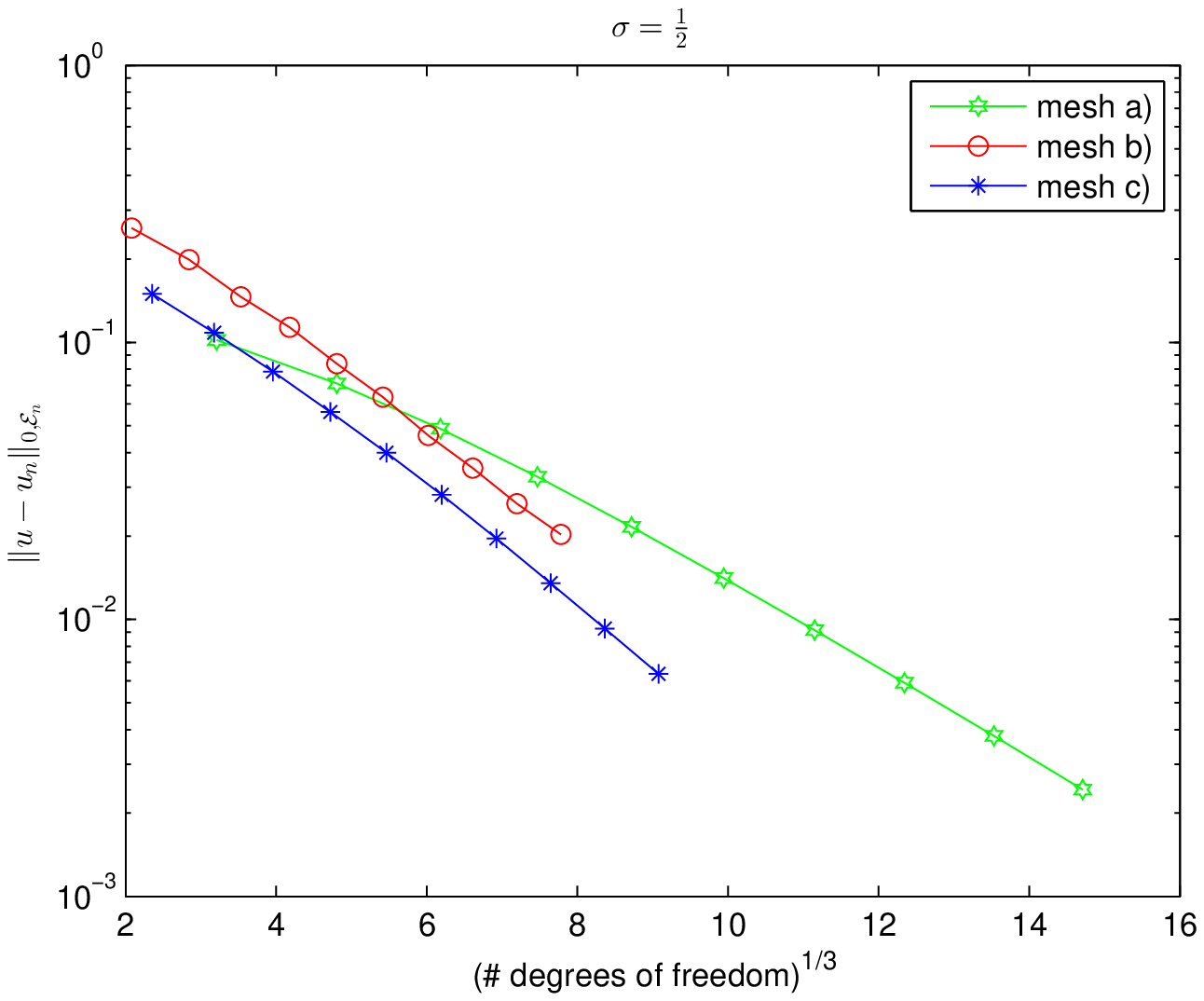}}
\caption{Error $\vert \u - \Pinabla_{\mathbf p} \un \vert_{1,n,\Omega}$ for the meshes in Figure \ref{figure possible meshes on Lshaped}, $\sigma=\frac{1}{2}$.
Left: the degree of accuracy is uniform and equal to the number of layers. Right: the degree of accuracy is varying over the mesh layers, $\mu=1$ in \eqref{assumption mu}.} \label{figure energy error sigma 05}
\end{figure}

\begin{figure}  [h]
\centering
\subfigure {\includegraphics [angle=0, width=0.49\textwidth]{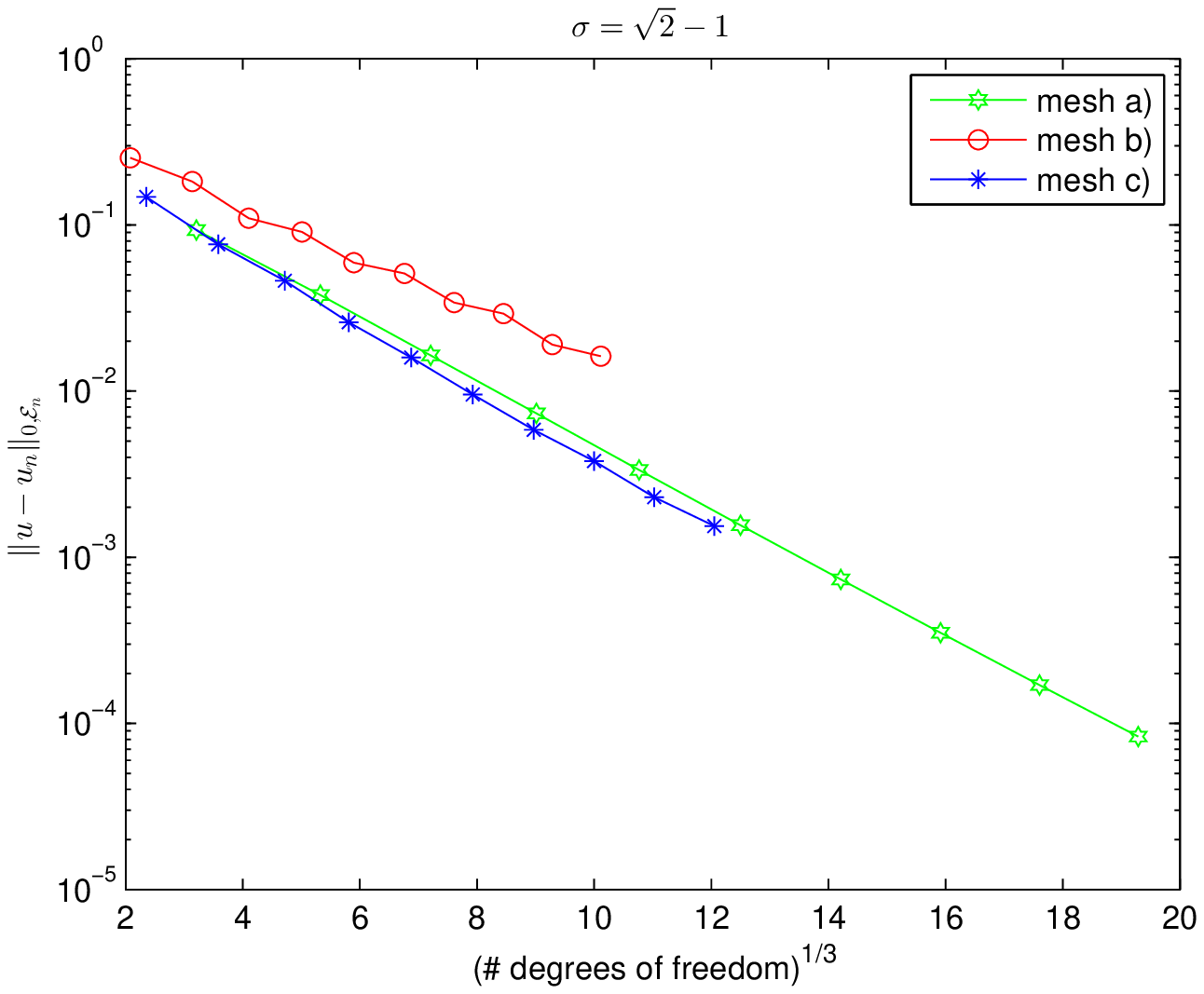}}
\subfigure {\includegraphics [angle=0, width=0.49\textwidth]{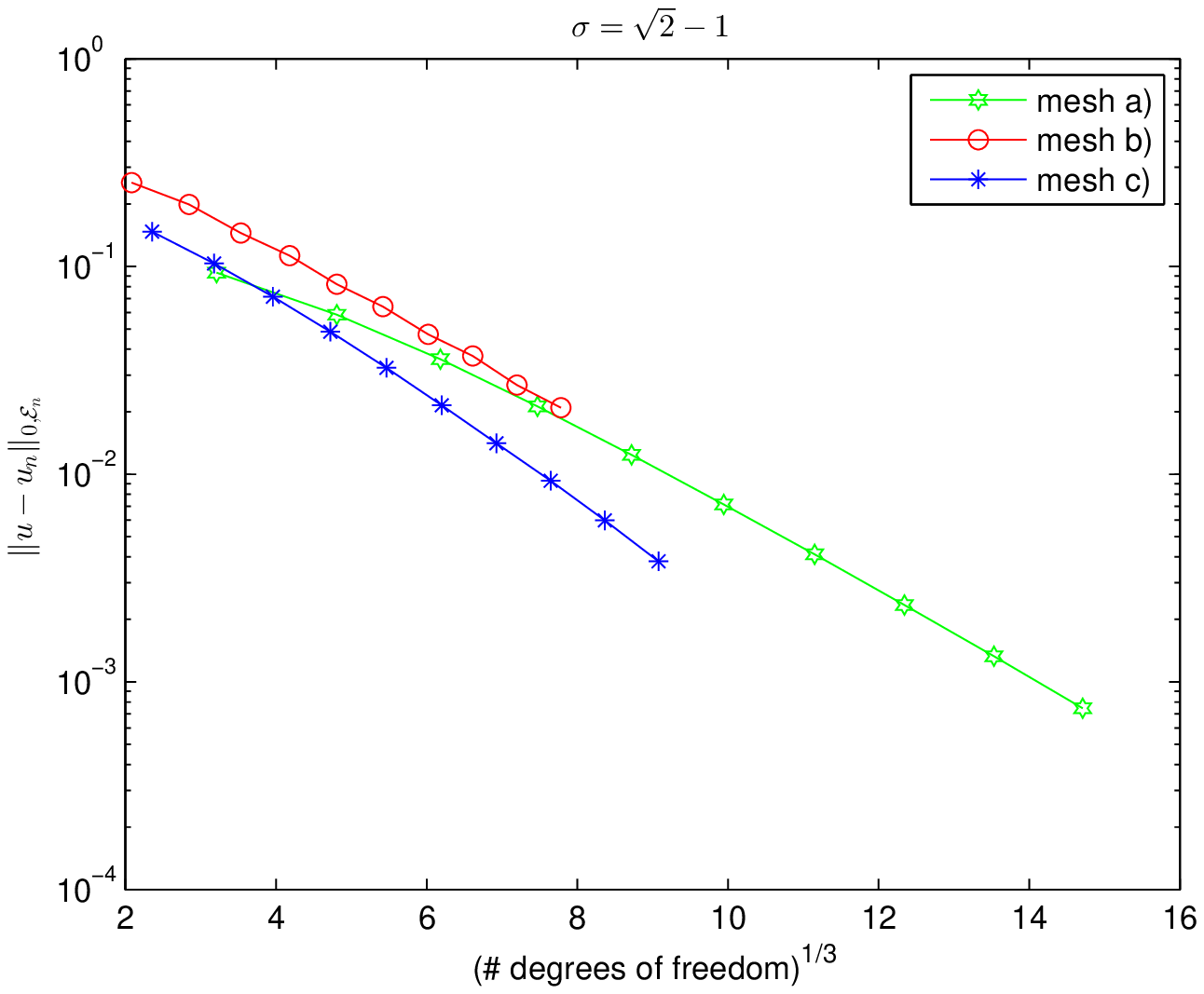}}
\caption{Error $\vert \u - \Pinabla_{\mathbf p} \un \vert_{1,n,\Omega}$ for the meshes in Figure \ref{figure possible meshes on Lshaped}, $\sigma=\sqrt 2 -1$.
Left: the degree of accuracy is uniform and equal to the number of layers. Right: the degree of accuracy is varying over the mesh layers, $\mu=1$ in \eqref{assumption mu}.} \label{figure energy error sigma sqrt2m1}
\end{figure}

\begin{figure}  [h]
\centering
\subfigure {\includegraphics [angle=0, width=0.49\textwidth]{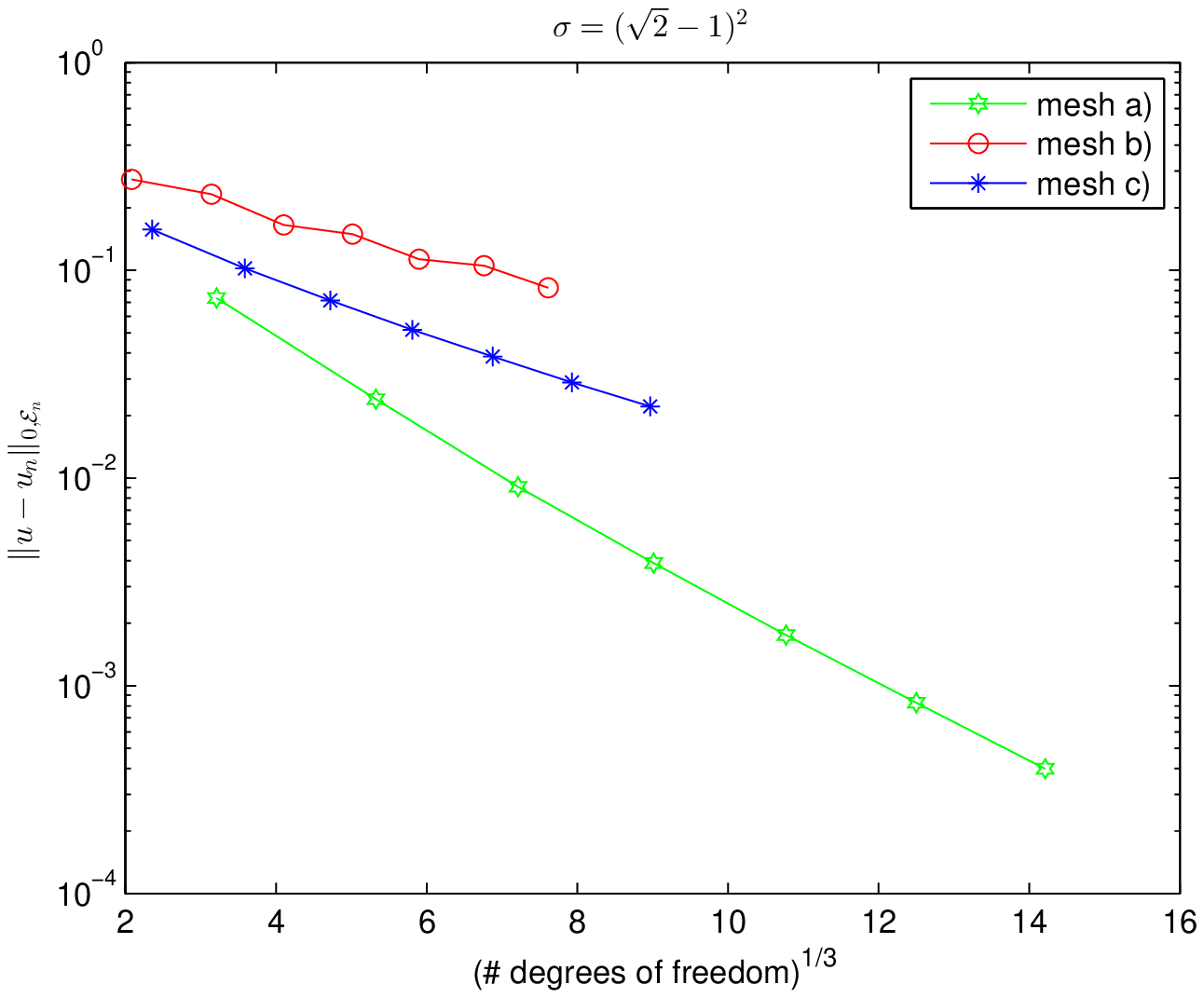}}
\subfigure {\includegraphics [angle=0, width=0.49\textwidth]{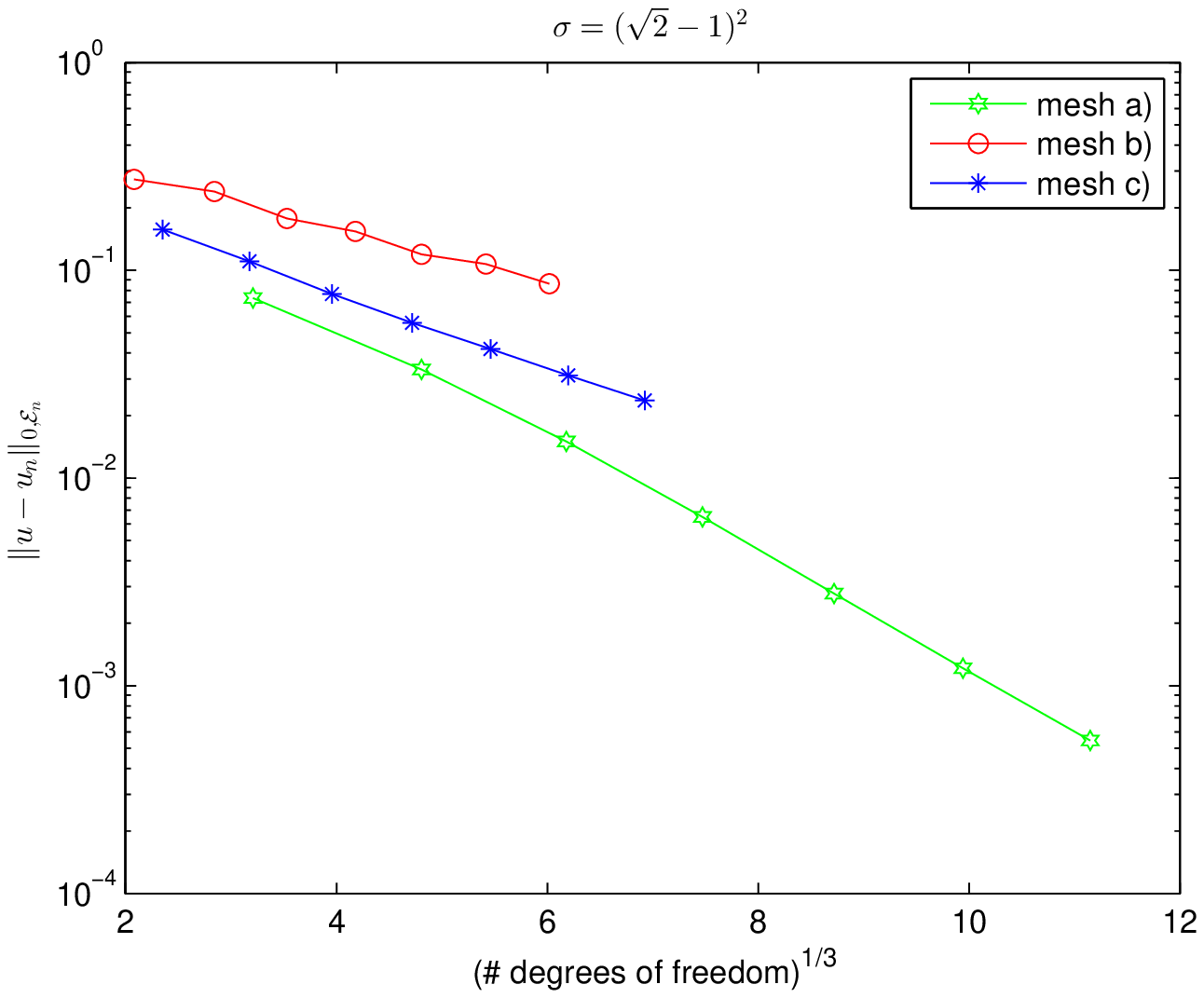}}
\caption{Error $\vert \u - \Pinabla_{\mathbf p} \un \vert_{1,n,\Omega}$ for the meshes in Figure \ref{figure possible meshes on Lshaped}, $\sigma=(\sqrt 2 -1)^2$.
Left: the degree of accuracy is uniform and equal to the number of layers. Right: the degree of accuracy is varying over the mesh layers, $\mu=1$ in \eqref{assumption mu}.} \label{figure energy error sigma sqrt2m1s}
\end{figure}

As mentioned previously, the sequence of meshes in Figure \ref{figure possible meshes on Lshaped} (center) does not satifsy assumptions (\textbf{D1}) and (\textbf{D4}).
Nevertheless, the expected exponential convergence rate is attained in all cases and for all geometric parameters $\sigma$.

\subsection{A comparison between $hp$ FEM and $hp$ VEM} \label{subsection a comparison between hp FEM and hp VEM}
We want now to show a comparison between the performances of $hp$ (quadrilateral) FEM and $hp$ VEM.
We stress that an analogous of Theorem \ref{theorem exponential convergence} holds for $hp$ FEM, see e.g. \cite{SchwabpandhpFEM}.
We consider again the benchmark with known solution \eqref{L shaped benchmark} and we consider the quadrilateral mesh in Figure \ref{figure quadmesh}. In the following we will denote such mesh with d)
whereas we denote with a), b) and c) the meshes depicted in Figure \ref{figure possible meshes on Lshaped} (left), (centre) and (right) respectively.
\begin{figure}  [h]
\centering
\includegraphics [angle=0, width=0.31\textwidth]{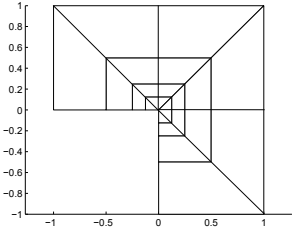}
\caption{Mesh used for the $hp$ FEM.} \label{figure quadmesh}
\end{figure}
In particular, we pick in both cases $\mathbf p_{\E}$ as in \eqref{assumption mu} $\forall \E \in \taun$, with $\mu=1$.
We discuss the case of sequences of meshes with grading parameter $\sigma$ equal to $\frac{1}{2}$, $\sqrt 2 -1$ and $(\sqrt 2 -1)^2$.

Since we cannot compute the \emph{true} energy error with the Virtual Element Method (it is \emph{not} computable since functions in the virtual space are not known explicitely),
in order to compare the two methods, we investigate the $L^2$ error on the skeleton $\mathcal E_n$ (it is computable in all cases a),$\dots$,d), since also the virtual functions are polynomials on $\mathcal E_n$), i.e.
\[
\Vert \u -\un \Vert _{0,\mathcal E _n}.
\]
The results are shown in Figure \ref{figure hp FEM vs hp VEM L2 skeleton}.
\begin{figure}  [h]
\centering
\subfigure {\includegraphics [angle=0, width=0.31\textwidth]{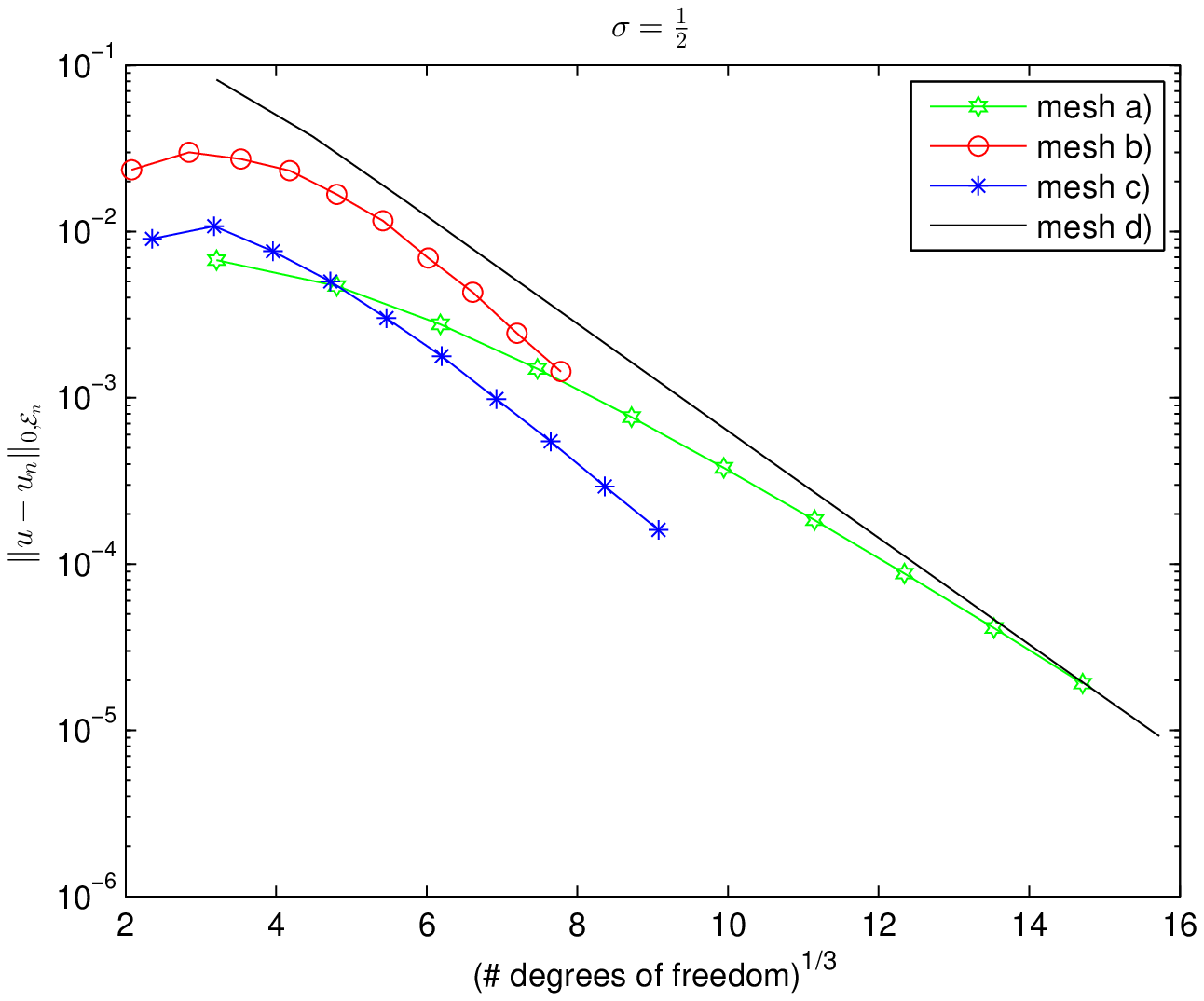}}
\subfigure {\includegraphics [angle=0, width=0.31\textwidth]{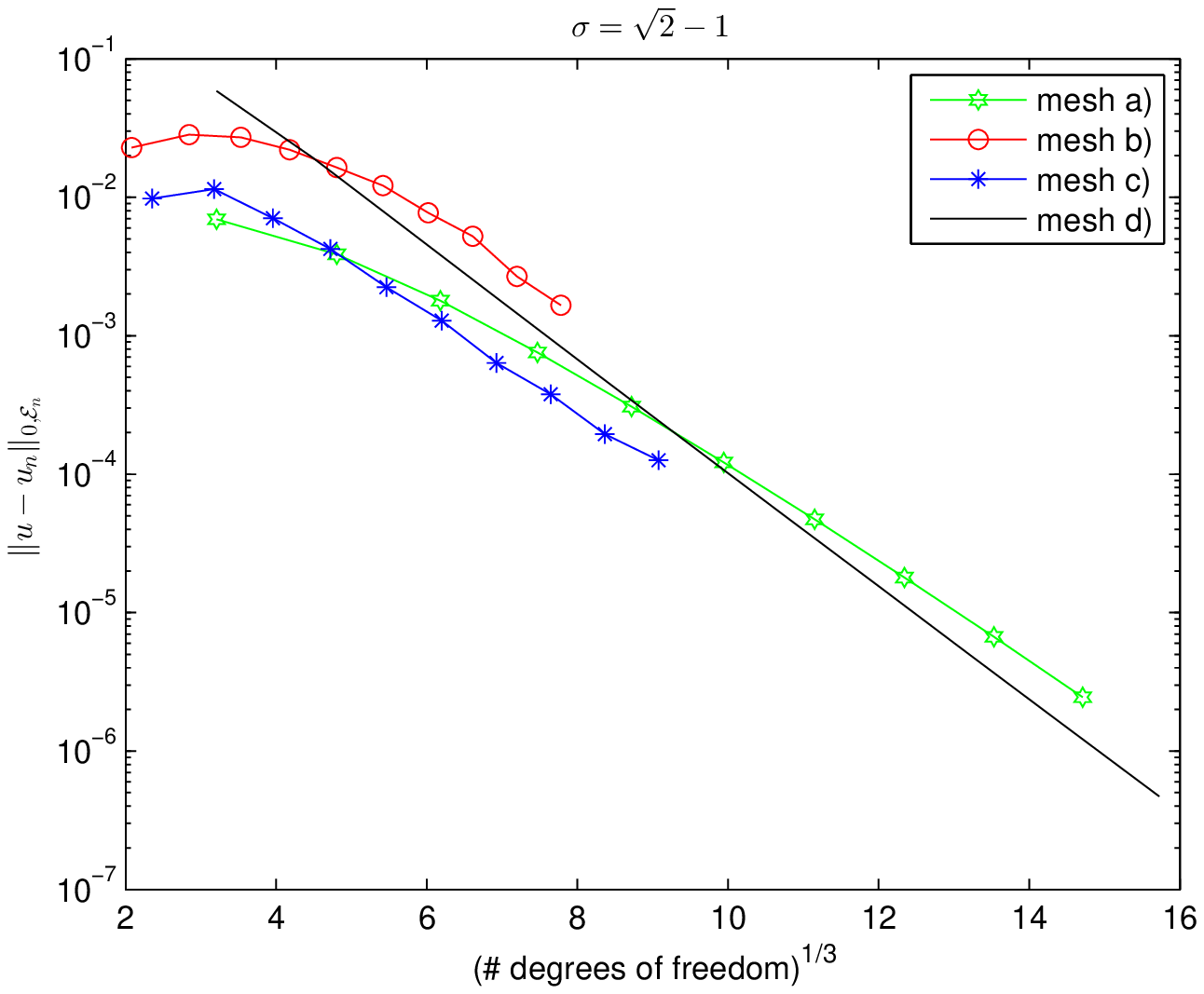}}
\subfigure {\includegraphics [angle=0, width=0.31\textwidth]{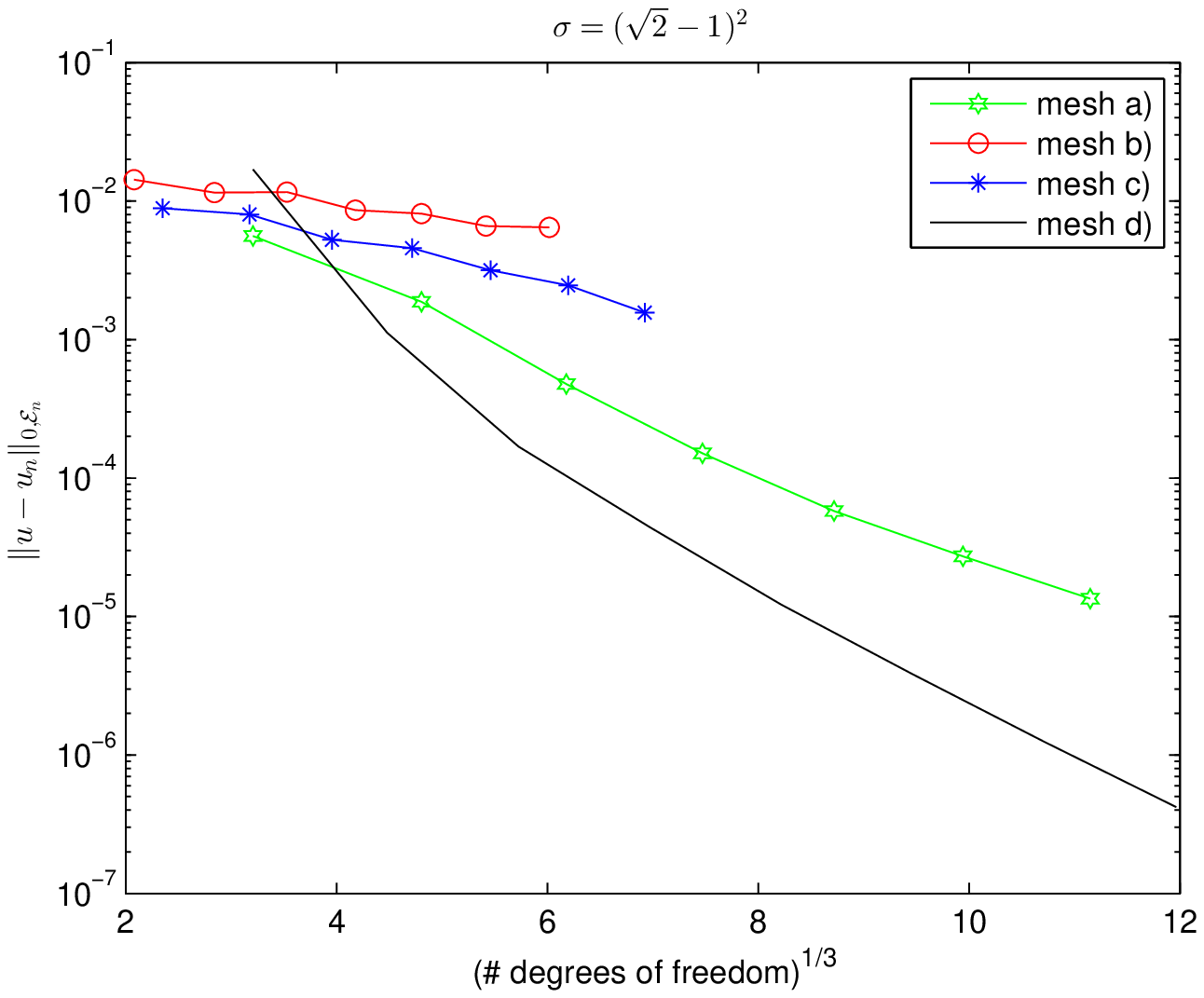}}
\caption{$hp$ FEM vs $hp$ FEM. $L^2$ error on the skeleton $\Vert \u - \un \Vert_{0,\mathcal E _n}$ for different sequence of meshes and different parameters $\sigma$.
Left: $\sigma=\frac{1}{2}$,  middle: $\sigma=\sqrt 2-1$, right: $\sigma=(\sqrt 2-1)^2$, linearly varying over the mesh layers degrees of accuracy ($\mu=1$ in \eqref{assumption mu}).} \label{figure hp FEM vs hp VEM L2 skeleton}
\end{figure}
It is possible to see that there is not a preferential choice; for instance, $hp$ VEM performs better than $hp$ FEM when $\sigma = \frac{1}{2}$, they perform almost the same when $\sigma = \sqrt 2-1$,
performs much worse when $\sigma = (\sqrt 2 -1)^2$.

In this sense, we can say that the two methods are comparable; nonetheless, the Virtual Element Methods leads to a huge flexibility in the choice of the domain meshing,
thus implying the possibility of constructing spaces with a minor number of degrees of freedom.\\

As a final remark, we observe that we could perform the same analysis in Section \ref{section exponential convergence for corner singularity on geometric meshes}
by modifying the definition of the local Virtual Spaces \eqref{local non uniform virtual space} into the \emph{serendipity} local Virtual Spaces
introduced in \cite{serendipityVEM}; this would additionally decrease the number of the degrees of freedom of the space, leading as a final output of the method to very small-sized linear systems.

\begin{appendices}
\section{}
 \label{section appendix A}
In this first appendix, we show the polynomial $hp$ inverse estimate \eqref{inverse estimate polynomials negative norm} and some technical background results.
We will use the properties of some particular Jacobi polynomials $\{ J_n^{\alpha,\beta} (x)\}_{n=0}^ {\infty}$, $\alpha,\,\beta\ge 0$, namely Legendre and shifted-ultraspherical polynomials. Henceforth, we denote with $\Ihat=[-1,1]$ the reference interval.

The following result was firstly presented in \cite{bernardi2001error}.
We stress that Lemma \ref{lemma 1D Jacobi expansion} holds for more general weights (i.e $-1 < \alpha \le \beta$), nonetheless
we discuss here only the case $0 \le \alpha \le \beta$ which is sufficient for our purpose.

\begin{lem} \label{lemma 1D Jacobi expansion}
Let $0 \le \alpha < \beta$. Then, $\forall \q \in \mathbb P_{\p}(\Ihat)$ with $\p\in \mathbb N$, it holds:
\begin{equation} \label{estimate weighted L2 1D}
\int_{\Ihat} (1-x^2)^{\alpha} \q(x)^2\dx  \le \c (\p+1)^{2(\beta-\alpha)} \int _{\Ihat} (1-x^2)^{\beta} \q(x)^2\dx,
\end{equation}
where $c$ is a positive contant depending on $\alpha$ and $\beta$, but not on $p$.
\end{lem}
\begin{proof}

We split the proof into three parts. The first two are results dealing with \ultraspherical  polynomials properties, while in the last one we show the assertion.

For the properties of \ultraspherical polynomials we refer to \cite{SchwabpandhpFEM, GhizzettiOssiciniquadratureformulae, shentangwangspectralmethods, bernardi1997maday, szego_orthogonal, AbramowitzStegun_handbook}.
We recall various facts that we will use throughout the proof about these polynomials.
\begin{itemize}
\item[*] The $n$-th \ultraspherical polynomial $\Jalphan$, $\alpha \ge 0$, is the $n$-th Jacobi polynomial $J_n^{\alpha,\beta}$ with $\alpha=\beta\ge 0$;
the sequence $\{\Jalphan\}_{n=0}^{+\infty}$ forms an orthogonal (but not normal) basis for the weighted Lebesgue space:
\[
L_{\rhoalpha}(\Ihat) := \left\{ u \text{ Lebesgue-measurable on } \Ihat \mid \int_{\Ihat} \rhoalpha(x) \vert u(x)\vert ^2 \dx < +\infty   \right\},
\]
where $\rhoalpha$ is the weighted 1D function $\rhoalpha (x)=(1-x^2)^{\alpha}$.
\item[*] Each $\Jalphan$ is the $n$-th eigenfunction of the Sturm-Liouville problem:
\begin{equation} \label{ultraspherical equation}
(\rho_{\alpha+1}(x) \Jalphan(x)')' + n (n+2\alpha+1) \rhoalpha(x)\Jalphan(x)=0,\quad \forall x\in \Ihat,
\end{equation}
with appropriate Dirichlet conditions at the endpoints of $\Ihat$:
\[
\Jalphan (\pm 1) = (-1)^n {{n+\alpha} \choose {n}}.
\]
\item[*] The following orthogonality relation holds for $n \ge 1$, see e.g. \cite[formula (4.3.3)]{szego_orthogonal}:
\begin{equation} \label{orthogonality Jacobi}
\int_{-1}^1 \Jalphan (x) J_{m}^{\alpha}(x) \rhoalpha(x)\dx = \delta_{n,m} \frac{2^{2\alpha+1}\Gamma(n+\alpha+1)^2}{(2n+2\alpha+1)n!\Gamma(n+2\alpha+1)},
\end{equation}
where $\delta_{n,m}$ is the Kronecker delta and $\Gamma$ is the Gamma function.
\item[*] The following asymptotic behaviour of the Gamma function holds:
\begin{equation} \label{Gamma behaviour}
\Gamma(t)= \sqrt{2\pi} e^{-t} t^{t-\frac{1}{2}} \left( 1+ O(1/t)    \right),\quad \text{for }t \rightarrow +\infty.
\end{equation}
\item[*] The following relation between \ultraspherical polynomials and their derivatives holds, see \cite[Theorem 19.3]{bernardi1997maday}:
\begin{equation} \label{link Jacobi derivatives}
(2n+2\alpha+1) \Jalphan(x)= \frac{n+2\alpha+1}{n+\alpha+1} J_{n+1}^\alpha (x)'  - \frac{n+\alpha}{n+2\alpha} J_{n-1}^\alpha (x)' .
\end{equation}
\end{itemize}
We start the proof of the theorem.
As a last comment, the details of steps 1 and 2 are carried out here (although they are already known),while step 3 is a detailed version of \cite[Theorem 19.3]{bernardi2001error}.\\
\textbf{1st STEP}
We want to show here:
\begin{equation} \label{first fact}
c_1 n \le \int_{\Ihat} (\Jalphan(x)')^2 \rho_{\alpha+1}(x)\dx \le c_2 n,
\end{equation}
where $c_1$ and $c_2$ are two positive constants independent on $n$, but depending on $\alpha$.

For the purpose,  we observe that \eqref{ultraspherical equation} and an integration by parts imply:
\begin{equation} \label{integrated version ultraspherical equation}
\int_{\Ihat} (\Jalphan (x)')^2 \rho_{\alpha+1}(x)\dx = n(n+2\alpha+1) \int_{\Ihat} (\Jalphan (x))^2 \rho_{\alpha}(x)\dx.
\end{equation}
We stress that one could also show \eqref{integrated version ultraspherical equation} by combining \eqref{orthogonality Jacobi} with \cite[formula (4.21.7)]{szego_orthogonal}.

Next, we estimate $\int_{\Ihat} (\Jalphan (x))^2 \rho_{\alpha}(x)\dx$.
We set for the purpose:
\[
2^{-2\alpha-1} (2n + 2\alpha + 1) \int_{\Ihat} \Jalphan (x)^2 \rho_{\alpha}(x) \dx =: g(n,\alpha) = \frac{\Gamma (n+\alpha+1)^2}{\Gamma (n+1)  \Gamma (n + 2\alpha+1)}.
\]
Function $g(n,\alpha)$ is increasing in $n$ for every fixed $\alpha > -1$. Besides, $\lim _{n \rightarrow +\infty} g(n,\alpha) =1$.
Thus:
\[
g(1,\alpha) \le 2^{-2\alpha - 1 }(2n+2\alpha + 1) \int_{\Ihat} \Jalphan (x)^2 \rho_{\alpha}(x) \dx \le 1,
\]
which implies:
\begin{equation} \label{conti qua sopra}
\frac{\widetilde c_1}{n} \le \int_{\Ihat} \Jalphan(x)^2 \rho_{\alpha}(x)\dx \le \frac{\widetilde c_2}{n},\quad \quad n \ge 1,
\end{equation}
where $\widetilde c_1$ and $\widetilde c_2$  are two positive constants independent on $n$, but dependent on $\alpha$.
Using that $\alpha \ge$, an explicit representation for the two constants $\widetilde c_1$ and $\widetilde c_2$ is given by:
\begin{equation} \label{choice ctilde}
\widetilde c_1 = \frac{2^{2\alpha+1}}{2\alpha +3} \frac{\Gamma (2+\alpha)^2}{\Gamma (2) \Gamma (2+2\alpha)};\quad\quad\quad\quad \widetilde c_2 = 2^{2\alpha}.
\end{equation}
The claim, i.e. \eqref{first fact}, follows combining \eqref{conti qua sopra} and \eqref{integrated version ultraspherical equation}.
An explicit choice for the two constants $c_1$ and $c_2$ in \eqref{integrated version ultraspherical equation} is given by:
\begin{equation} \label{choice c12}
c_1 = \widetilde c_1 = \frac{2^{2\alpha+1}}{2\alpha +3} \frac{\Gamma (2+\alpha)^2}{\Gamma (2) \Gamma (2+2\alpha)};\quad\quad\quad\quad  c_2 = (2\alpha+2) 2^ {2\alpha}.
\end{equation}
\textbf{2nd STEP}
We show secondly the following bound:
\begin{equation} \label{second fact}
\int_{\Ihat} (\Jalphan(x)')^2 \rhoalpha(x) \dx \le b n^2,
\end{equation}
for some positive constant $b$ independent on $n$ but depending on $\alpha$.
We will prove this fact by induction. The cases $n=1$, $2$ are obvious since we have positive left and right-hand sides. Assume then that \eqref{second fact} holds up to $n$ and we show the inequality for $n+1$.

We observe that the following inequalities involving the coefficients in \eqref{link Jacobi derivatives} are valid (we recall that $\alpha,\beta \ge 0$):
\begin{equation} \label{bounds on coefficients}
2n \le 2n + 2\alpha + 1 \le (2\alpha+3) n, \quad \quad 1\le \frac{n+2\alpha + 1}{n + \alpha + 1} \le 2,\quad \quad \frac{1}{2} \le \frac{n+\alpha}{n+2\alpha} \le 1.
\end{equation}
Then, using \eqref{orthogonality Jacobi}, \eqref{link Jacobi derivatives} and \eqref{bounds on coefficients}, we get:
\begin{equation} \label{long estimate}
\begin{split}
&\int_{\Ihat} (J_{n+1}^\alpha(x)')^2 \rho_{\alpha} (x) \dx \le \int_{\Ihat} \left( \left( \frac{n+2\alpha+1}{n+\alpha+1} \right)  J_{n+1}^\alpha(x)' \right)^2 \rho_\alpha(x) \dx \\
& = \int_{\Ihat} \left( (2n+2\alpha+1) \Jalphan (x)  \right)^2\rho_\alpha(x) \dx + \int _{\Ihat} \left(  \frac{n+\alpha}{n+2\alpha} J_{n-1}^{\alpha}(x)'  \right) ^2  \rho_{\alpha}(x) \dx\\
& \le \underbrace{(2\alpha+3) ^2}_{=:c_\alpha} n^2 \int_{\Ihat} \Jalphan(x)^2\rho_\alpha(x) \dx + \int_{\Ihat} \left( J_{n-1}^\alpha(x)'  \right)^2 \rho_{\alpha} \dx.
\end{split}
\end{equation}
We apply \eqref{conti qua sopra} and the induction hypotesis to the first and second term in the right-hand side of \eqref{long estimate} respectively, obtaining:
\[
\int_{\Ihat} (J_{n+1}^{\alpha} (x)')^2 \rho_{\alpha}(x) \dx \le c_\alpha \widetilde c_2 n + b (n-1)^2 =: \widetilde c n + b(n-1)^2,
\]
where $\widetilde c_2$ is defined in \eqref{choice ctilde}.
Taking $b$ large enough, for instance $b \ge \frac{\widetilde c}{4}$, the following holds:
\begin{equation} \label{stupid bound}
\widetilde c n+  b(n-1)^2 \le b (n+1)^2.
\end{equation}
We point out that we have to take power 2 in the right-hand side of \eqref{second fact} because with smaller powers \eqref{stupid bound} would not be true.\\
\textbf{3rd STEP}
We show \eqref{estimate weighted L2 1D}.
Let $\q \in \mathbb P_p(\Ihat)$. We expand it into a derivated Jacobi sum:
\[
\q(x)=\sum_{n=1}^{\p+1} a_n \Jalphan(x)'.
\]
Then, noting from \eqref{ultraspherical equation} and an integration by parts implies that the \ultraspherical polynomials $J_{n}^{\alpha+1}$ are $L^2$ orthogonal with respect to the weight $\rho_{\alpha+1}$, we have:
\begin{equation} \label{step 1 expanded}
\int_{\Ihat} \q(x)^2 \rho_{\alpha+1}(x) \dx  = \sum_{n=1}^{\p+1} a_n^2 \int _{\Ihat} (\Jalphan(x)')^2 \rho_{\alpha+1} (x)\dx \ge c_1 \sum_{n=1}^{\p+1} a_n^2 n,
\end{equation}
where $c_1$ is defined in \eqref{choice c12}.
On the other hand, \eqref{second fact} implies:
\begin{equation} \label{step 2 expanded}
\int_{\Ihat} \q(x)^2 \rhoalpha(x)\dx \le  \left( \sum_{n=1}^{\p+1} |a_n| \left( \int_{\Ihat} (\Jalphan(x)')^2 \rhoalpha(x) \right)^{\frac{1}{2}}  \right)^2 \le b \left( \sum_{n=1}^{\p+1} |a_n|n  \right)^2,
\end{equation}
where $b$ is introduced in \eqref{second fact}.
Combining \eqref{step 1 expanded} with \eqref{step 2 expanded} and using a Cauchy-Schwarz inequality for sequences, lead to:
\begin{equation} \label{result differenza 1}
\int_{\Ihat} \q(x)^2 \rhoalpha(x) \dx \le b \left(\sum_{n=1}^{\p+1} a_n^2 n\right) \left(\sum_{n=1}^{\p+1}n\right)\le b c_1^{-1} (\p+1)^2 \int _{\Ihat} \q(x)^2 \rho_{\alpha+1}(x)\dx.
\end{equation}
This is in fact the thesis when $\beta-\alpha=1$. The case $\beta-\alpha\in \mathbb N$ is straightforward; it suffices in fact to iterate enough time the above computations.

Assume now $\alpha\in (\beta-1,\beta)$. Then:
\[
\alpha= \frac{\beta-1}{r}+\frac{\beta}{s},\quad\text{ with } \frac{1}{r}+ \frac{1}{s}=1\quad \text{and }\frac{1}{r}=\beta-\alpha<1.
\]
In order to conclude, using an Holder inequality and \eqref{result differenza 1}:
\[
\begin{split}
\int_{\Ihat} \q(x)^2 \rhoalpha (x)\dx 	&=\int _{\Ihat} \q(x)^{\frac{2}{r}} \rho_{\frac{\beta-1}{r}}(x) \q(x)^{\frac{2}{s}}\rho_{\frac{\beta}{s}}(x) \dx \le \left( \int_{\Ihat} \q(x)^2 \rho_{\beta-1}(x)\dx\right)^{\frac{1}{r}} \left(\int_{\Ihat} \q(x)^2\rho_{\beta}(x)\dx \right)^{\frac{1}{s}}\\
					&\le (b c_1^{-1})^{\frac{1}{r}} (\p+1)^{\frac{2}{r}}   \left( \int _{\Ihat} \q(x)^2\rho_{\beta}(x)\dx \right)^{\frac{1}{r}+\frac{1}{s}} = (b c_1^{-1})^{\frac{1}{r}} (\p+1)^{2(\beta-\alpha)} \int_{\Ihat} \q(x)^2 \rho_{\beta}(x)\dx.
\end{split}
\]

We point out that in order to prove the case $\alpha\in (\beta-1,\beta)$ one could also use interpolation theory, see \cite{triebel,tartar}.
Nonetheless, we believe that a direct computation is easily readable.
\end{proof}
The following lemma is a quasi-one dimensional result on trapezoids. The idea is pretty similar to that in \cite[Lemma D.3]{melenk2003hp}, although our result employs a different class of weight functions.
We also point out that Lemma \ref{lemma trapezoid} can be generalized to the case $-1 < \alpha < \beta$, see \cite{melenk2003hp}.
\begin{lem} \label{lemma trapezoid}
Let $d\in (0,1)$. Let $a$, $b\in \mathbb R$ such that $-1+ad< 1+bd$. We set the $(a,b,d)$-trapezoid as:
\[
D(a,b,d)=D=\{(x,y)\in \mathbb R^2 \mid y\in [0,d],\; -1+ay\le x \le 1+by\}.
\]
We associate to each $y^*$ the segment:
\begin{equation} \label{interval star}
I(y^*)=I^*=[-1+a y^*, 1+ b y^*].
\end{equation}
For every $\Phi \in \mathcal C^0(\overline D)$ such that:
\begin{equation} \label{requirements on cubic bubble trapezoidal}
\Phi(\cdot, y^*) \in \mathbb P_3(I^*) \text{ is concave};\; \Phi (x,y^*) \ge 0,\; \forall x \in I^* ;\; \Phi=0 \text{ only at the endpoints of } I^*,\;\forall y^* \in [0,d],\;
\end{equation}
the following quasi-one dimensional $p$ inverse estimate holds:
\[
\int_{D} \Phi^{\alpha}(x,y) \q(x,y)^2 \dx\dy\le \c (\p+1)^{2(\beta-\alpha)} \int _{D} \Phi^{\beta}(x,y)\q(x,y)^2\dx\dy, \quad \forall \q \in \mathbb P_p(D),
\]
where $\c$ is a positive constant depending only on $\alpha$ and $\beta$, but not on $p$, and where $\beta>\alpha\ge 0$.
\end{lem}
\begin{proof}
\textbf{1st STEP}$\,$ 
Let $\psi(x)=(1-x^2)$ be the 1D bubble function associated to the reference interval $\Ihat:=[-1,1]$.
Given $y^*\in[0,d]$, we set $F$ the affine function mapping $\Ihat$ in $I^*$.
Let $\psi ^*(x) = \psi(F^{-1}(x)): I^* \rightarrow \mathbb R$.

Then, there exist two positive constants $c_1$ and $c_2$ depending only on $a$, $b$ and $y^*$ such that
\begin{equation} \label{puntual equivalence with dependence}
c_1 \psi^*(x) \le \Phi (x,y^*) \le c_2 \psi^*(x),\quad \forall x\in I^*.
\end{equation}
This follows from the fact that both $\psi^*(\cdot)$ and $\Phi(\cdot, y^*)$ are two positive quadratic/cubic concave polynomials annihilating only at the endpoints of the segment for all $y^* \in [0,d]$, see \eqref{requirements on cubic bubble trapezoidal}.

Since $\Phi$ is by hypotesis a continuous function in $y^*$, then $c_1$ and $c_2$ depend continuously on $y^*$.
Having that $y^*$ lives in the compact set $[0,d]$, then $c_1$ and $c_2$ attain maximum and minimum respectively.
Further, such extremal points are strictly positive due to the positiveness of $c_1$ and $c_2$ seen as functions of $y^*$, see \eqref{requirements on cubic bubble trapezoidal}.
Therefore, we can write:
\begin{equation} \label{puntual equivalence without dependence}
\overline{c}_1 \psi^*(x) \le \Phi(x,y^*) \le \overline{c}_2 \psi^*(x),\quad \forall x \in I^*,
\end{equation}
where $ 0 < \overline{c}_1=\min_{y^*\in [0,d]}(c_1(y^*))$ and $\overline c _1 \le\overline{c}_2=\max_{y^*\in [0,d]}(c_2(y^*))$ are now independent on $y ^*$.\\
\textbf{2nd STEP}$\,$ We investigate a 1D inverse inequality. In particular, from Lemma \ref{lemma 1D Jacobi expansion} and from Step 1, we have:
\begin{equation} \label{first 1D estimate}
\begin{split}
\int_{I^*} \Phi(x,y^*)^{\alpha} \q(x,y^*)^2 \dx \le \overline{c}_2^{\alpha} \int_{I^*} \psi^*(x) ^{\alpha} \q(x,y^*)^2\dx \le \frac{\overline{c}_2^{\alpha}}{\overline c_1^\beta} c (\p+1)^{2(\beta-\alpha)}\int_{I^*}\Phi(x,y^*)^{\beta}\q(x,y^*)^2 \dx,
\end{split}
\end{equation}
where $c$ is independent on $y^*$.\\
\textbf{3rd STEP}$\,$ The statement of the lemma is achieved by means of an integration of \eqref{puntual equivalence without dependence} over $y^*\in [0,d]$.
\end{proof}

We show now a global inverse estimate on triangles.
Again, the following result also holds for weight $-1 < \alpha \le \beta$, see \cite{melenk2003hp}.
\begin{thm} \label{theorem L2 weighted inverse estimate on triangles}
Let $0 \le \alpha \le \beta$. Let $\That$ be the reference triangle of vertices $(0,0)$, $(1,0)$ and $(0,1)$. Let $\bThat$ be the cubic bubble function associated with $\That$;
in particular, $\bThat \in \mathbb P_3(\That)$ is such that $\bThat|_{\partial \That} = 0$. Then:
\begin{equation} \label{an inverse estimate on equilateral triangle}
\int_{\That} \bThat^\alpha \q^2 \le \c (\p+1)^{2(\beta-\alpha)} \int _{\That} \bThat^\beta \q^2,\quad \forall \q \in \mathbb P_p(\That),
\end{equation}
where $\c$ is a positive constant independent on $p$.
\end{thm}
\begin{proof}
The proof is similar to that in \cite[Theorem D2]{melenk2003hp} and for this reason we only sketch it.
The idea consists in partitioniong $\That$ into six (overlapping) trapezoidals and applying some $\p$ inverse estimates analogous to those presented in Lemma \ref{lemma trapezoid}.
In particular, the six overlapping trapezoids $D_1,\dots, D_6$ are built, for instance, as in Figure \ref{figure trapezoids}.
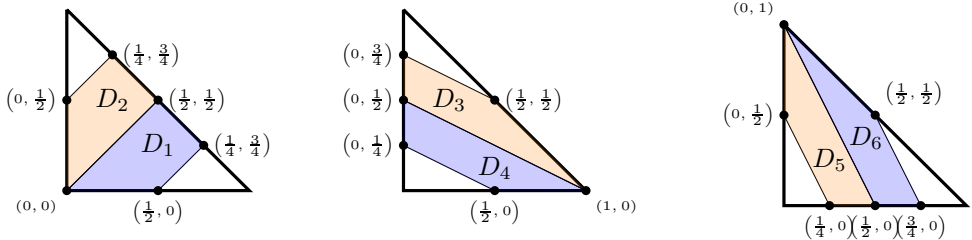
\begin{figure}[H]
\centering
\begin{minipage}{0.30\textwidth}
\begin{center}
\begin{tikzpicture}[scale=1.2]
\draw (-1,-1) node[black, below left] {\tiny{$(0,0)$}}; \draw (0,0) node[black, right] {\tiny{$\left(\frac{1}{2},\frac{1}{2}\right)$}}; \draw (-0.5,0.5) node[black, right] {\tiny{$\left(\frac{1}{4},\frac{3}{4}\right)$}};
\draw (-1,0) node[black, left] {\tiny{$\left(0,\frac{1}{2}\right)$}}; \draw (0,-1) node[black, below] {\tiny{$\left(\frac{1}{2},0\right)$}}; \draw (0.5,-0.5) node[black, right] {\tiny{$\left(\frac{1}{4},\frac{3}{4}\right)$}};
\draw[black, very thick, -] (-1,1) -- (-1,-1) -- (1,-1) -- (-1,1);
\draw[black, -] (-0.5, 0.5) -- (-1,0) -- (-1,-1)  -- (0,0)-- (-0.5,0.5);
\fill[orange,opacity=0.2] (-0.5, 0.5) -- (-1,0) -- (-1,-1) -- (0,0) -- (-0.5,0.5);
\draw[black, -] (0, 0) -- (-1,-1) -- (0,-1) -- (0.5, -0.5) -- (0,0);
\fill[blue,opacity=0.2] (0, 0) -- (-1,-1) -- (0,-1) -- (0.5, -0.5) -- (0,0);
\fill[black] (-1,-1) circle(0.05cm); \fill[black] (0,-1) circle(0.05cm); \fill[black] (0.5,-0.5) circle(0.05cm); \fill[black] (0,0) circle(0.05cm); \fill[black] (-0.5,0.5) circle(0.05cm); \fill[black] (-1,0) circle(0.05cm);
\draw(-0.5,0) node[black] {{$D_2$}}; \draw(0,-0.5) node[black] {{$D_1$}};
\end{tikzpicture}
\end{center}
\end{minipage}
\begin{minipage}{0.30\textwidth}
\begin{center}
\begin{tikzpicture}[scale=1.2]
\draw (-1,0.5) node[black, left] {\tiny{$\left(0,\frac{3}{4} \right)$}}; \draw (0,0) node[black, right] {\tiny{$\left(\frac{1}{2},\frac{1}{2}\right)$}}; \draw (-1,0) node[black, left] {\tiny{$\left(0,\frac{1}{2}\right)$}};
\draw (-1,-0.5) node[black, left] {\tiny{$\left(0,\frac{1}{4}\right)$}}; \draw (0,-1) node[black, below] {\tiny{$\left(\frac{1}{2},0\right)$}}; \draw (1,-1) node[black, below right] {\tiny{$\left(1,0\right)$}};
\draw[black, very thick, -] (-1,1) -- (-1,-1) -- (1,-1) -- (-1,1);
\draw[black, -] (0, 0) -- (-1,0.5) -- (-1,0)  -- (1,-1)-- (0,0);
\fill[orange,opacity=0.2] (0, 0) -- (-1,0.5) -- (-1,0)  -- (1,-1)-- (0,0);
\draw[black, -] (0, -1) -- (-1,-0.5) -- (-1,0)  -- (1,-1)-- (0,-1);
\fill[blue,opacity=0.2] (0, -1) -- (-1,-0.5) -- (-1,0)  -- (1,-1)-- (0,-1);
\fill[black] (-1,0.5) circle(0.05cm); \fill[black] (0,-1) circle(0.05cm); \fill[black] (-1,-0.5) circle(0.05cm); \fill[black] (0,0) circle(0.05cm); \fill[black] (1,-1) circle(0.05cm); \fill[black] (-1,0) circle(0.05cm);
\draw(-0.5,0) node[black] {{$D_3$}}; \draw(0,-0.75) node[black] {{$D_4$}};
\end{tikzpicture}\end{center}
\end{minipage}
\begin{minipage}{0.30\textwidth}
\begin{center}
\begin{tikzpicture}[scale=1.2]
\draw (-1,1) node[black, above left] {\tiny{$(0,1)$}}; \draw (0,0) node[black, above right] {\tiny{$\left(\frac{1}{2},\frac{1}{2}\right)$}}; \draw (0.5,-1) node[black, below] {\tiny{$\left(\frac{3}{4}, 0 \right)$}};
\draw (-1,0) node[black, left] {\tiny{$\left(0,\frac{1}{2}\right)$}}; \draw (0,-1) node[black, below] {\tiny{$\left(\frac{1}{2},0\right)$}}; \draw (-0.5,-1) node[black, below] {\tiny{$\left(\frac{1}{4},0\right)$}};
\draw[black, very thick, -] (-1,1) -- (-1,-1) -- (1,-1) -- (-1,1);
\draw[black, -] (0, -1) -- (-1,1) -- (-1,0)  -- (-0.5,-1)-- (0,-1);
\fill[orange, opacity=0.2] (0, -1) -- (-1,1) -- (-1,0)  -- (-0.5,-1)-- (0,-1);
\draw[black, -] (0, -1) -- (-1,1) -- (0,0) -- (0.5, -1) -- (0,-1);
\fill[blue,opacity=0.2] (0, -1) -- (-1,1) -- (0,0) -- (0.5, -1) -- (0,-1);
\fill[black] (-0.5,-1) circle(0.05cm); \fill[black] (0,-1) circle(0.05cm); \fill[black] (0.5,-1) circle(0.05cm); \fill[black] (0,0) circle(0.05cm); \fill[black] (-1,1) circle(0.05cm); \fill[black] (-1,0) circle(0.05cm);
\draw(-0.5,-0.5) node[black] {{$D_5$}}; \draw(-0.1,-0.25) node[black] {{$D_6$}};
\end{tikzpicture}
\end{center}
\end{minipage}
\caption{Overlapping trapezodails covering the reference triangle $\That$} \label{figure trapezoids}
\end{figure}

We observe that we can apply Lemma \ref{lemma trapezoid} since its hypotesis are satisfied.
In fact, on all the $D_i$, $i=1,\dots, 6$, $\bThat |_{D_i}$ is a continuous function whose restriction on every segment parallel to the basis of the trapezoidal is a cubic concave polynomial annihilating only at the endpoints of the segment.

Estimate \eqref{an inverse estimate on equilateral triangle} follows then from an overlapping argument applying Lemma \ref{lemma trapezoid} on all the elements, noting that $\That = \cup_{i=1}^6 D_i$.
\end{proof}

We discuss also the following result, firstly presented in \cite[Lemma B.3]{bank2013saturation}.
\begin{lem} \label{lemma Bank Parsania Sauter}
Let $\T$ be a triangle, $\bT$ the associated cubic bubble function. Then:
\[
\vert \q \bT \vert_{1,\T} \le \c \frac{\p+1}{h_\T} \Vert \q \bT^{\frac{1}{2}}\Vert_{0,\T}, \quad \forall q \in \mathbb P_\p(\T),
\]
where $c$ is a positive constant independent on $h_\T$ and $\p$, $h_\T=\diam (\T)$.
\end{lem}
\begin{proof}
The proof is split into three parts; the first two of them are technical results dealing with Legendre-type approximations, while the third one deals with the proof of the lemma.\\
\textbf{1st STEP}
The following estimate holds: for all $\q \in \mathbb P_p(\Ihat)$, $\Ihat:=[-1,1]$
\begin{equation} \label{estimate using derivated Legendre sum}
\Vert (1-x^2) \q'(x)\Vert _{0,\Ihat} \le c (\p+1) \Vert (1-x^2)^{\frac{1}{2}} \q(x)\Vert _{0,\Ihat},
\end{equation}
where $c$ is a constant independent on $\p$.

We recall that the following holds, see \cite[formula (3.39)]{SchwabpandhpFEM}
\[
\int_{\Ihat} (1-x^2)^k L_n^{(k)}(x)^2\dx = 
\begin{cases}
\frac{2}{2n+1} \frac{(n+k)!}{(n-k)!} & \text{if } n \ge k,\\
0                                                           & \text{otherwise},\\
\end{cases}
\]
where $L_n^{(k)}(x)$ is the $k$-th derivative of the $n$-th Legendre polynomial. Then:
\begin{equation} \label{application Schwab formula}
\int_ {\Ihat} (1-x^2)^2 L_n''(x)^2\dx = \frac{2}{2n+1} \frac{(n+1)!}{(n-1)!} (n-1)(n+2) \le \left (n+\frac{1}{2} \right)^2 \int_{\Ihat} (1-x^2) L_n'(x)^2\dx.
\end{equation}
Therefore, expanding $\q$ into a derivated Legendre sum
\[
\q(x) = \sum_{n=1}^{\p+1} c_n L_n'(x),
\]
we have, owing to orthogonality of the second derivative of Legendre polynomials with respect to the $L^2$ $(1-x^2)^2$-weighted inner product and owing to \eqref{application Schwab formula}:
\[
\begin{split}
\int_{\Ihat} \q'(x) (1-x^2)^2 \dx 	& = \sum_{n=1}^{\p+1} c_n^2 \int_{\Ihat} L_n''(x)^2 (1-x^2)^2 \dx \le \sum_{n=1}^{\p+1} c_n^2 \left(n+\frac{1}{2}\right)^2 \int_{\Ihat} L_n'(x)^2 (1-x^2)\dx \\
						& \le \left( \p + \frac{3}{2} \right)^2 \int_{\Ihat} \q^2(x) (1-x^2)\dx \le \frac{3}{2}\left( \p + 1 \right)^2 \int_{\Ihat} \q^2(x) (1-x^2) \dx.\\
\end{split}
\]
\textbf{2nd STEP}
We show now the following 1D estimate. For $a<b$, let 
$$
b_{[a,b]}(x):= \frac{(x-a)(b-x)}{(b-a)^2}
$$
be the 1D quadratic bubble function, then:
\begin{equation} \label{stima bank 1D}
\left \Vert (b_{[a,b]}\q)' \right\Vert_{0,[a,b]} \le \c \frac{p+1}{b-a} \left\Vert b_{[a,b]}^{\frac{1}{2}}\q \right\Vert_{0,[a,b]},\quad \forall \q \in \mathbb P_\p([a,b]),
\end{equation}
where $c$ is a positive constant independent on $\p$.
It is sufficient to show \eqref{stima bank 1D} on the reference interval $[-1,1]$, since the general result follows from a scaling argument.

Owing to $\Vert b_{[-1,1]}'\Vert_{\infty,[-1,1]}=\frac{1}{2} < 1$, the Leibniz derivation rule and a triangular inequality, we can write:
\begin{equation} \label{Leibniz plus triangular}
\Vert (b_{[-1,1]}\q)' \Vert_{0,[-1,1]} \le \Vert b'_{[-1,1]} \q \Vert _{0,[-1,1]} + \Vert b_{[-1,1]} \q' \Vert_{0,[-1,1]} \le \Vert \q \Vert_{0,[-1,1]} + \Vert b_{[-1,1]} \q' \Vert_{0,[-1,1]}.
\end{equation}
Applying \eqref{estimate weighted L2 1D} (with $\alpha=0$ and $\beta=1$) and \eqref{estimate using derivated Legendre sum} to the first and second term of \eqref{Leibniz plus triangular} respectively, we get \eqref{stima bank 1D}.\\
\textbf{3rd STEP}
We apply now \eqref{stima bank 1D} and we show the claim of the lemma.

Without loss of generality, we work on the reference triangle $\That=\T$ of vertices $(0,0)$, $(1,0)$ and $(0,1)$. The statement follows from a scaling argument.
The cubic bubble function on $\That$, which is given by the product of the barycentric coordinates, can be rewritten as:
\[
b_{\That} = b_{[0,1-x]}(y) (1-x) b_{[0,1]}(x).
\]
We only show the bound on the partial derivative with respect to $y$. The general case is an easy consequence.
\[
\begin{split}
\Vert \partial_y (b_{\That} \q) \Vert ^2_{0,\That} 	&= \int_0^1 \int _0^{1-x} (\partial_y (b_{\That}(x,y) q(x,y) ))^2 \dy \dx\\
								&= \int_0^1 b_{[0,1]} ^2(x)(1-x)^2  \int_0^{1-x} \left( \partial_y (b_{[0,1-x]}(y) \q(x,y))\right)^2\dy  \dx.
\end{split}
\]
We note that:
\[
\begin{split}
 (1-x)^2 \int_0^{1-x} \left( \partial_y b_{[0,1-x]}(y) \q(x,y)\right)^2\dy	&= (1-x)^2 \Vert \partial_y (b_{[0,1-x]}(\cdot) \q(x,\cdot))\Vert^2 _{0,[0,1-x]} \\
								&\underbrace{\le}_{\eqref{stima bank 1D}} c(\p+1)^2 \Vert b_{[0,1-x]}^{\frac{1}{2}}(\cdot) \q(x,\cdot)\Vert^2 _{0,[0,1-x]}.
\end{split}
\]
Since $b_{[0,1]} \le 1-x$, we get $b^2_{[0,1]}(x) b_{[0,1-x]}(y) \le b_{\That} (x,y)$ and consequently:
\[
\Vert \partial_2 (b_{\That} \q) \Vert ^2_{0,\That} \le \c (\p+1)^2 \int_0^1 \int_0^{1-x} b^2_{[0,1]} (x) b_{[0,1-x]}(y) \q^2(x,y)\dy\dx
\le \c(\p+1)^2 \Vert b_{\That}^{\frac{1}{2}} \q \Vert^2_{0,\That}.
\]
\end{proof}
We are now ready for the inverse estimate involving the $H^{-1}$ norm of polynomials.
\begin{thm} \label{corollary inverse polynomial estimate with negative norm}
Let $\E\subseteq \mathbb R^2$ be a polygon. Assume that there exists $\taun (\E)$ subtriangulation of $\E$ such that $h_\E\approx h_\T$, where $h_\omega=\diam (\omega)$, $\omega\subseteq\mathbb R^2$.
Let $\q\in \mathbb P_\p(\E)$, $\p \in \mathbb N$. Then:
\[
\Vert q \Vert _{0,\E} \le \c \frac{(\p+1)^2}{h_\E} \Vert \q\Vert _{-1,\E},
\]
where $\Vert q \Vert _{-1,\E}:=\Vert q \Vert_{(H_0^1(\E))^*}$.
\end{thm}
\begin{proof}
Let $\bE$ be the ``patch-bubble'' function, defined on each $\T\in \taun(\E)$ as the local cubic bubble function $\bT$ introduced in Lemma \ref{lemma Bank Parsania Sauter}. Then:
\begin{equation} \label{estimates first step}
\Vert \q \Vert_{-1,\E} = \sup_{\Phi \in H_0^1(\E),\,\Phi\ne 0} \frac{(q,\Phi)_{0,\E}}{\vert \Phi \vert_{1,\E}}\ge \frac{(\q,\q\bE)_{0,\E}}{\vert \q \bE\vert_{1,\E}}
= \frac{\Vert \q\sqrt \bE \Vert^2_ {0,\E}}{ \left(\sum_{\T\in \taun(\E)} \vert \q \bT \vert ^2_{1,\T}\right)^{\frac{1}{2}}}.
\end{equation}
Using now Lemma \ref{lemma Bank Parsania Sauter}, \eqref{estimates first step} and the geometric assumption (\textbf{D2}), we obtain:
\begin{equation}
\Vert \q \Vert_{-1,\E} \ge \c \frac{\min_{\T \in \taun (\E)} h_\T}{\p+1} \Vert \q \sqrt{\bE} \Vert_{0,\E} \ge \c \widetilde\gamma \frac{h_\E}{\p+1} \left( \sum_{\T\in \taun(\E)} \Vert \q\sqrt{\bT} \Vert ^2_{0,\T}\right)^{\frac{1}{2}}.
\end{equation}
Finally, we apply Theorem \ref{theorem L2 weighted inverse estimate on triangles} with $\alpha=0$ and $\beta=1$ and get:
\[
\Vert \q \Vert_{-1,\E} \ge \c \frac{h_\E}{(\p+1)^2} \left( \sum_{\T\in \taun(\E)} \Vert \q\Vert ^2_{0,\T}\right)^{\frac{1}{2}} = \c \frac{h_\E}{(\p+1)^2} \Vert \q \Vert_{0,\E}.
\]
\end{proof}
\section{} \label{section appendix B}
In this second appendix, we discuss the following standard $hp$ polynomial inverse estimate on triangles:
\begin{thm} \label{theorem hp inverse estimate H1 with L2 on triangles}
Let $\T\subseteq \mathbb R^2$ be a triangle and let $h_\T$ denote the diameter of $\T$. Then:
\begin{equation} \label{hp inverse estimate H1 with L2 on triangles}
\vert \q \vert _ {1,\T} \le \cinv \frac{\p^2}{h_\T} \Vert \q \Vert _{0,\T},\quad \forall \q \in \mathbb P _\p(\T),\quad \p\ge 1,
\end{equation}
where $\cinv$ is a positive constant independent on $h_\T$, $\p$ and $\q$.
\end{thm}
We note that inequality \eqref{hp inverse estimate H1 with L2 on triangles} is a very well-known and widely used result. It is stated for instance in \cite[Theorem 4.76]{SchwabpandhpFEM}.
Nonetheless, we were not able to find an explicit proof in literature.
\begin{proof}[Proof of Theorem \ref{theorem hp inverse estimate H1 with L2 on triangles}]
We show the result on the reference triangle $\That$ of vertices $(0,0)$, $(1,0)$, $(0,1)$.
The statement will follow from a scaling argument.

We consider a decomposition of $\That$ into the three overlapping parallelograms $P_1$, $P_2$ and $P_3$ depicted in Figure \ref{figure parallelograms}.
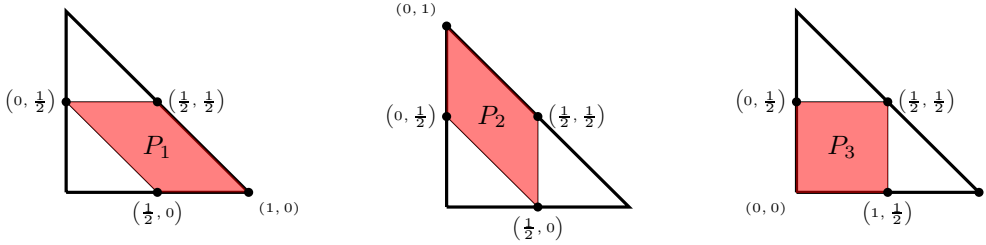
\begin{figure}[H]
\centering
\begin{minipage}{0.30\textwidth}
\begin{center}
\begin{tikzpicture}[scale=1.2]
\draw[black, very thick, -] (0,0) -- (0,2) -- (2,0) -- (0,0);
\draw[black, -] (2,0) node[black, below right] {\tiny{$\left(1,0 \right)$}} -- (1,1) node[black, right] {\tiny{$\left(\frac{1}{2},\frac{1}{2} \right)$}} -- (0,1) node[black, left] {\tiny{$\left(0,\frac{1}{2} \right)$}} -- (1,0) node[black, below] {\tiny{$\left(\frac{1}{2},0 \right)$}}
 -- (2,0);
\fill[red,opacity=0.5] (2,0) -- (1,1) -- (0,1) -- (1,0) -- (2,0);
\fill[black] (2,0) circle(0.05cm); \fill[black] (1,1) circle(0.05cm); \fill[black] (0,1) circle(0.05cm); \fill[black] (1,0) circle(0.05cm);
\draw(1,0.5) node[black] {{$P_1$}};
\end{tikzpicture}
\end{center}
\end{minipage}
\begin{minipage}{0.30\textwidth}
\begin{center}
\begin{tikzpicture}[scale=1.2]
\draw[black, very thick, -] (0,0) -- (0,2) -- (2,0) -- (0,0);
\draw[black, -] (1,0) node[black, below] {\tiny{$\left(\frac{1}{2},0 \right)$}} -- (1,1) node[black, right] {\tiny{$\left(\frac{1}{2},\frac{1}{2} \right)$}} -- (0,2) node[black, above left] {\tiny{$\left(0,1\right)$}} -- (0,1) node[black, left] {\tiny{$\left(0,\frac{1}{2}\right)$}}
 -- (1,0);
\fill[red,opacity=0.5] (1,0) -- (1,1) -- (0,2) -- (0,1) -- (1,0);
\fill[black] (0,2) circle(0.05cm); \fill[black] (1,1) circle(0.05cm); \fill[black] (0,1) circle(0.05cm); \fill[black] (1,0) circle(0.05cm);
\draw(0.5,1) node[black] {{$P_2$}};
\end{tikzpicture}
\end{center}
\end{minipage}
\begin{minipage}{0.30\textwidth}
\begin{center}
\begin{tikzpicture}[scale=1.2]
\draw[black, very thick, -] (0,0) -- (0,2) -- (2,0) -- (0,0);
\draw[black, -] (0,0) node[black, below left] {\tiny{$\left(0,0 \right)$}} -- (1,0) node[black, below] {\tiny{$\left(1,\frac{1}{2} \right)$}} -- (1,1) node[black, right] {\tiny{$\left(\frac{1}{2},\frac{1}{2} \right)$}} -- (0,1) node[black, left] {\tiny{$\left(0,\frac{1}{2} \right)$}}
 -- (0,0);
\fill[red,opacity=0.5] (0,0) -- (1,0) -- (1,1) -- (0,1) -- (0,0);
\fill[black] (2,0) circle(0.05cm); \fill[black] (1,1) circle(0.05cm); \fill[black] (0,1) circle(0.05cm); \fill[black] (1,0) circle(0.05cm);
\draw(0.5,0.5) node[black] {{$P_3$}};
\end{tikzpicture}
\end{center}
\end{minipage}
\caption{Overlapping parallelograms covering the reference triangle $\That$} \label{figure parallelograms}
\end{figure}
\noindent We can write:
\begin{equation} \label{decomposition inverse estimate}
\vert \q \vert _{1,\That} \le \vert \q \vert_{1,P_1} + \vert \q \vert _{1,P_2} + \vert \q \vert _{1,P_3}.
\end{equation}
We only have to prove that:
\begin{equation} \label{inverse estimate on parallelograms}
\vert \q \vert _{1,P_i} \le \c_1 \p^2 \Vert \q \Vert _{0,P_i},\quad i=1,2,3.
\end{equation}
In particular, it suffices to prove the same inequality on the reference square $\Qhat = [-1,1]^2$ and then using an affine transformation in order to deduce the assertion of the theorem from \eqref{decomposition inverse estimate}.
Thus, we must prove:
\begin{equation} \label{inverse inequality on square}
\vert \q \vert_{1,\Qhat} \le c \p^2 \Vert \q \Vert _{0,\Qhat},\quad \forall \q \in \mathbb P_{\p}(\Qhat).
\end{equation}
For the purpose, we have to show:
\[
\Vert \partial _i \q \Vert_{0,\Qhat}\le c \p^2 \Vert \q \Vert_{0,\Qhat},\quad i=1,2.
\]
Owing to \cite[Theorem 3.96]{SchwabpandhpFEM}, we can write:
\begin{equation} \label{inverse inequality 1D}
\Vert \partial _i \q \Vert_{0,\widetilde I} \le c \p^2 \Vert \q \Vert_{0,\widetilde I},
\end{equation}
where
\[
\widetilde I = \begin{cases}
[-1,1] \times \{\widetilde y\},\; \widetilde y \in [-1,1]& \text{if } i=1,\\
\{\widetilde x\} \times [-1,1],\; \widetilde x \in [-1,1]& \text{if } i=2.\\
\end{cases}
\]
 Here, the constant $c$ does not depend on $\widetilde y$.
Integrating \eqref{inverse inequality 1D} in $y$ (if $i=1$) or in $x$ (if $i=2$) from $-1$ to $1$, we get \eqref{inverse inequality on square}.

\end{proof}

\end{appendices}

\section*{Acknowledgement}
The first author has received funding from the European Research Council (ERC) under the European Unions Horizon 2020 research and innovation programme (grant agreement no. 681162). 
{\footnotesize
\bibliography{bibliogr}

\begin{thebibliography}{10}

\bibitem{AbramowitzStegun_handbook}
M.~Abramowitz and I.~A. Stegun.
\newblock {\em Handbook of {M}athematical {F}unctions with {F}ormulas,
  {G}raphs, and {M}athematical {T}ables}.
\newblock 1964.

\bibitem{adamsfournier}
R.~A. Adams and J.~J.~F. Fournier.
\newblock {\em Sobolev {S}paces}, volume 140.
\newblock Academic {P}ress, 2003.

\bibitem{equivalentprojectorsforVEM}
B.~Ahmad, A.~Alsaedi, F.~Brezzi, L.D. Marini, and A.~Russo.
\newblock {E}quivalent {P}rojectors for {V}irtual {E}lement {M}ethod.
\newblock {\em Computers \& Mathematics with Applications}, 66(3):376--391,
  2013.

\bibitem{Andersen-geo}
O.~Andersen, H.M. Nilsen, and X.~Raynaud.
\newblock On the use of the {V}irtual {E}lement {M}ethod for geomechanics on
  reservoir grids.
\newblock \url{http://arxiv.org/abs/arXiv:1606.09508}.

\bibitem{streamvirtualelementformulationstokesproblempolygonalmeshes}
P.~F. Antonietti, L.~Beir{\~a}o~da Veiga, D.~Mora, and M.~Verani.
\newblock A {S}tream {V}irtual {E}lement formulation of the {S}tokes problem on
  polygonal meshes.
\newblock {\em SIAM Journal of Numerical Analysis}, 52(1):386--404, 2014.

\bibitem{absv_VEM_cahnhilliard}
P.~F. Antonietti, L.~Beir{\~a}o~da Veiga, S.~Scacchi, and M.~Verani.
\newblock A $\mathcal{C} ^1$ {V}irtual {E}lement {M}ethod for the
  {C}ahn--{H}illiard equation with polygonal meshes.
\newblock {\em SIAM Journal on Numerical Analysis}, 54(1):34--56, 2016.

\bibitem{babuvska1988regularity}
I.~Babu{\v{s}}ka and B.~Q. Guo.
\newblock Regularity of the solution of elliptic problems with piecewise
  analytic data. {P}art {I}. {B}oundary value problems for linear elliptic
  equation of second order.
\newblock {\em SIAM {J}ournal on {M}athematical {A}nalysis}, 19(1):172--203,
  1988.

\bibitem{babuvska1989regularity}
I.~Babu{\v{s}}ka and B.~Q. Guo.
\newblock Regularity of the solution of elliptic problems with piecewise
  analytic data. {P}art {II}: {T}he trace spaces and application to the
  boundary value problems with nonhomogeneous boundary conditions.
\newblock {\em SIAM {J}ournal on {M}athematical {A}nalysis}, 20(4):763--781,
  1989.

\bibitem{babuskasurihpversionFEMwithquasiuniformmesh}
I.~Babu\v{s}ka and M.~Suri.
\newblock The $hp$ version of the {F}inite {E}lement {M}ethod with
  {Q}uasiuniform {M}eshes.
\newblock {\em Mod{\'e}lisation math{\'e}matique et analyse num{\'e}rique},
  21(2):199--238, 1987.

\bibitem{BabuSuri_optimalconvergence}
I.~Babu\v{s}ka and M.~Suri.
\newblock The optimal convergence rate of the $p$-version of the {F}inite
  {E}lement {M}ethod.
\newblock {\em SIAM Journal on {N}umerical {A}nalysis}, 24(4):750--776, 1987.

\bibitem{bank2013saturation}
R.~E. Bank, A.~Parsania, and S.~Sauter.
\newblock Saturation estimates for $hp$-finite element methods.
\newblock {\em Computing and Visualization in Science}, 16(5):195--217, 2013.

\bibitem{BLM_VEMsmalldeformation}
L.~Beir\~ao~da Veiga, C.~Lovadina, and D.~Mora.
\newblock A {V}irtual {E}lement {M}ethod for elastic and inelastic problems on
  polytope meshes.
\newblock {\em Computer Methods in Applied Mechanics and Engineering}, 295:327
  -- 346, 2015.

\bibitem{VEMvolley}
L.~Beir{\~a}o~da Veiga, F.~Brezzi, A.~Cangiani, G.~Manzini, L.D. Marini, and
  A.~Russo.
\newblock Basic principles of {V}irtual {E}lement {M}ethods.
\newblock {\em Mathematical Models and Methods in Applied Sciences},
  23(01):199--214, 2013.

\bibitem{serendipityVEM}
L.~Beir{\~a}o~da Veiga, F.~Brezzi, L.~D. Marini, and A.~Russo.
\newblock Serendipity nodal {VEM} spaces.
\newblock \url{http://arxiv.org/abs/1510.08477}, 2015.

\bibitem{BBMR_generalsecondorder}
L.~Beirao~da Veiga, F.~Brezzi, L.~D. Marini, and A.~Russo.
\newblock Mixed {V}irtual {E}lement {M}ethods for general second order elliptic
  problems on polygonal meshes.
\newblock {\em Mathematical Modelling and Numerical Analysis}, 50(3):727--747,
  2016.

\bibitem{bbmr_VEM_generalsecondorderelliptic}
L.~Beir{\~a}o~da Veiga, F.~Brezzi, L.~D. Marini, and A.~Russo.
\newblock Virtual {E}lement {M}ethod for general second-order elliptic problems
  on polygonal meshes.
\newblock {\em Mathematical {M}odels and {M}ethods in {A}pplied {S}ciences},
  26(4):729--750, 2016.

\bibitem{VEMelasticity}
L.~Beir{\~a}o~da Veiga, F.~Brezzi, and L.D. Marini.
\newblock {V}irtual {E}lements for linear elasticity problems.
\newblock {\em SIAM Journal of Numerical Analysis}, 51:794--812, 2013.

\bibitem{hitchhikersguideVEM}
L.~Beir{\~a}o~da Veiga, F.~Brezzi, L.D. Marini, and A.~Russo.
\newblock The {H}itchhiker's {G}uide to the {V}irtual {E}lement {M}ethod.
\newblock {\em Mathematical Models and Methods in Applied Sciences},
  24(8):1541--1573, 2014.

\bibitem{hpVEMbasic}
L.~Beir{\~a}o~da Veiga, A.~Chernov, L.~Mascotto, and A.~Russo.
\newblock {Basic principles of $hp$ {V}irtual {E}lements on quasiuniform
  meshes}.
\newblock {\em Mathematical {M}odels and {M}ethods in {A}pplied {S}ciences},
  26(4):1567--1598, 2016.

\bibitem{BLM_MFD}
L.~Beir{\~a}o~da Veiga, K.~Lipnikov, and G.~Manzini.
\newblock {\em The Mimetic {F}inite {D}ifference {M}ethod for elliptic
  problems}, volume~11.
\newblock Springer, 2014.

\bibitem{beiraolovadinarusso_stabilityVEM}
L.~Beir{\~a}o~da Veiga, C.~Lovadina, and A.~Russo.
\newblock Stability analysis for the {V}irtual {E}lement {M}ethod.
\newblock \url{http://arxiv.org/abs/1607.05988}, 2016.

\bibitem{BLV_StokesVEMdivergencefree}
L.~Beir{\~a}o~da Veiga, C.~Lovadina, and G.~Vacca.
\newblock Divergence free {V}irtual {E}lements for the {S}tokes problem on
  polygonal meshes.
\newblock \url{http://arxiv.org/abs/1510.01655}, 2015.

\bibitem{BeiraoManzini_VEMarbitraryregularity}
L.~Beir{\~a}o~da Veiga and G.~Manzini.
\newblock A {V}irtual {E}lement {M}ethod with arbitrary regularity.
\newblock {\em IMA Journal of Numerical Analysis}, 34(2):759--781, 2014.

\bibitem{ManziniBeirao_VEMresidualaposteriori}
L.~Beir{\~a}o~da Veiga and G.~Manzini.
\newblock Residual a posteriori error estimation for the {V}irtual {E}lement
  {M}ethod for elliptic problems.
\newblock {\em Mathematical {M}odelling and {N}umerical {A}nalysis},
  49(2):577--599, 2015.

\bibitem{Benedetto-VEM-2}
M.~F. Benedetto, S.~Berrone, A.~Borio, S.~Pieraccini, and S.~Scial{\`o}.
\newblock A hybrid mortar {V}irtual {E}lement {M}ethod for discrete fracture
  network simulations.
\newblock {\em Journal of Computational Physics}, 306:148 -- 166, 2016.

\bibitem{Berrone-VEM}
M.F. Benedetto, S.~Berrone, S.~Pieraccini, and S.~Scial\`o.
\newblock The {V}irtual {E}lement {M}ethod for {D}iscrete {F}racture {N}etwork
  simulations.
\newblock {\em Comput.Meth.Appl.Mech.Engrg.}, 280:135--156, 2014.

\bibitem{Benedetto-VEM-3}
M.F. Benedetto, S.~Berrone, and S.~Scial{\`o}.
\newblock A globally conforming method for solving flow in discrete fracture
  networks using the {V}irtual {E}lement {M}ethod.
\newblock {\em Finite Elem. in Anal. and Design}, 109:23 -- 36, 2016.

\bibitem{bernardi2001error}
C.~Bernardi, N.~Fi{\'e}tier, and R.~G. Owens.
\newblock An error indicator for mortar element solutions to the {S}tokes
  problem.
\newblock {\em IMA {J}ournal of {N}umerical {A}nalysis}, 21(4):857--886, 2001.

\bibitem{bernardimaday1992polynomialinterpolationinsobolev}
C.~Bernardi and Y.~Maday.
\newblock Polynomial interpolation results in {S}obolev spaces.
\newblock {\em Journal of {C}omputational and {A}pplied {M}athematics},
  43(1):53--80, 1992.

\bibitem{bernardi1997maday}
C.~Bernardi and Y.~Maday.
\newblock {S}pectral {M}ethods, in the {H}andbook of {N}umerical {A}nalysis
  {V}, {P}.{G} {C}iarlet \& {J}.-{L}. {L}ions eds, 1997.

\bibitem{BrennerScott}
S.~C. Brenner and L.~R. Scott.
\newblock {\em The mathematical theory of Finite Element Methods}, volume~15.
\newblock Texts in Applied Mathematics, Springer-Verlag, New York, third
  edition, 2008.

\bibitem{BLS_MFD}
F.~Brezzi, K.~Lipnikov, and M.~Shashkov.
\newblock Convergence of the mimetic finite difference method for diffusion
  problems on polyhedral meshes.
\newblock {\em SIAM Journal on Numerical Analysis}, 43(5):1872--1896, 2005.

\bibitem{Brezzi-Marini:2012}
F.~Brezzi and L.D. Marini.
\newblock Virtual {E}lement {M}ethod for plate bending problems.
\newblock {\em Computer Methods in Applied Mechanics and Engineering},
  253:455--462, 2013.

\bibitem{Gatica-1}
E.~Caceres and G.N. Gatica.
\newblock A mixed {V}irtual {E}lement {M}ethod for the pseudostress-velocity
  formulation of the {S}tokes problem.
\newblock {\em IMA Journal of Numerical Analysis}, 2016.
\newblock DOI: 10.1093/imanum/drw002.

\bibitem{cangianigeorgoulishouston_hpDGFEM_polygon}
A.~Cangiani, E.~H Georgoulis, and P.~Houston.
\newblock $hp$-version {D}iscontinuous {G}alerkin methods on polygonal and
  polyhedral meshes.
\newblock {\em Mathematical Models and Methods in Applied Sciences},
  24(10):2009--2041, 2014.

\bibitem{cangianigeorgulispryersutton_VEMaposteriori}
A.~Cangiani, E.~H. Georgoulis, T.~Pryer, and O.~J. Sutton.
\newblock A posteriori error estimates for the {V}irtual {E}lement {M}ethod.
\newblock \url{http://arxiv.org/abs/1603.05855}, 2016.

\bibitem{CGM_nonconformingStokes}
A.~Cangiani, V.~Gyrya, and G.~Manzini.
\newblock The non-conforming {V}irtual {E}lement {M}ethod for the {S}tokes
  equations.
\newblock \url{http://arxiv.org/abs/1608.01210}.

\bibitem{cangianimanzinisutton_VEMconformingandnonconforming}
A.~Cangiani, G.~Manzini, and O.~J. Sutton.
\newblock {Conforming and nonconforming {V}irtual {E}lement {M}ethods for
  elliptic problems}.
\newblock \url{http://arxiv.org/abs/1507.03543}, 2015.

\bibitem{cockburn_HDG}
B.~Cockburn, J.~Gopalakrishnan, and R.~Lazarov.
\newblock Unified hybridization of {D}iscontinuous {G}alerkin, mixed, and
  {C}ontinuous {G}alerkin methods for second order elliptic problems.
\newblock {\em SIAM Journal on Numerical Analysis}, 47(2):1319--1365, 2009.

\bibitem{dipietroErn_hho}
D.~A. Di~Pietro and A.~Ern.
\newblock Hybrid {H}igh-order {M}ethods for variable-diffusion problems on
  general meshes.
\newblock {\em Comptes Rendus Math{\'e}matique}, 353(1):31--34, 2015.

\bibitem{dubinerspectraltriangles}
M.~Dubiner.
\newblock Spectral {M}ethods on triangles and other domains.
\newblock {\em Journal of {S}cientific {C}omputing}, 6(4):345--390, 1991.

\bibitem{Topology-VEM}
A.L. Gain, G.H. Paulino, S.D. Leonardo, and I.F.M. Menezes.
\newblock Topology optimization using polytopes.
\newblock {\em Computer Methods in Applied Mechanics and Engineering},
  293:411--430, 2015.

\bibitem{Paulino-VEM}
A.L. Gain, C.~Talischi, and G.H. Paulino.
\newblock On the {V}irtual {E}lement {M}ethod for {T}hree-{D}imensional
  {E}lasticity {P}roblems on {A}rbitrary {P}olyhedral {M}eshes.
\newblock In press on {CMAME}. DOI: 10.1016/j.cma.2014.05.005.

\bibitem{GhizzettiOssiciniquadratureformulae}
A.~Ghizzetti and A.~Ossicini.
\newblock {\em Quadrature {F}ormulae}.
\newblock Birkh{\"a}user, 1970.

\bibitem{GilletteRandBajaj_generalizedbarycentric}
A.~Gillette, A.~Rand, and C.~Bajaj.
\newblock Error estimates for generalized barycentric interpolation.
\newblock {\em Advances in computational mathematics}, 37(3):417--439, 2012.

\bibitem{guowang_JacobiapproximationSobolev}
B.~Guo and L.~Wang.
\newblock Jacobi approximations in non-uniformly {J}acobi-weighted {S}obolev
  spaces.
\newblock {\em Journal of {A}pproximation {T}heory}, 128(1):1--41, 2004.

\bibitem{hmps_harmonicpolynomialsapproximationandTrefftzhpdgFEM}
R.~Hiptmair, A.~Moiola, I.~Perugia, and C.~Schwab.
\newblock Approximation by harmonic polynomials in star-shaped domains and
  exponential convergence of {T}refftz $hp$-dg{F}{E}{M}.
\newblock {\em Mathematical Modelling and Numerical Analysis}, 48(3):727--752,
  2014.

\bibitem{koornwinderclassicalorthogonalpolynomials}
T.~Koornwinder.
\newblock Two-variable analogues of the classical orthogonal polynomials.
\newblock {\em Theory and applications of special functions}, pages 435--495,
  1975.

\bibitem{lishen_optimalerrorestimateJacobi}
H.~Li and J.~Shen.
\newblock Optimal error estimates in {J}acobi-weighted {S}obolev spaces for
  polynomial approximations on the triangle.
\newblock {\em Mathematics of {C}omputation}, 79(271):1621--1646, 2010.

\bibitem{melenk2003hp}
J.~M. Melenk.
\newblock $hp$--interpolation of non--smooth functions.
\newblock {\em SIAM Journal of Numerical Analysis}, 43:127--155, 2005.

\bibitem{talischi2010polygonal}
I.~F.~M. Menezes, G.~H. Paulino, A.~Pereira, and C.~Talischi.
\newblock Polygonal {F}inite {E}lements for topology optimization: {A} unifying
  paradigm.
\newblock {\em International Journal for Numerical Methods in Engineering},
  82(6):671--698, 2010.

\bibitem{MoraRiveraRodriguez_aposSteklov}
D.~Mora, G.~Rivera, and R.~Rodriguez.
\newblock A posteriori error estimates for a virtual elements method for the
  steklov eigenvalue problem.
\newblock \url{http://arxiv.org/abs/1609.07154}.

\bibitem{VEMchileans}
D.~Mora, G.~Rivera, and R.~Rodr{\'i}guez.
\newblock A {V}irtual {E}lement {M}ethod for the {S}teklov eigenvalue problem.
\newblock {\em Mathematical Models and Methods in Applied Sciences},
  25(08):1421--1445, 2015.

\bibitem{Bordas-VEM}
S.~Natarajan, S.PA Bordas, and E.T. Ooi.
\newblock Virtual and smoothed finite elements: a connection and its
  application to polygonal/polyhedral finite element methods.
\newblock {\em International Journal for Numerical Methods in Engineering},
  104(13):1173--1199, 2015.

\bibitem{Helmholtz-VEM}
{Perugia, I.}, {Pietra, P.}, and {Russo, A.}
\newblock A plane wave {V}irtual {E}lement {M}ethod for the {H}elmholtz
  problem.
\newblock {\em Mathematical Modelling and Numerical Analysis}, 50(3):783--808,
  2016.

\bibitem{Weisser_basic}
S.~Rjasanow and S.~Wei{\ss}er.
\newblock Higher order bem-based fem on polygonal meshes.
\newblock {\em SIAM Journal on Numerical Analysis}, 50(5):2357--2378, 2012.

\bibitem{SchwabpandhpFEM}
C.~Schwab.
\newblock {\em $p$-and $hp$-Finite Element Methods: Theory and Applications in
  Solid and Fluid Mechanics}.
\newblock Clarendon Press Oxford, 1998.

\bibitem{shentangwangspectralmethods}
J.~Shen, T.~Tang, and L.-L. Wang.
\newblock {\em Spectral {M}ethods: algorithms, analysis and applications},
  volume~41.
\newblock Springer Science \& Business Media, 2011.

\bibitem{SukumarTabarraeipolygonalintroduction}
N.~Sukumar and A.~Tabarraei.
\newblock Conforming {P}olygonal {F}inite {E}lements.
\newblock {\em International Journal for Numerical Methods in Engineering},
  61:2045--2066, 2004.

\bibitem{szego_orthogonal}
G.~Szeg\H{o}.
\newblock {\em Orthogonal polynomials}, volume~23.
\newblock American Mathematical Soc., 1939.

\bibitem{tartar}
L.~Tartar.
\newblock {\em An introduction to Sobolev spaces and interpolation spaces},
  volume~3.
\newblock Springer Science \& Business Media, 2007.

\bibitem{triebel}
H.~Triebel.
\newblock {\em Interpolation theory, function spaces, differential operators}.
\newblock North-Holland, 1978.

\bibitem{Wriggers-contact}
P.~Wriggers, W.~T. Rust, and B.~D. Reddy.
\newblock A virtual element method for contact.
\newblock {\em Computational Mechanics}, pages 1--12, 2016.
\newblock DOI: 10.1007/s00466-016-1331-x.

\end{thebibliography}
}
\bibliographystyle{plain}
\end{document}